\documentclass{amsart}[12pt]
\usepackage{amsfonts}
\usepackage{amscd}
\usepackage{amsmath}
\usepackage{amssymb}
\usepackage{amsthm}
\usepackage{latexsym}
\usepackage{verbatim}
\usepackage{color}
\usepackage{bbm}
\usepackage{graphicx}
\usepackage{caption}
\usepackage{subcaption}
\usepackage[utf8]{inputenc}
\usepackage{float}
\usepackage{placeins}

\newtheorem{theorem}{Theorem}[section]
\newtheorem{definition}[theorem] {Definition}
\newtheorem{lemma}[theorem]{Lemma}
\newtheorem{corollary}[theorem]{Corollary}

\newtheorem{proposition}[theorem]{Proposition}
\newtheorem{conjecture}[theorem]{Conjecture}
\theoremstyle{definition}
\newtheorem{example}[theorem]{Example}

\addtocontents{toc}{\setcounter{tocdepth}{-10}} 

\newcommand{\C}{\mathbb{C}}
\DeclareMathOperator{\Lam}{Lam}
\DeclareMathOperator{\Rel}{Rel}
\DeclareMathOperator{\PM}{PM}
\DeclareMathOperator{\metric}{met}
\DeclareMathOperator{\md}{md}
\newcommand{\Z}{\mathbb{Z}}
\newcommand{\R}{\mathbb{R}}
\newcommand{\RP}{\mathbb{RP}}
\newcommand{\E}{\mathbb{E}}
\newcommand{\D}{\mathbb{D}}
\newcommand{\compose}{{\circ}}

\usepackage{tikz}
\usetikzlibrary{arrows}
\tikzstyle{every picture}=[> = to]
\tikzset{cdlabel/.style={execute at begin node=$\scriptstyle,execute at end node=$}}
\tikzset{implication/.style={double equal sign distance, -implies}}
\tikzset{biimplication/.style={double equal sign distance, implies-implies}}

\def\mystrut{{\rule[-2ex]{0ex}{4.5ex}{}}}
\def\cbar{\overline{\C}}
\def\g{\gamma}
\def\G{\Gamma}
\def\De{\Delta}
\def\ds{\displaystyle}
\def\HD{{\rm HD}}
\def\N{\mbox{$\mathbb N$}}
\newcommand{\REFEQN}[1] { \begin{equation}\label{#1} }
\newcommand{\ENDEQN}{\end{equation}}
\newcommand{\REFTHM}[1] { \begin{theorem}\label{#1} }
\newcommand{\ENDTHM}{\end{theorem}}
\newcommand{\REFNTH}[1] { \begin{newthm}\label{#1} }
\newcommand{\ENDNTH}{\end{newthm}}
\newcommand{\REFPROP}[1]{\begin{proposition}\label{#1} }
\newcommand{\ENDPROP}{\end{proposition} }
\newcommand{\REFLEM}[1]{\begin{lemma}\label{#1} }
\newcommand{\ENDLEM}{\end{lemma} }
\newcommand{\REFCOR}[1]{\begin{corollary}\label{#1} }
\newcommand{\ENDCOR}{\end{corollary} }
\newcommand\NE{\mathit{NE}}

\DeclareMathOperator{\met}{met}
\newcommand{\abs}[1]{\lvert #1 \rvert}

\begin{document}

\title{Degree-$d$-invariant laminations}

\author[W. Thurston]{William P. Thurston}
\address[W. Thurston]{Department of Mathematics, Cornell University\\ Ithaca, NY, 14853, USA, and  \newline King Abdulaziz University \\ Jeddah 22254, Saudi Arabia}

\author[Baik]{Hyungryul Baik}
\address[Baik]{Department of Mathematical Sciences, KAIST\\
		291 Daehak-ro Yuseong-gu, Daejeon, 34141, South Korea}
\email{hrbaik@kaist.ac.kr}

\author[Gao]{Gao Yan}
\address[Gao Yan]{Department of Mathematics, Sichuan University, Chengdu, 610065, China}
\email{gyan@scu.edu.cn}

\author[Hubbard]{John H. Hubbard}
\address[Hubbard]{Department of Mathematics, Cornell University \\ Ithaca, NY 14853, USA}
\email{jhh8@cornell.edu}

\author[Lei]{Tan Lei}
\address[Tan Lei]{Universit\'{e} d'Angers, Facult\'{e} des sciences, LAREMA \\ 49045 Angers cedex 01, France}
\email{tanlei@math.univ-angers.fr}

\author[Lindsey]{Kathryn A. Lindsey}
\address[Lindsey]{Department of Mathematics, University of Chicago\\ Chicago, IL 60637, USA}
\email{lindseka@bc.edu}

\author[D. Thurston]{Dylan P. Thurston}
\address[D. Thurston]{Department of Mathematics, Indiana University--Bloomington \\ Bloomington, IN 47405, USA}
\email{dpthurst@indiana.edu}


\date{\today}

\begin{abstract}

Degree-$d$-invariant laminations of the disk model the dynamical action of a degree-$d$ polynomial; such a lamination defines an equivalence relation on $S^1$ that corresponds to dynamical rays of an associated polynomial landing at the same multi-accessible points in the Julia set.  Primitive majors are certain subsets of degree-$d$-invariant laminations consisting of critical leaves and gaps.  The space $\PM(d)$ of primitive degree-$d$ majors is a spine for the set of monic degree-$d$ polynomials with distinct roots and serves as a parameterization of a subset of the boundary of the connectedness locus for degree-$d$ polynomials.  The core entropy of a postcritically finite polynomial is the topological entropy of the action of the polynomial on the associated Hubbard tree.  Core entropy may be computed directly, bypassing the Hubbard tree, using a combinatorial analogue of the Hubbard tree within the context of degree-$d$-invariant laminations.

 \end{abstract}
\maketitle


\section*{Preface}

During the last year of his life, William P. Thurston developed a
theory of {\it degree-$d$-invariant laminations}, a tool that he hoped
would lead to what he called ``a qualitative picture of [the dynamics
of] degree~$d$ polynomials." Thurston discussed his research on this
topic in his seminar at Cornell University and was in the process of
writing an article on this topic, but he passed away before completing
the manuscript.  Part I of this document consists of Thurston's
unfinished manuscript.  As it stands, the manuscript is beautifully
written and contains a lot of his new ideas. However, he discussed
ideas that are not in the unfinished manuscript and some details are
missing. Part II consists of supplementary material written by the other
authors based on what they learned from him throughout his seminar and
email exchanges with him. Tan Lei also passed away during the preparation of Part II.  William Thurston's vision was far beyond what we could
write here, but, hopefully, this paper will serve as a starting point
for future researchers.

Each semester since moving to Cornell University in 2003, William Thurston taught a seminar course titled  ``Topics in Topology," which was familiarly (and perhaps more accurately) referred to by participants as ``Thurston seminar."   On the first day of each semester, Thurston asked the audience what mathematical topics they would like to hear about, and he tailored the direction of the seminar according to the interests of the participants.  Although the course was nominally a seminar in topology, he discussed other topics as well, including combinatorics, mathematical logic, and complex dynamics. Between 2010 and 2012, a high percentage of the seminar participants were dynamicists, and so (with the exception of one semester)
  Thurston's seminar during this period primarily focused on topics in complex dynamics.  He discussed his topological characterization of rational maps on the Riemann sphere, as well as how to understand complex polynomials via topological entropy and laminations on the circle. During this time, Thurston developed many beautiful ideas, motivated in part by discussions with people in his seminar and in part by email exchanges with others who were at a distance.

His seminar was not an organized lecture series; it was much more than that. He talked about ideas that he was developing at that very moment, as opposed to previously known findings. Thurston invited members of the seminar to be actively involved in the exploration. He often demonstrated computer experiments in class, and he encouraged seminar participants to experiment with his codes.  Seminar participants frequently received drafts of Thurston's manuscript as his thinking evolved. Those fortunate enough to learn from Thurston observed how his understanding of the subject gradually transformed into a beautiful theory.

\vspace{1cm}
\begin{center}
{\noindent \bf \sc \Large Part I}
\end{center}
\vspace{.5cm}

\begin{center}
{\bf \large Degree-$d$-invariant laminations}

\bigskip
William P. Thurston

February 22, 2012

\end{center}

\section{Introduction}

Despite years of strong effort by an impressive group of insightful and hardworking mathematicians and many advances, our overall understanding and global picture of the dynamics of degree $d$ rational maps and even degree $d$ complex polynomials has remained sketchy and unsatisfying.

The purpose of this paper is to develop at least a sketch for a skeletal qualitative picture of degree $d$ polynomials. There are good
theorems characterizing and describing examples individually or in small-dimensional families, but that is not our focus. We hope instead to contribute toward developing and clarifying the global picture of the connectedness locus for degree $d$ polynomials, that is, the higher-dimensional analogues of the Mandelbrot set.

To do this, the main tool will be the theory of degree-$d$-invariant laminations.

We hope that by developing a better picture for degree $d$ polynomials, we will develop insights that will carry on to better understand degree $d$ rational maps, whose global description is even more of a mystery.

\section{Some definitions and basic properties}

A degree $d$ polynomial map $z\mapsto P(z) : \C\to \C$ always ``looks like" $z\mapsto z^d$ near $\infty$. More precisely, it is known that $P$ is
conjugate to $z\mapsto z^d$ in some neighborhood of $\infty$.

We may as well specialize to monic polynomials such that the center of mass of the roots is at the origin (so that the coefficient of $z^{d-1}$ is $0$),
since any polynomial can be conjugated into that form. In that case, we can choose the conjugating map to converge to the identity near $\infty$; this
uniquely determines the map.  As Douady and Hubbard noted, we can use the dynamics to extend the conjugacy near $\infty$ inward toward $0$ step by step. If the Julia set is connected, we obtain in this way a Riemann mapping of the complement of the Julia set to the complement of the closed unit disk in $\hat\C$ that conjugates the dynamics outside the Julia set (which is the attracting basin of $\infty$) to the standard form $z\mapsto z^d$.

It has been known since the time of Fatou and Julia (and easy to show) that the Julia set is the boundary of the attracting basin of $\infty$. We want to investigate the topology of the Julia set, and how this topology varies among polynomial maps of degree $d$. As long as the Julia set is locally connected, the Julia set is the continuous image of the unit circle that is the boundary of the Riemann map around $\infty$; the key question is to understand the identifications on the circle made by these maps, and the way the identifications vary as the polynomial varies.

Define a
\emph{treelike} equivalence relation on the unit circle to be  a closed equivalence relation such that for any two distinct equivalence classes,
 their convex hulls in the unit disk are disjoint.
(A relation $R \subset X\times X$ on a topological space $X$ is \emph{closed} if it is a closed subset of $ X\times X$.)
The condition that the convex hulls of equivalence classes be disjoint comes from the topology of the plane: it translates into the condition that if we take the complement of the open unit disk and make the given identifications on the circle, two simple closed curves that cross the circle quotient using different equivalence classes cannot have intersection number $1$, otherwise the quotient space would not embed in the plane.

\begin{figure}[htpb]
\centering
\vbox{
\includegraphics[width=2.3in]{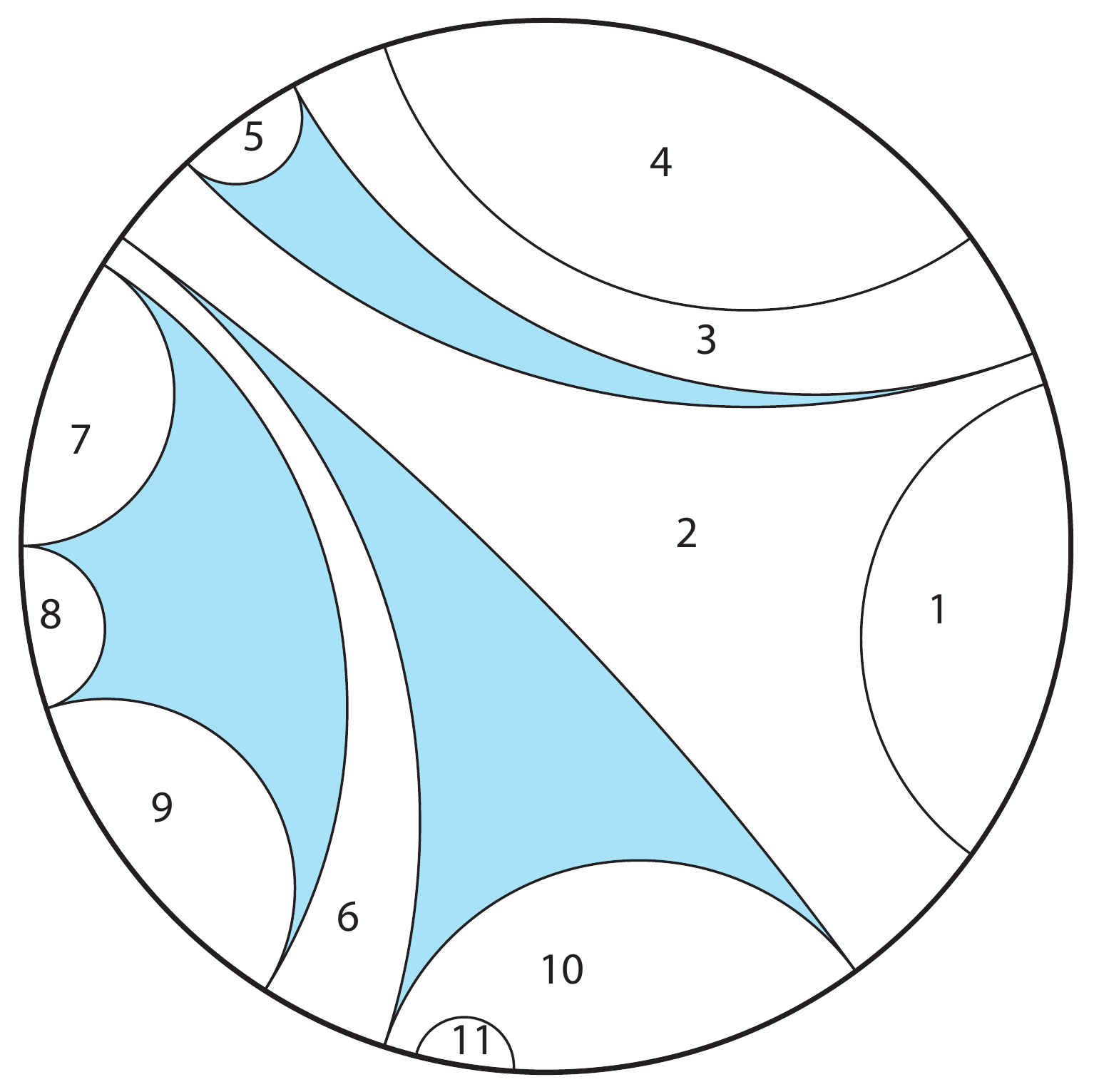}
\includegraphics[width=2.3in]{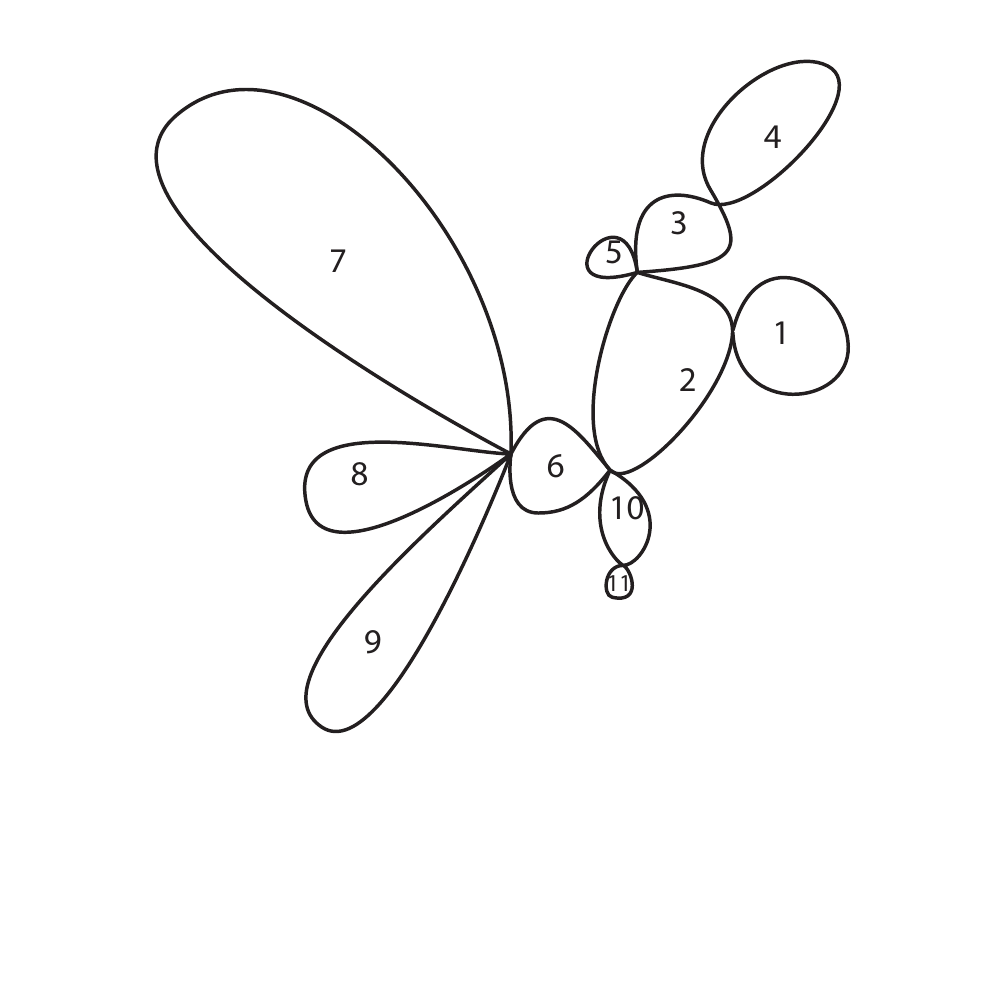}
}
\caption{On the left is a finite lamination associated with a treelike equivalence relation with finitely many non-trivial equivalences: the two ends of any leaf are equivalent, and (as a consequence) the vertices of the polygons are all equivalent. The leaves are drawn as geodesics in the Poincar\'e  disk model of the hyperbolic plane. The leaves and shaded areas can be shrunk to points by a pseudo-isotopy of the plane to obtain (topologically) the figure on the right.}
\end{figure}

To develop a geometric understanding of the possible equivalence relations, it helps to translate the concept into the language of laminations.  Given a treelike equivalence relation $R$, there is an associated lamination $\Lam(R)$ of the open disk, where the leaves of $\Lam(R)$
consist of boundaries of convex hulls of equivalence classes intersected with the open disk.  The regions bounded by leaves are called \emph{gaps}.  Some of the gaps of $\Lam(R)$ are \emph{collapsed gaps},
which touch the circle in a single equivalence class, while other gaps are \emph{intact gaps}. An intact gap necessarily touches the circle in an uncountable set, and its boundary mod $R$ is  homeomorphic to a circle.

Conversely, given a lamination $\lambda$ there is an equivalence relation $\Rel(\lambda)$, usually obtained by taking the transitive closure of the relation that equates endpoints
of leaves. However, this relation may not be topologically closed.  One can take the closure, which may not be an equivalence relation.  It is possible to alternate the operations
of transitive closure and topological closure  by transfinite induction until it stabilizes to a closed equivalence relation $\Rel(\lambda)$, or simply define $\Rel(\lambda)$
as the intersection of all closed equivalence relations that identify endpoints of leaves of $\lambda$.

 Often, but not always, the operations $\Lam$ and $\Rel$ are inverse to each other.  If $R$ has equivalence classes that are Cantor sets, then $\Rel(\Lam(R))$ will collapse these only to a circle, not to a point. If $\lambda$ has chains of leaves with common endpoints, then any leaves in the interior of the convex hull of the chain of vertices will disappear, and
new edges may also appear.

A circular order on a set intuitively means an embedding of the set on a circle, up to orientation-preserving homeomorphism. In combinatorial terms, a circular order can be given by a function $C$ from triples of elements to the set $\{-1,0,1\}$ (interpreted as giving the sign of the area of the triangle formed by three elements in the plane) such that $C(a,b,c) = 0$ if and only if the elements are not distinct, and $C$ is a 2-cocycle, meaning for any four elements $a$, $b$, $c$ and $d$, the sum of the values for the boundary of a 3-simplex labeled with these elements is 0. (This implies, in particular, that $C(a,b,c) = -C(b,a,c)$, etc.)  With this definition, it can be shown that any countable circularly ordered set can be embedded in the circle in such a way that $C(a,b,c)$ is the sign of the area of the triangle formed by the images of $a$, $b$ and $c$.  One way to prove this assertion   is to first  observe that a circular order can be `cut' at a point $a$ to into a linear order, defined by  $b < c \iff C(a,b,c) = 1$.  It is reasonably well-known that a countable linear order on a set is induced by an embedding of that set in the interval, from which it follows that a circular ordering is induced by an embedding in $S^1$.

An \emph{interval} $J$ of a circularly-ordered set is a subset with a linear order (=total order)
satisfying the condition that
\begin{itemize}
\item[i.] {For  $a,b,c \in J$,
$a < b < c \iff C(a,b,c) = 1$, and}
\item[ii.] {For any $x$ and any $a, b \in J$, if
 $C(a,x,b) = 1$ then $ x \in J$.}
 \end{itemize}
Usually the linear order is determined by the subset, but there is an exception if the subset is the entire circularly ordered set.
In that case the definition is like cutting a circle somewhere to form an interval.
   Just as for linear orders, any two elements $a, b$ in a circularly ordered set determine a closed interval $[a,b]$, as well as an open interval $(a,b)$, etc.

The \emph{degree} of a map $f: X \to Y$ between two circularly ordered sets is the minimum size of a partition of $X$ into intervals such that on each interval, the circular order is preserved.  If there is no such finite partition,
the degree is $\infty$.

A treelike equivalence relation $R$ is \emph{degree-$d$-invariant}
if
\begin{itemize}
\item[i.] {If $sRt$ then $s^d R t^d$, and}
\item[ii.] {For any equivalence class $C$,  the total degree of the restrictions of $z \mapsto z^d$ to the equivalence classes on the set $C^{1/d}$ is $d$.}
\end{itemize}

Similarly, a lamination $\lambda$ is \emph{degree-$d$-invariant} if
\begin{itemize}
\item[i.] {If there is a leaf with endpoints $x$ and $y$, then either $x^d = y^d$ or there is a leaf with endpoints $x^d$ and $y^d$}
\item[ii.] {If there is a leaf with endpoints $x$ and $y$, there is a set of $d$ disjoint leaves with one endpoint in $x^{1/d}$ and the other endpoint in $y^{1/d}$.}
\end{itemize}

These conditions on leaves imply related conditions for gaps, from the behavior of the boundary of the gap.  In particular, an equivalence relation $R$ is degree-$d$-invariant if and only if
$\Lam(R)$ is degree-$d$-invariant.
As customary,  we will use the word `quadratic' as a synonym for degree 2, `cubic' for degree 3, `quartic' for degree 4, etc.

It is worth pointing out that the map $f_d: z \mapsto z^d$ is centralized by a dihedral group of  symmetries of order $2(d-1)$, generated by reflection $z \mapsto \overline z$ together with
rotation $z \mapsto \zeta z$ where $\zeta$ is a primitive $(d-1)$th root of unity.  Therefore, the set of all degree-$d$-invariant laminations and the set of all degree-$d$-invariant relations
has the same symmetry group.

If $R$ is a degree-$d$-invariant equivalence relation, a \emph{critical class}
$C$ is an equivalence class that maps with degree greater than 1. Similarly, a critical gap in $\Lam(R)$ is a gap whose intersection with the circle is mapped with degree greater than 1. The \emph{criticality} of a class or a gap
$C$ is its degree minus 1.  A critical class may correspond to either  a critical leaf, which must have criticality 1, or a critical gap.  A critical gap may be a collapsed gap that corresponds to a critical class, or an intact gap.

\begin{proposition}
For any degree-$d$-invariant equivalence relation $R$, the total criticality of
equivalence classes of $R$ together with intact critical gaps of $\Lam(R)$ (i.e the sum of the criticality of the critical classes and the intact critical gaps) equals $d-1$.
\end{proposition}

\begin{proof}
First, let's establish that if $d > 1$ there is at least one critical leaf or critical gap.  We can do this by extending the
map $z \mapsto z^d$ to the disk: first extend linearly on each edge of the lamination, then foliate each gap by vertical lines and extend linearly to each of the leaves of this foliation. If there were neither collapsing leaves nor collapsing gaps, this would be a homeomorphism on each leaf and on each gap, so the map would be a covering map, impossible since the disk is simply-connected.

Now consider any  critical leaf $\overline{x y}$, the circle with $x$ and $y$ identified is the union of two circles mapping with total degrees
$d_1$ and $d_2$ where $d_1 + d_2 = d$.  Proceed by induction: if the total criticality in  these two circles is $d_1-1$ and $d_2-1$, then the total criticality in the whole circle is $d-1$.

Now consider any  critical gap $G$.  Let $X$ be the union of $S^1$ with the boundary of $G$, and $X'$ be its image
under $z \mapsto z^d$, extended linearly on each edge.  Any leaf in the image of the boundary of $G$ has $d$ preimages, from which it follows readily that the total criticality of the map on the complementary regions of $X$ is $d-1$.
\end{proof}

\begin{figure}[htpb] \label{fig:majors}
\centering
\vbox{
\includegraphics[width=2.1in]{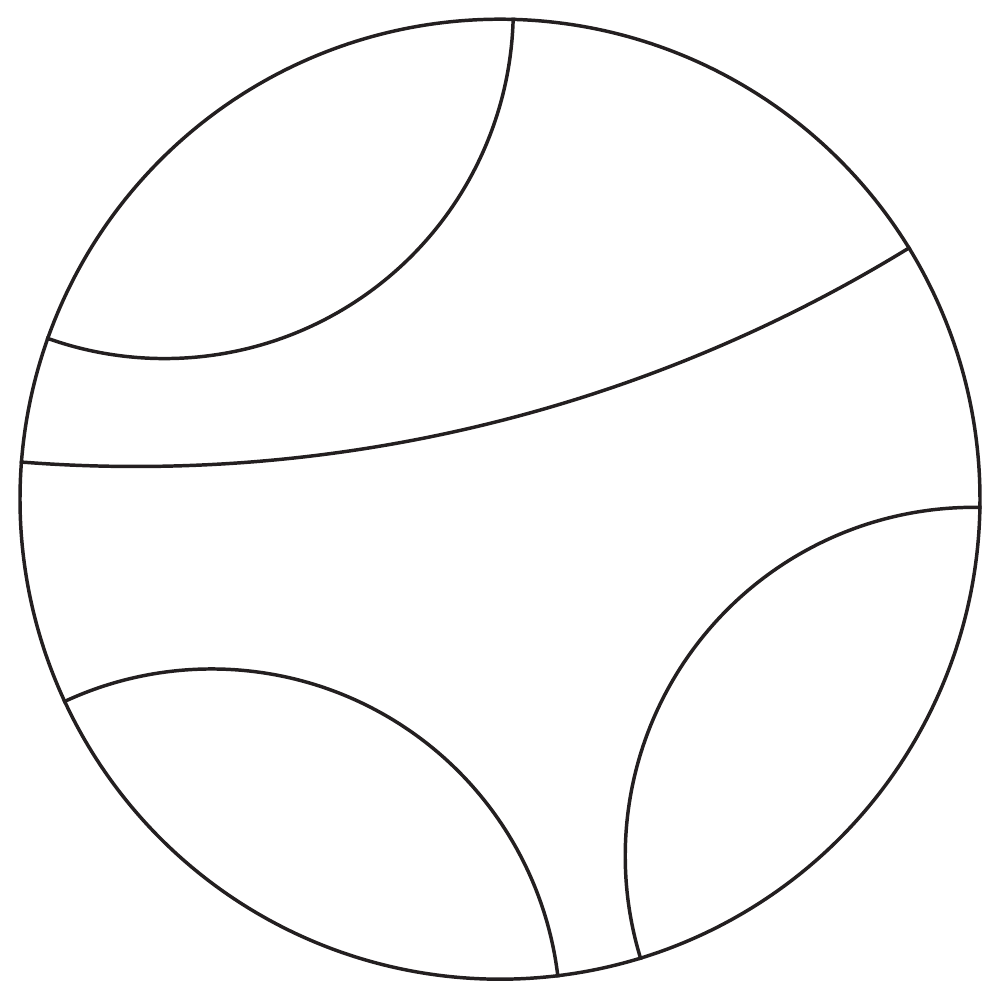}
\includegraphics[width=2.1in]{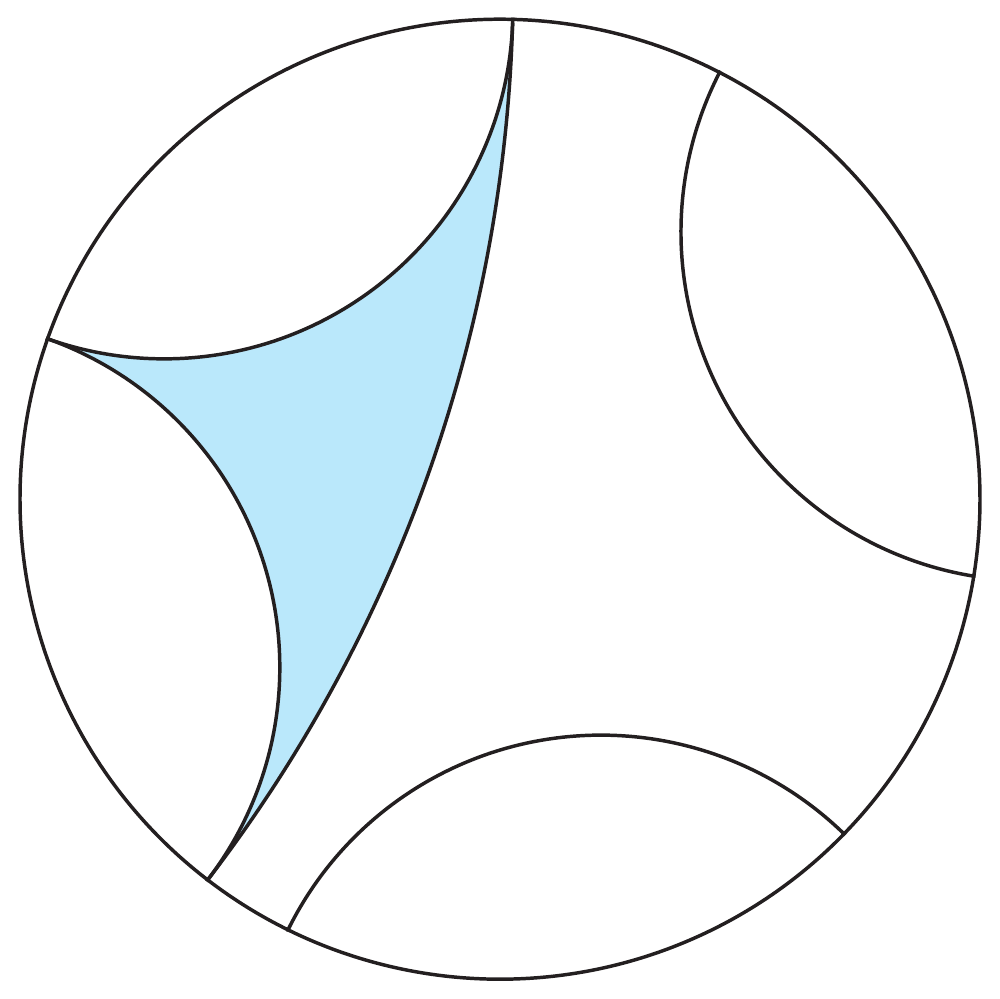}
}
\caption{Here are two examples of primitive majors for a quintic-invariant lamination.  Each intact gap that touches $S^1$ touches with total length $2\pi/5$. As a consequence, the two endpoints of any leaf are collapsed to a single point under $z \mapsto z^5$.  Each isolated leaf has criticality 1; the three vertices of the shaded triangular gap in the example at right map to a single point, so it has criticality 2.
If the leaves of the lamination at left move around so that two of them touch, then a third leaf joining the non-touching endpoints of the touching leaves is implicit, as in the figure at right. The example at right can be perturbed so that any one of the three leaves forming the ideal triangle disappears and the other two become disjoint.}
\end{figure}

\begin{definition}
The \emph{major} of a degree-$d$-invariant lamination is the set of critical leaves and critical gaps.  The \emph{major} of a degree-$d$-invariant equivalence relation $R$ is the set of equivalence classes corresponding to critical leaves and critical gaps of $\Lam(R)$.
\end{definition}

\section{Majors }
\label{section:Majors}
What are the possibilities for the major of a degree-$d$-invariant lamination?

We'll say that the major of a lamination is \emph{primitive} if each critical gap is a polygon whose vertices are all identified by $z \mapsto z^d$.  We will see later that any degree-$d$-invariant lamination is contained in a degree-$d$-invariant lamination whose major is primitive; there are sometimes several possible invariant enlargements, and sometimes uncountably many. The  invariant enlargements might not come from invariant equivalence relations, but they are useful nonetheless for understanding the global structure of invariant laminations, invariant equivalence relations, and Julia sets for degree $d$ polynomials.

Even without a predefined lamination, we can define a \emph{primitive degree-$d$ major} to be a collection of disjoint  leaves and polygons each of whose vertices are identified under $z \mapsto z^d$, with total criticality $d-1$.
In Section \ref{section:LaminationsFromMajors} we will show how to construct a degree-$d$-invariant lamination whose major is the given major.  First though, we will analyze the set $\PM(d)$
  of all primitive degree-$d$ majors.

  An element $m \in \PM(d)$ determines a quotient graph $\gamma(m)$ obtained by identifying each equivalence class to a point. The path-metric of $S^1$ defines a path-metric on $\gamma(m)$. In addition, $\gamma(m)$ has the structure of a planar graph, that is, an embedding in the plane well-defined up to isotopy, obtained by shrinking each leaf and each ideal polygon of the lamination to a point.  These graphs have the property that $H^1(\gamma(m))$ has rank $d$, and every cycle has length a multiple of $2\pi/d$.  Every edge must be accessible from the infinite component of the complement, so the metric and the planar embedding together with the starting point, that is, the image of $1 \in \C$, is enough to define the major.

The metric  $\metric(m)$ on the circle induced by the path-metric on $\gamma(m)$  determines a metric $\md$ on  $\PM(d)$,
  define as the sup difference of metrics,
  \[
  \md(m, m') = \sup_{x, y \in S^1} \left | \metric(m)(x,y) - \metric(m')(x,y)\right | .
  \]
  In particular, there is a well-defined topology on $\PM(d)$.

\medskip
There is a recursive method one can use to construct and analyze primitive degree-$d$ majors, as follows.
Consider any isolated leaf $l$ of a major $m \in PM(d)$. The endpoints of
$l$ split the circle into two intervals, $A$ of length $k/d$ and $B$ of length $(d-k)/d$, for some integer $k$.  If we identify the endpoints of $A$ and make an affine reparametrization to stretch it  by a factor of $d/k$
so it fits around the unit circle and place the identified endpoints at 0,  the leaves of $m$ with endpoints in $A$ become a primitive degree $k$ major.  Similarly, we get a primitive degree $(d-k)$ major from $B$.
If $p$ is any ideal $j$-gon of an element $m \in \PM(d)$, we similarly can derive a sequence of $j$ primitive majors whose total degree is $d$.

\bigskip

   \begin{theorem}
   \label{theorem:pmd_into_polynomialspace}
 The space $\PM(d)$ can be embedded in the space of monic degree $d$ polynomials
  as a spine for the set of polynomials with distinct roots, that is, the complement of the discriminant locus.
The spine (image of the embedding) consists of monic polynomials whose critical values are all on the unit circle and such that the center of mass of the roots is at the origin.
\end{theorem}

A \emph{spine} is a lower-dimensional subset of a manifold such that the manifold deformation-retracts to the spine.
\begin{proof}
We'll start by defining a map from the space of polynomials with distinct zeroes to $\PM(d)$.  Given a polynomial $p(z)$, we look at the gradient vector field for $|p(z)|$,
or better, the gradient of the smooth function $|p(z)|^2$. Any simple critical point for $p$ is a saddle point for this flow.   Near infinity, the flow lines align very closely to the flow lines for the gradient of $|z|^{2d}$,
 so each flow line that doesn't tend toward a critical point goes to infinity with some asymptotic argument (angle).  Define a lamination whose leaves join the pairs of outgoing separatrices for the critical points.  If there are critical points of higher multiplicity or if the unstable manifolds from some critical points coincide with stable manifolds for others (separatrices collide), then define an equivalence relation that decrees that for any flow-line, if the limit of asymptotic angles on the left is not the same as the limit of asymptotic angles on the right, then the two limits are equivalent.  For each such equivalence class, adjoin the ideal polygon that is the boundary of its convex hull.

We claim that the lamination so defined is a primitive degree-$d$
major. To see that, notice that the level sets of $|p(z)|$ are mapped
under $p(z)$ to concentric circles, so the flow lines for the gradient flow map to the perpendicular rays emanating from the origin. Each equivalence class coming from discontinuities in the asymptotic angles therefore maps to a single point, since there are no critical points for the image foliation except at the origin in the range of $p$. The total criticality is $d-1$ since the degree of $p'(z)$ is $d-1$.

There are many different polynomials that give any particular major. In the first place, note that for any polynomial, if we translate in the domain by any constant (which amounts to translating all the roots in some direction, keeping the vectors between them constant), the asymptotic angles of the upward flow lines from the critical points do not change, so then lamination does not change. We may as well keep the roots centered at the origin. Algebraically, this is equivalent to saying that the critical points are centered at the origin (by considering the derivative of the defining polynomial).

 You can think of the spine this way.  In the domain of $p$, cut $\C$ along each upward separatrix of each critical point. This cuts $\C$ into $d$ pieces, each containing one root of $p$ (that you arrive at by flowing downhill; whenever there is a path from $z_1$ to $z_2$ whose downhill flows don't ever meet a critical point, they end up at the same root).    For any  region, if you glue the two sides of each edge together, starting at the lowest critical point and matching points with equal values for $p(z)$, you obtain a copy of $\C$ which is really just a copy of the image plane; the seams have joined together to become rays.

You can construct other polynomials associated with the same major by varying the length of the cuts. (These cuts are classical branch cuts).  Assuming for now that we are in the generic case with no critical gaps, for each leaf choose a positive real number.  Take one copy of $\C$ for each region of the complement of the lamination, and for each leaf $l$
$\overline{xy}$ on its boundary, make a slit on the ray at angle $x^d = y^d$ from $\infty$ to the point $r_l x^d$. Now glue the slit copies of $\C$ together so as to be compatible with the parametrization by $\C$.   (This is equivalent to forming a branched cover of $\C$, branched over the various
critical values, with given combinatorial information or  \emph{Hurwitz data} describing the branching.)  By uniformization theory, the
 Riemann surface obtained by gluing the copies of $\C$ together is analytically equivalent to $\C$, and the induced map is a polynomial.

In the nongeneric case (in which there are critical gaps),  the set of possibilites branches: it  is not parametrized by a product of Euclidean spaces, because of the different possibilities for the combinatorics of saddle connections.  Let's look more closely at  the possible structure of saddle connections for the gradient flow of $|p(z)|$.   A small regular neighborhood of the union
of the upward separatrices for the critical points has boundary
consisting of $k+1$ lines (here $k$ is the multiplicity of the
critical point); these map to the leaves of the boundary of the corresponding gap of the major.
The union of upward separatrices is a tree whose leaves are the vertices of the gap polygon. The height function
$log|p(z)|$ has the property that induces a function with exactly one local minimum on each boundary component of the regular neighborhood projected to the graph.

We can bypass the enumeration of all such structures by observing that there is a direct way to define a retraction to the case when all critical values equal 1.  We think of the graph as a metric graph where the height function $log|p|$ has speed 1.  Multiply
the height function on the compact edges of the graph by $(1-t)$ while adding the  constant times $t$ to the height function on each unbounded component that makes it continuous at the vertices.  When $t=1$, the compact edges all collapse, only the unbounded edges remain, and the polygonal gap contains a single multiple critical point with critical value 1.

\end{proof}

\begin{corollary} \label{Corollary:braid}
$\PM(d)$ is a $K(B_d, 1)$ where $B_d$ is the $d$-strand braid group.  In other words, $\pi_1(\PM(d)) = B_d$ and all higher homotopy groups are trivial.
\end{corollary}
This follows from the well-known fact that the complement of the discriminant locus is a $K(B_d, 1)$.
A loop in the space of primitive degree-$d$ majors can be thought of as braiding the $d$ regions in the complement of the union of its leaves and critical gaps.

A primitive quadratic major is just a diameter of the unit circle, so  $\PM(2)$ is itself a circle.  This corresponds to the fact that the 2-strand braid group is $\Z$.

A primitive cubic major is either an equilateral triangle inscribed in the circle, or a pair of chords that each cut off a segment of angle $2 \pi/3$.
  There are a circle's worth of equilateral triangles.  To each primitive cubic  that consists of a pair of leaves, there is associated a unique diameter
 that bisects the central region.  Thus, the space of two-leaf primitive cubic majors fibers over $S^1$, with fiber an interval.  When the diameter is
turned around by an angle of $\pi$, the interval maps to itself by reversing the orientation, so this is a Moebius band.  The boundary of the Moebius
band is attached to the circle of equilateral triangle configurations, wrapping around 3 times, since, given an equilateral triangle, you can remove any of its three edges to get a limit of two-leaf majors.

\begin{figure}[htpb]
\centering
\includegraphics[width=4.5in]{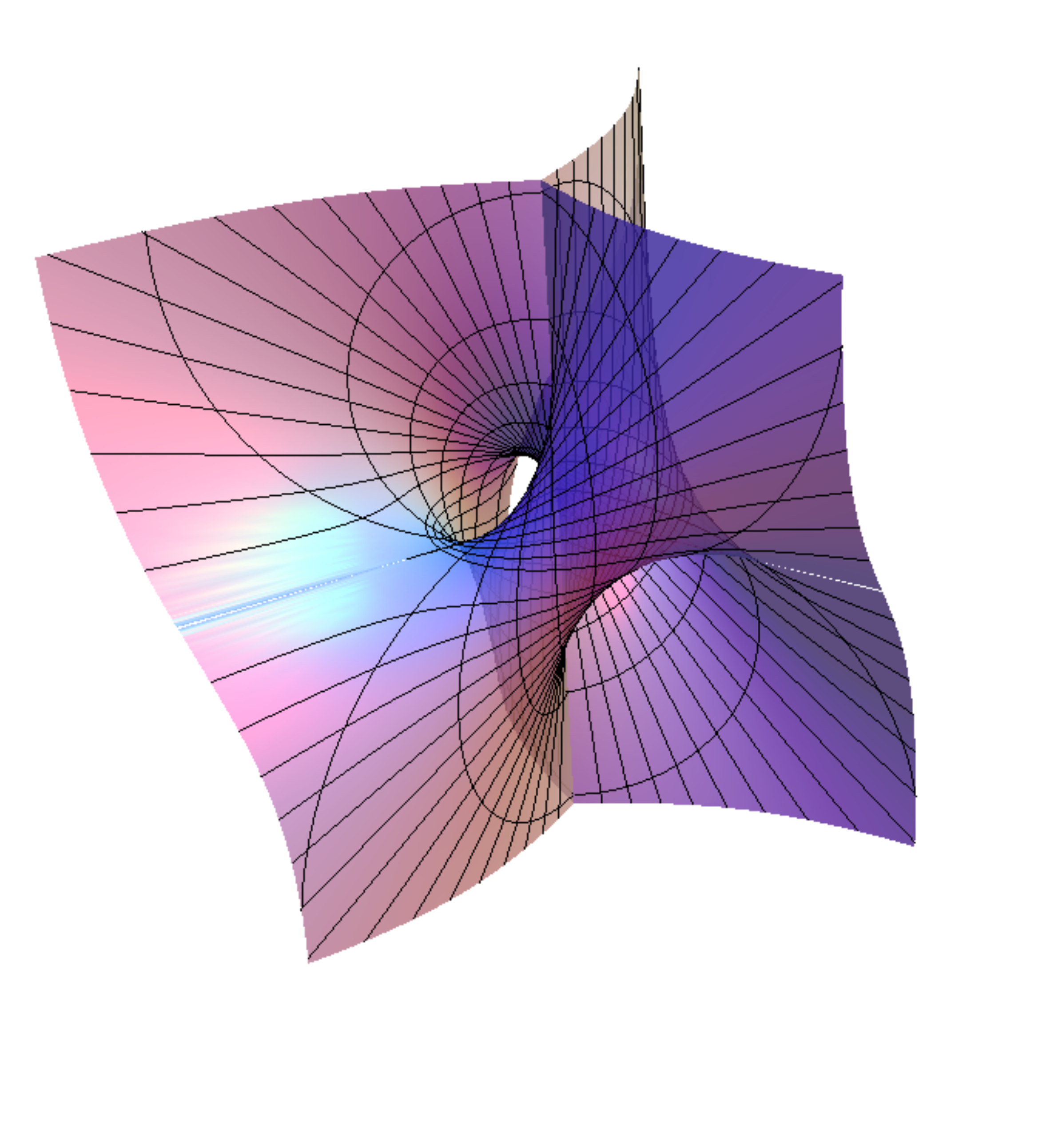}
\caption{This is the space $\PM(3)$ embedded in $S^3$.  It is formed by a $3/2$-twisted Moebius band centered on a horizontal circle that grows wider and wider and wider until its boundary comes together in  the vertical line, which closes into a circle passing through the point at infinity, and wraps around this circle 3 times.  The complement of $\PM(3)$ consists of 3 chambers which are connected by tunnels, each tunnel burrowing through one chamber and arriving  $2/3$ of the way around the post in the middle.  A journey through all three tunnels ties a trefoil knot: such a trefoil knot is the locus where the discriminant of a cubic polynomial is 0, and $\PM(3)$ is a spine for the complement of the trefoil.}
\label{figure:3dimensionalSpine}
\end{figure}

The space of cubic polynomials can be normalized to make the leading coefficient 1 (monic) and, by changing coordinates by a translation, make the second coefficient (the sum of the roots)  equal to 0.  By multiplying by a positive real constant, one can then normalize so that the other two coordinates, as an element of $\C^2$, are on the unit sphere.  The discriminant locus intersects the sphere in a trefoil knot. The spine can be embedded in $S^3$ as follows: start with a $3/2$-twisted Moebius band centered around a great circle in $S^3$.  This great circle can be visualized via stereographic projection of $S^3$ to $\R^3$ as the unit circle in the $xy$-plane.   The Moebius band can be arranged so that it is generated by geodesics perpendicular to the horizontal great circle, in a way that is invariant by a circle action that rotates the horizontal great circle at a speed of 2 while rotating the $z$-axis, completed to a great circle by adding the point at infinity, at a speed of 3. Extend the perpendicular geodesics all the way to the $z$-axis; this attaches the boundary of the Moebius band by wrapping it three times around the vertical great circle.

The spine gives a graphic description for one presentation of the 3-strand braid group,
\[ B_3 = \left \langle a, b | a^2 = b^3 \right \rangle,
\]
where $a$ is represented by the core circle of the Moebius band, and $b$ is the circle of equilateral triangles, the two loops being connected via a short arc to a common base point. The presentation is an amalgamated free product of two copies of $\Z$ over subgroups of index $2$ and $3$.  As braids, $a$ is a 180 degree flip of three strands in a line, while $b$ is a 120 degree rotation of 3 strands forming a triangle; the square of $a$ and the cube of $b$  is a 360 degree rotation of all strands, which generates the center of the 3-strand braid group.

\section{Generating Invariant laminations from Majors}
\label{section:LaminationsFromMajors}
Let $m$ be a degree-$d$ primitive major. How can we construct a degree-$d$-invariant lamination having $m$ for its major?

A leaf of a lamination is defined by an unordered pair of distinct points on the unit circle. The space of possible leaves is topologically an open Moebius band. To see this, consider that any leaf divides the circle into two intervals.  The line connecting the midpoints of these two intervals is the unique diameter perpendicular to the leaf. For each
diameter, there is an interval's worth of leaves, parametrized by the point at which a leaf intersects the perpendicular diameter.  When the interval is rotated 180 degrees, the parameter is reversed, so the space is an open Moebius band.  One way to graphically represent the Moebius band is as the region outside the unit disk in $\RP^2$. The pair of tangents to the unit circle at the endpoints of a leaf intersect somewhere in this region, which is homeomorphic to a Moebius band.

\begin{figure}[htpb]
\centering
\includegraphics[width=3.5in]{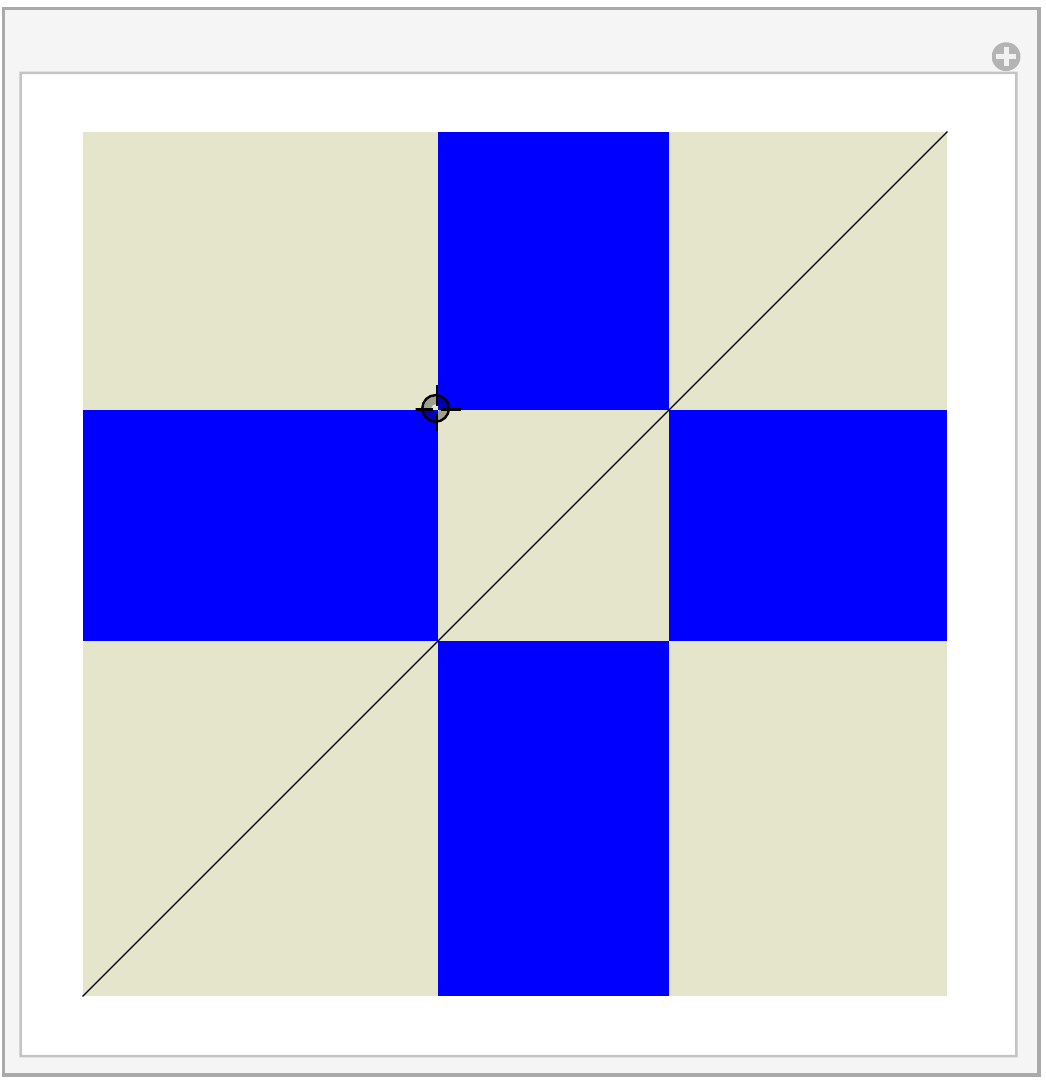}
\caption{
A leaf $\overline{xy}$ of a lamination can be represented by a pair of points $\{(x,y), (y,x) \}$ on a torus.
Leaves that are excluded by $\overline{xy}$ because they intersect it are represented in two shaded rectangles, and leaves compatible with it are represented in two squares of sidelength $b-a \mod 1$ and $a-b \mod 1$.
}
\label{fig:ExcludedLeaves}
\end{figure}

Another way to represent the information is by passing to the double cover, the set of ordered pairs of distinct points on a circle, that is, the torus $S^1 \times S^1$ minus the diagonal, which is a $(1,1)$-circle on the torus, going once around each axis.   Here $S^1=\R/\Z$. For $x,y\in S^1$, we use
$\overline{xy}$ to represent the geodesic in the unit disc linking $e^{2\pi i x}$ and $e^{2\pi i y}$. Each leaf $\overline{xy}$ is represented twice on the torus, as $(x,y)$ and
$(y,x)$.

If a lamination contains a leaf $l = \overline{xy}$, then a certain set $X(l)$ of other leaves are excluded from the lamination because they cross $\overline{xy}$.  On the torus, if you draw the horizontal and vertical circles through
the two points $(x,y)$ and $(y,x)$, they subdivide the torus into four rectangles having the same vertex set; the remaining two common vertices are $(x,x)$ and $(y,y)$.  The leaves represented by points in the interior of two of the rectangles constitute $X(l)$, while the leaves represented by the closures of the other two rectangles are all compatible with the given leaf $l = \overline{xy}$.  We will call this good, compatible region $G(l)$.
The compatible rectangles are actually squares,  of sidelengths $a-b \mod 1$
and $b-a \mod 1$.
 They form a checkerboard pattern, where the two squares of $G(l)$ are bisected by the diagonal.  Another way to express it is that two points $p$ and $q$ define compatible leaves if and only if you can travel on the torus, without crossing the diagonal, from one point to either
the other point or the other point reflected in the diagonal,  heading in a direction between north and west or between south and east (using the same conventions as on maps,
where up is north, left is west, etc. )

Given a  set  $S$ of leaves, the excluded region $X(S)$ is the union of the excluded regions $X(l)$ for $l \in S$, and the good region $G(S)$  is the intersection of the good regions $G(l)$ for $l \in S$.  If $S$ is a finite lamination, then $G(S)$ is a finite union of rectangles that are disjoint except for corners.
\begin{figure}[htpb]
\label{fig:TorusMajorPlot}
\centering
\vbox{
\includegraphics[width=2.2in]{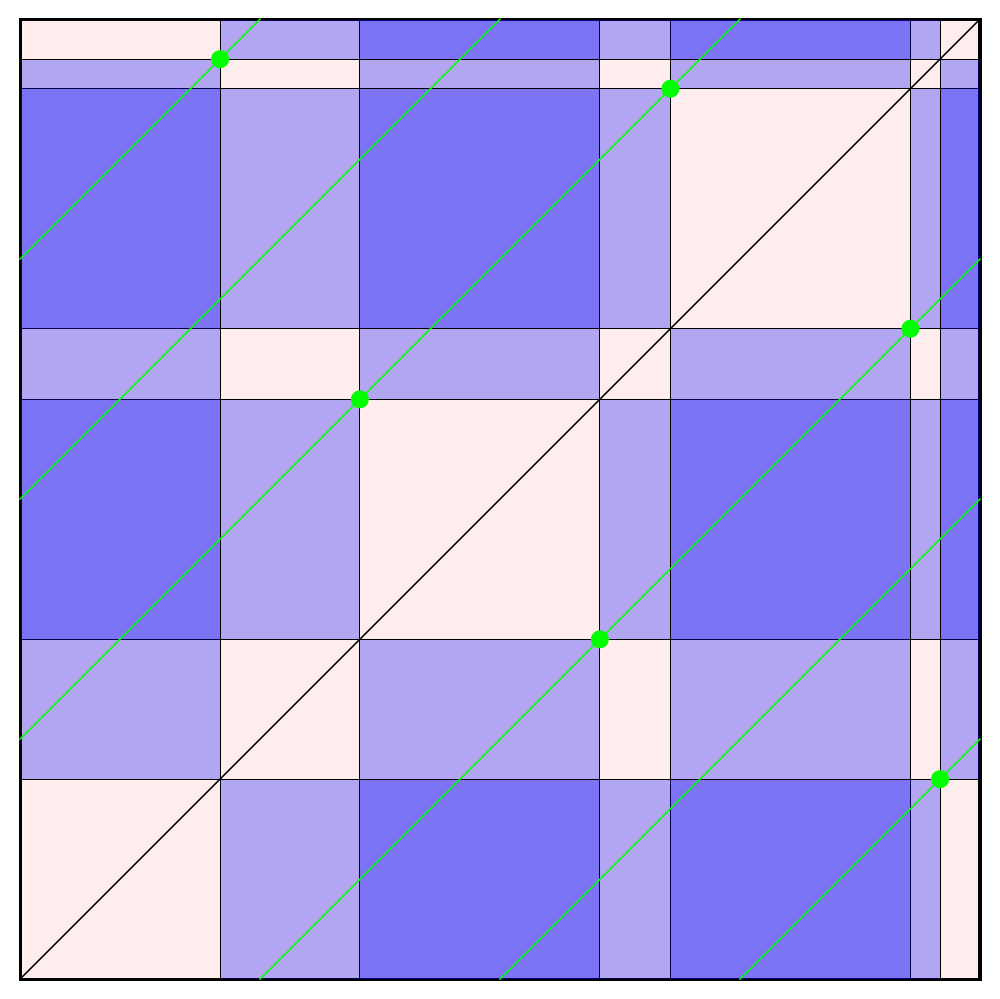}
\includegraphics[width=2.2in]{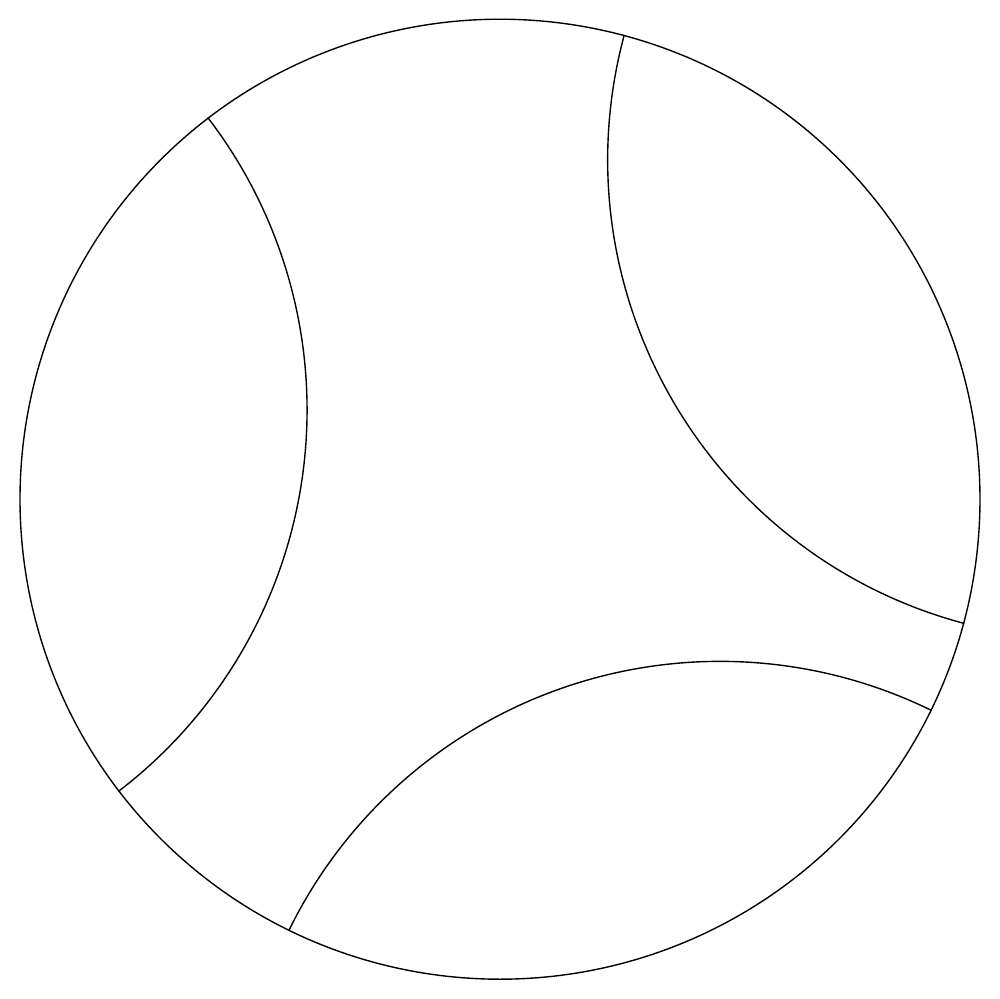}
}
\caption{On the left is a plot of showing the excluded region $X(m)$, shaded, together with the compatible region $G(m)$ on the torus, where $m \in \PM(4)$ is a primitive degree-$4$ major. The figure is symmetric by reflection in the diagonal.
The  quotient of the torus by this symmetry is a Moebius band. Note that $G(m)$ is made up of three $1/4 \times 1/4$ squares (one of them wrapped around) corresponding to the regions that touch the circle in only one edge, together with  $3 \cdot 3=9$ additional rectangles with total area is $1/4^2$, corresponding to the region that touches $S^1$ in 3 intervals.  The 6 green dots represent the leaves of the major, one dot for each orientation of the leaf.
}
\end{figure}

In the particular case of a primitive major lamination $m \in \PM(d)$,  each region of the disk minus $m$ touches $S^1$ in a union of one or more intervals $J_1 \cup \dots \cup J_k$ of total length $1/d$. This determines a finite union of rectangles $(J_1 \cup \dots \cup J_k) \times (J_1 \cup \dots \cup J_k)$ of $G(m)$ whose total area is $1/d^2$ that maps under  the degree $d^2$ covering map
$(x,y) \mapsto d\cdot (x, y)$ to the entire torus.
\begin{figure}[htpb]
\label{fig:TorusMajorPlotPentagon}
\centering
\vbox{
\includegraphics[width=2.2in]{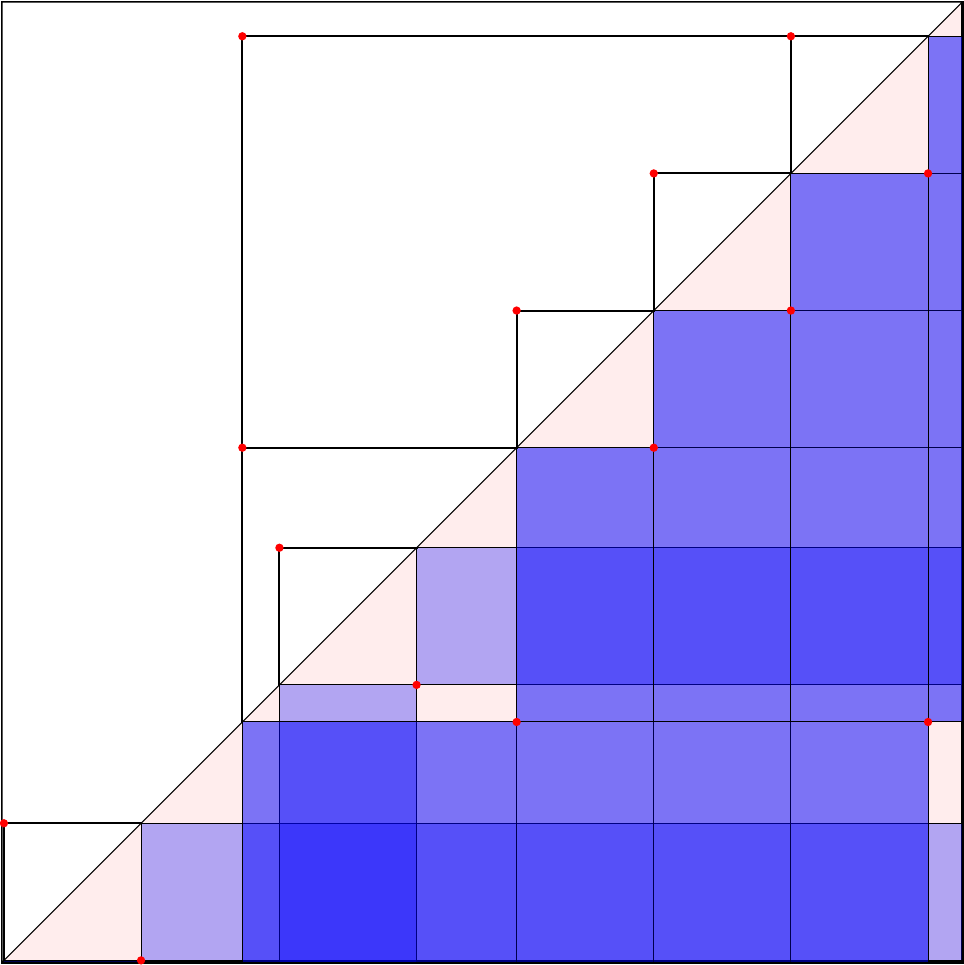}
 \includegraphics[width=2.2in]{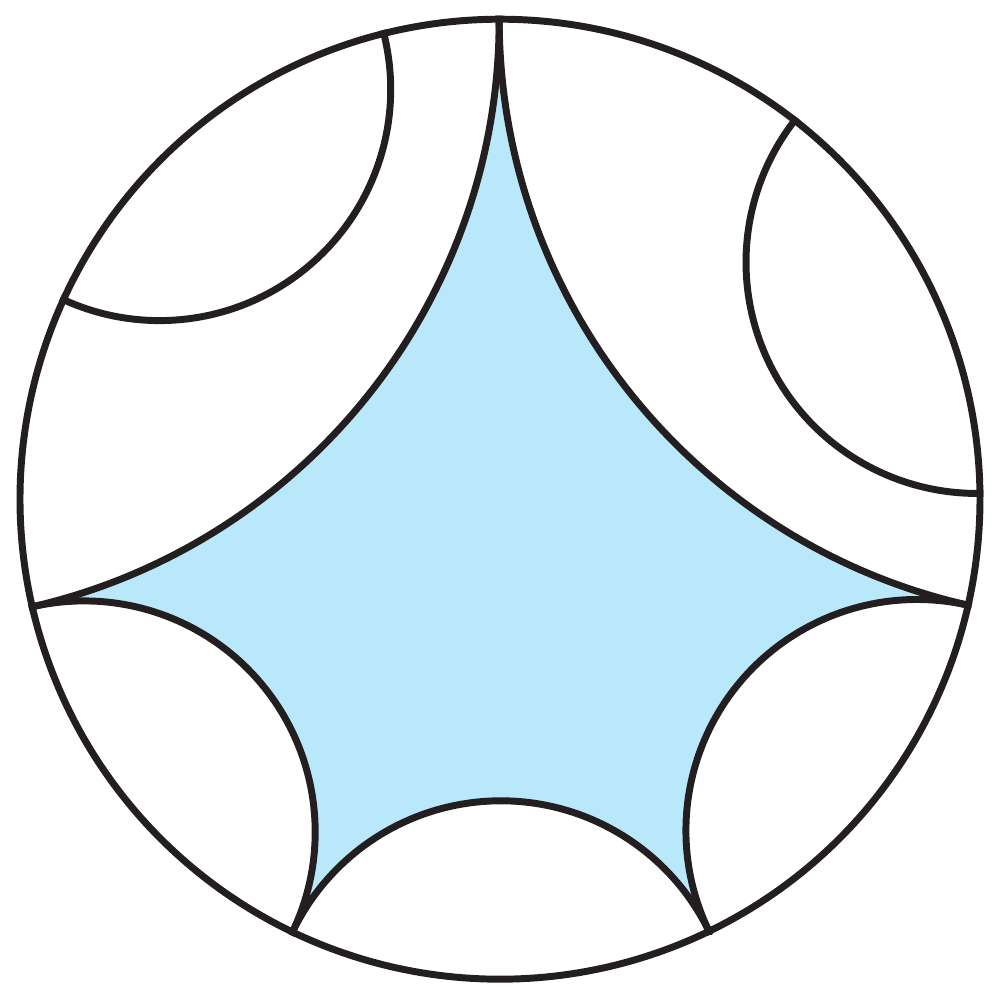}
}
\label{fig:TorusMajorWithPentagon}
\caption{Here is a primitive heptic (degree-7) major with a pentagonal gap, shown in a variation of the torus plot, along with the standard Poincar\'e disk picture.
On the left, half of the torus has been replaced by a drawing that indicates each leaf of the lamination by a path made up of a horizontal segment and a vertical segment. The upper left triangular picture transforms to the Poincar\'e
disk picture by collapsing the horizontal and vertical edges of the triangle to a point, bending the collapsed triangle so that the it goes to the unit disk with  the collapsed edges going to
 $1 \in \C$, then straightening each rectilinear path into the hyperbolic geodesic with the same endpoints.  Notice how the ideal
pentagon corresponds to a rectilinear 10-gon, where two pairs of edges of the 10-gon overlap.
The lower right triangle is a fundamental  domain for a group $\Gamma$ that is the covering group of the torus together with lifts  of reflection through the  diagonal of the torus. The horizontal edge of the
triangle is glued by an element $\gamma \in \Gamma$ to the vertical edge. The action of $\gamma$ on the edge is the same as a 90 degree rotation through the center of the square; $\gamma$ itself
follows this by reflection through the image edge.
 The quotient space $\E^2 / \Gamma$ is topologically a Moebius band. As an orbifold it is $(X*)$, the Moebius band with mirrored boundary.
}
\end{figure}

Consequently we have (here we use $\D^2$ to denote the unit disc):
\begin{proposition}
For any primitive degree-$d$ major $m$, the total area of $G(m)$ is $1/d$.  Almost every point $p \in T^2$
has exactly $d$ preimages in $G(m)$ by the degree $d^2$ covering map $f_d: (x,y) \mapsto d\cdot (x, y)$ with one preimage representing a leaf in each  of the $d$ regions of $\D^2 \setminus m$,
 and all points having at least $d$ preimages in $G(m)$ with at least one preimage representing a leaf in
each region of $\D^2 \setminus m$.
\end{proposition}

For $m \in \PM(d)$ we can now define a sequence of backward-image  laminations $b_i(m)$.
Let $b_0(m) = m$ and inductively define $b_{i+1}(m)$ to be the union
of $m$ with the preimages under $f_d$ of $b_i(m)$ that are in $G(m)$.
\begin{figure}[htpb]
\centering
\vbox{
\includegraphics[width=2.4in]{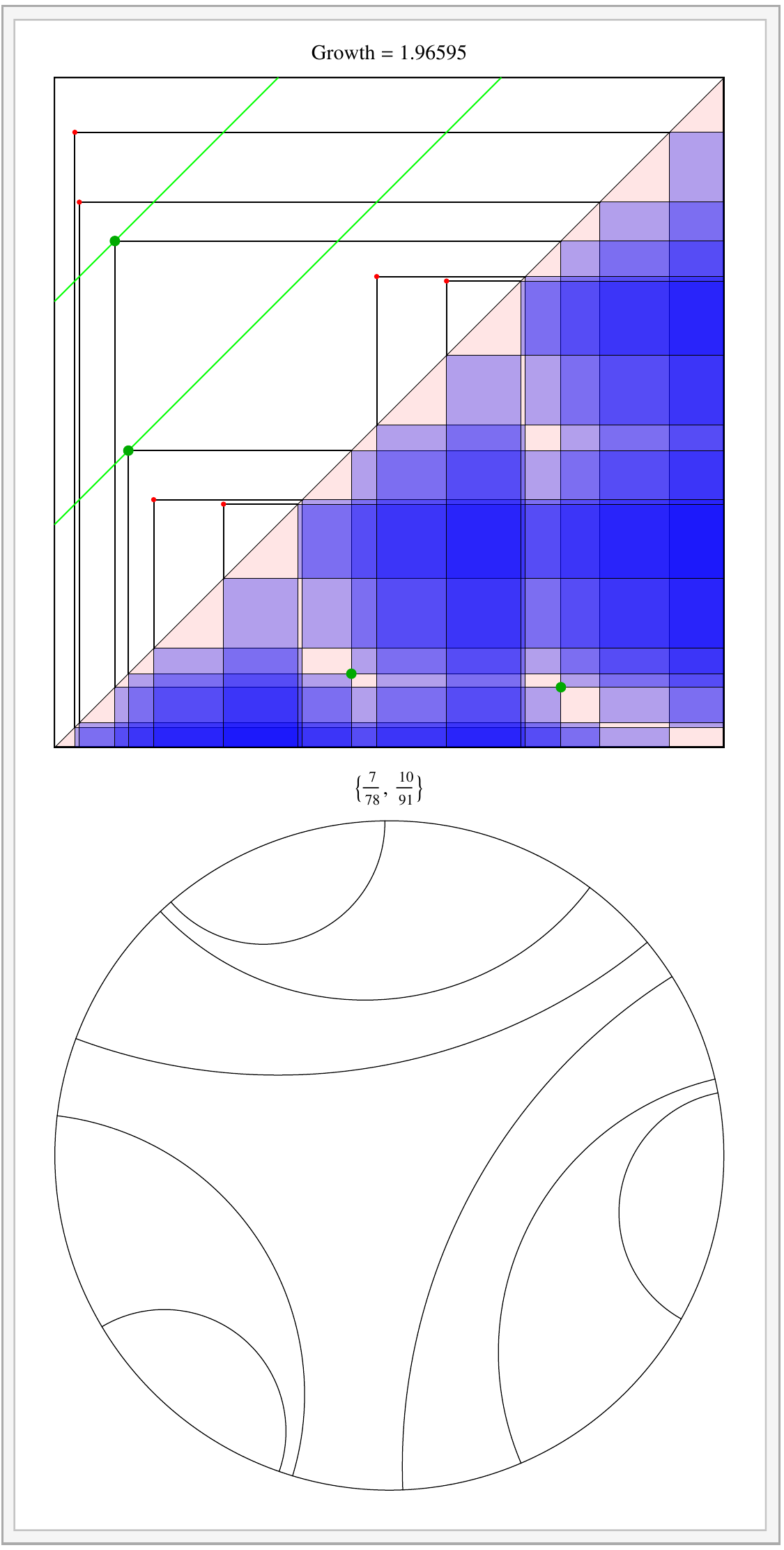}
\includegraphics[width=2.4in]{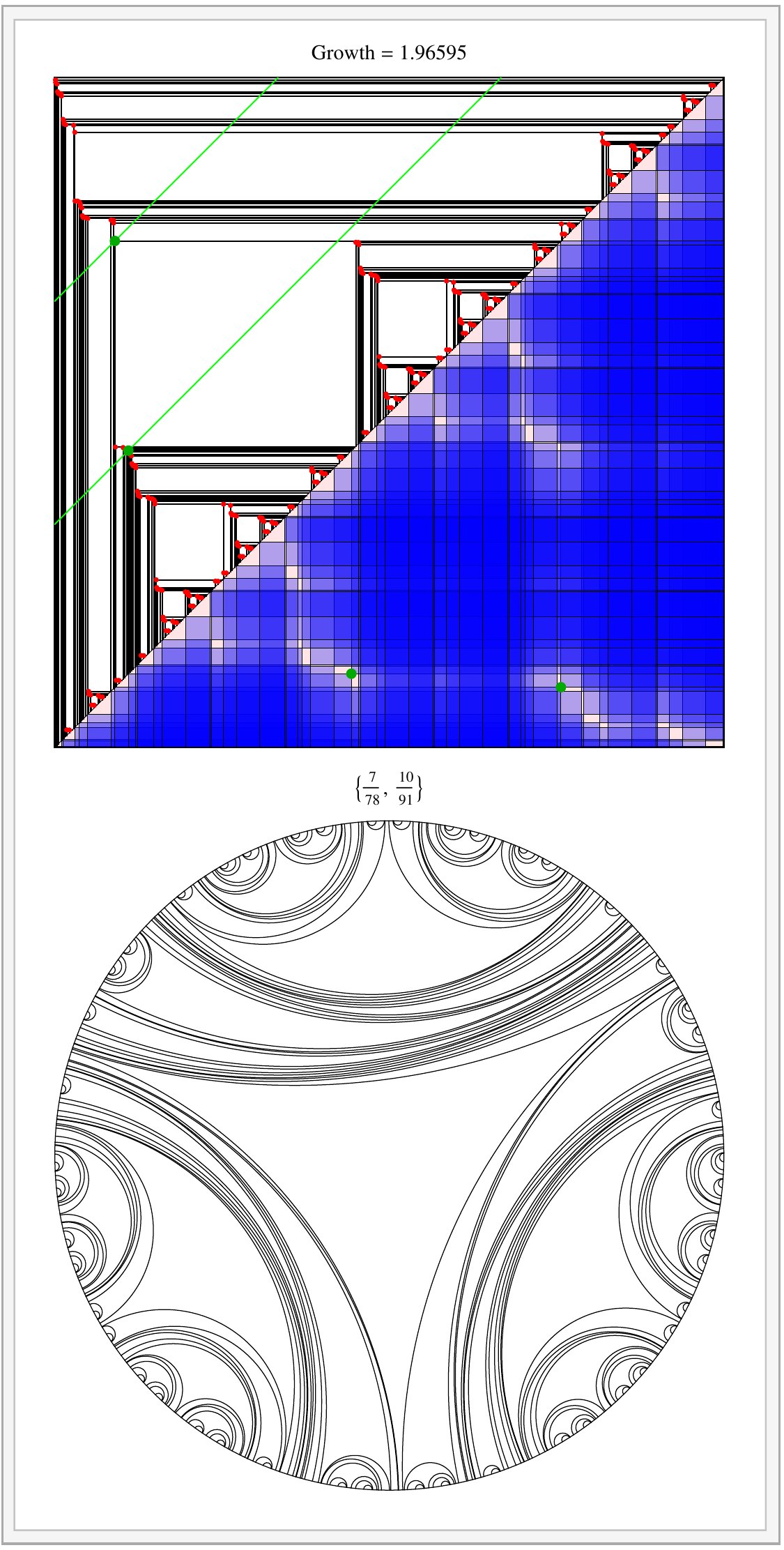}
}
\label{fig:CubicInvariantLamination}
\caption{
On the left is  stage 1 ($b_1(m)$) in building a cubic-invariant lamination for $m = \{ (\frac{10}{91}, \frac{121}{273}),(\frac{7}{78},\frac{59}{78})\} \in \PM(3)$. The two longer leaves of $m$ subdivide the disk into 3 regions,
each with two new leaves induced by the map $f_3$. On the right is a later stage that gives a reasonable approximation of $b_\infty(m)$.}
\end{figure}

\begin{proposition}
For each $i$, $b_i(m)$ is a lamination.
\end{proposition}
\begin{proof}
We need to check that the good preimages under $f_d$ of two leaves of $b_i$ are compatible.
If they are in different regions of $\D^2 \setminus m$ they are obviously compatible.
For two leaves in a single region of $\D^2 \setminus m$, note that when the boundary of
the region is collapsed by collapsing $m$, it becomes a circle of length $1/d$ that is mapped
homeomorphically to $S^1$.  By the inductive hypothesis, the two image leaves are compatible;
since the compatibility condition is identical on the small circle and its homeomorphic image, the leaves are compatible.
\end{proof}

Note that this is an increasing sequence, $b_i(m) \subset b_{i+1}(m)$.   By induction, the good region
$G(b_i(m))$ has area $1/d^m$.

It follows readily that:
\begin{theorem}
The closure $b_\infty(m)$ of the union of all $b_i(m)$ is a degree-$d$-invariant lamination
having $m$ as its major.
\end{theorem}

The lamination $b_\infty(m)$ may have various issues concerning its quality. In particular, it does not always happen that $b_\infty = \Lam(\Rel(b_\infty(m)))$.   We will study the quality of these laminations
later, and develop tools for studying more general degree-$d$-invariant laminations by embedding them in $b_\infty(m)$ for some $m$.
\medskip

Note that as a finite lamination $\lambda$ varies, a compatible rectangle can become thinner and thinner as two endpoints approach each other, and disappear in the limit.
 Thus $G(\lambda)$ is not continuous in the Hausdorff topology.

On the other hand,

\begin{proposition} The map from $\PM(d)$ to the space compact subsets of $T$ endowed with the Hausdorff topology defined by $m \mapsto X(m)$ is a homeomorphism onto its image, i.e, the topology of $\PM(d)$ coincides with Hausdorff topology on the set $\{X(m) : m \in \PM(d)\}$.
\end{proposition}

\begin{proof}
Perhaps the main point is that $X(m)$ is a union of fat rectangles, with height and width at least $1/d$, so pieces of $X(m)$ can't shrink and suddenly disappear.

Suppose $m$ and $m'$ are majors that are within $\epsilon$ in the metric $\md$, and suppose $p \in X(m)$, so there is some leaf $l_m$ of $m$ intersecting the leaf $l_p$ represented by $p$. The endpoints
of $l_m$ have distance $0$ in the quotient graph $S^1 / m$, so there must be a path of length no greater
than $\epsilon$ on $S^1 / m'$ connecting these two points. In particular, assuming $\epsilon < 1/d$,
 some leaf of $m'$ comes within at most $\epsilon$ of intersecting $l_p$, therefore a leaf of $m'$ intersects a leaf near $l_p$ and is in $X(m')$.  Since this works symmetrically between $m$ and $m'$, it follows that they are close in the Hausdorff metric.

Conversely,  given $m \in \PM(d)$ and $\epsilon > 0$, we will show that there exists $\delta > 0$ such that any $m' \in \PM(d)$ with the Hausdorff distance between $X(m)$ and $X(m')$ is less than $\delta$,
the distance $\md(m, m') < \epsilon$.  We will do this by induction on $d$. It is obvious for $d  = 2$, since a major is a diameter and $X(m)$ is a union of two squares that intersect at two corners that are the representatives of the single leaf of $m$.

When $d > 2$, we choose a region of $\D^2 \setminus m$ that touches $S^1$ in a single interval $A$ of length $1/d$, bounded by
a leaf $l$.  Suppose $X(m')$ is Hausdorff-near $X(m)$.  Then most of the  square $A \times A$ of $G(m)$ is in $G(m')$, and most of the rectangles $A \times (S^1 \setminus A)$ and $(S^1 \setminus A) \times A$ of
$X(m)$ is also in $X(m')$. This implies that $m'$ has a nearby $l'$ spanning an interval $A'$ of length $1/d$.  Now we can look at the complementary regions, with $l$ or $l'$ collapsed, normalized by the affine transformation that makes the
 restriction of $m$ or $m'$ a primitive degree-$(d-1)$ major having the collapsed point at 0.
 Call these new majors $m_1$ and $m'_1$. The excluded region $X(m_1)$
 is  obtained from $X(m)$ intersected with a $(d-1) \times (d-1)$ square on the torus, and similarly for $X(m'_1)$; hence
if  the Hausdorff distance between $X(m)$ and $X(m')$ is small, so is that between $X(m_1)$ and $X(m'_1)$.  By induction, we can conclude
that the pseudo-metrics $\met(m)$ and $\met(m')$ on $S^1$  induced from the quotient graphs are close.
\end{proof}
\section{Cleaning laminations: quality and compatibility}

We will use the parametrization of the circle by \emph{turns}, that is, numbers interpreted as fractions of the way around the circle.  Thus $\tau \in [0,1]$ corresponds to the point $\exp(2\pi i \tau)$ in the unit circle in the complex plane.

\begin{figure}[htpb]
\centering
\vbox{
\includegraphics[width=2.5in]{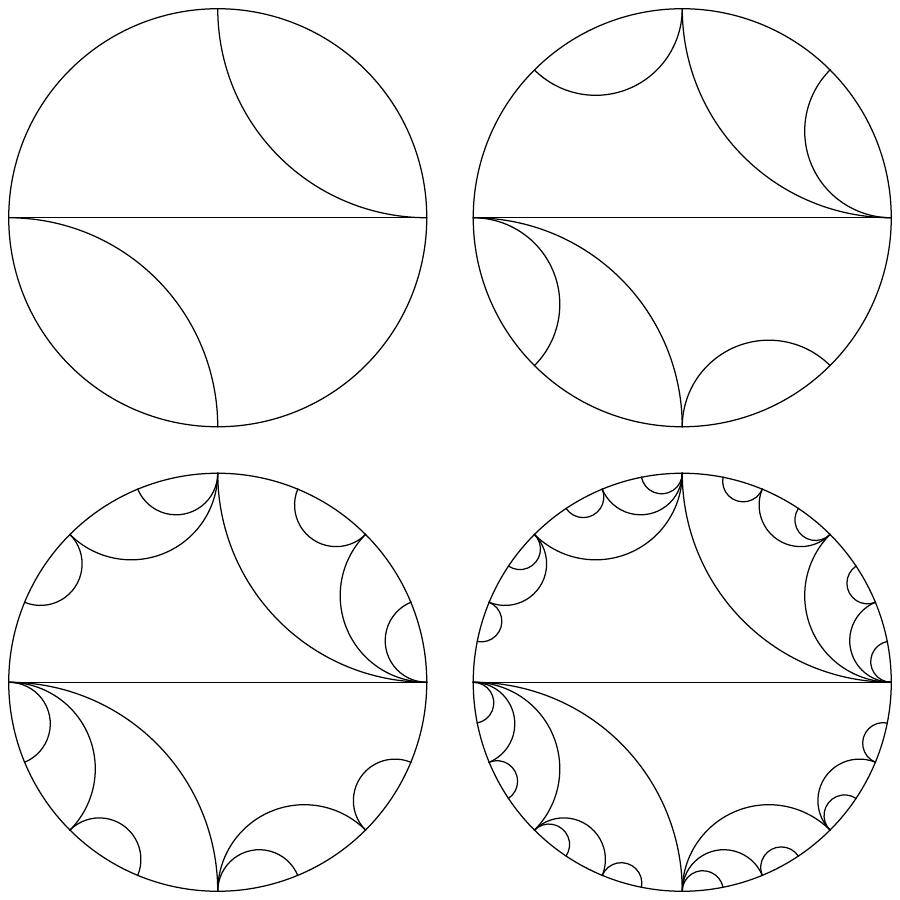}
\includegraphics[width=2.5in]{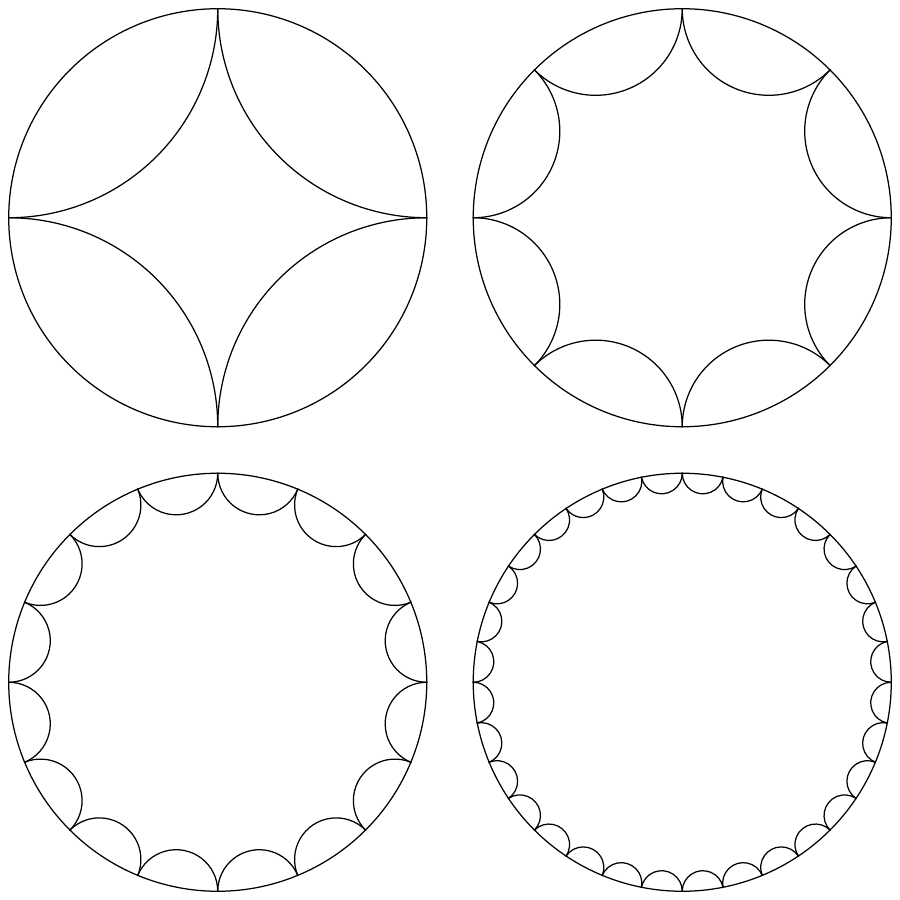}
}
\caption{
The top four laminations are successive steps of building an invariant quadratic lamination for  the major $\{0,\frac12\}$.  The bottom four laminations are obtained by performing
$\Lam\compose\Rel$.
}
\label{fig:UncleanLamination}
\end{figure}

Let's look at the quadratic major $\overline{0, \frac{1}{2} }$. We'll use a variation of the process of backward lifting, depicted in Figure \ref{fig:UncleanLamination}, which is really the limit
as $\epsilon \to 0$ of the standard process applied to $\{-\epsilon, \frac12-\epsilon\}$.
  The first backward lift adds 2 new leaves,
$\overline{0, \frac{1}{4}} $ and $\overline{\frac{1}{2}, \frac{3}{4}} $, and at each successive stage, a new leaf is added joining the midpoint of each interval between endpoints to the last clockwise endpoint of the interval.   (The full construction would also join each of these midpoints to the counterclockwise endpoint of the interval.)
In the limit, there are only a countable set of leaves, but they are joined in a single tree. The closed equivalence relation they generate collapses the entire circle to a point!
\begin{figure}[!htb]
\centering
\includegraphics[width=3in]{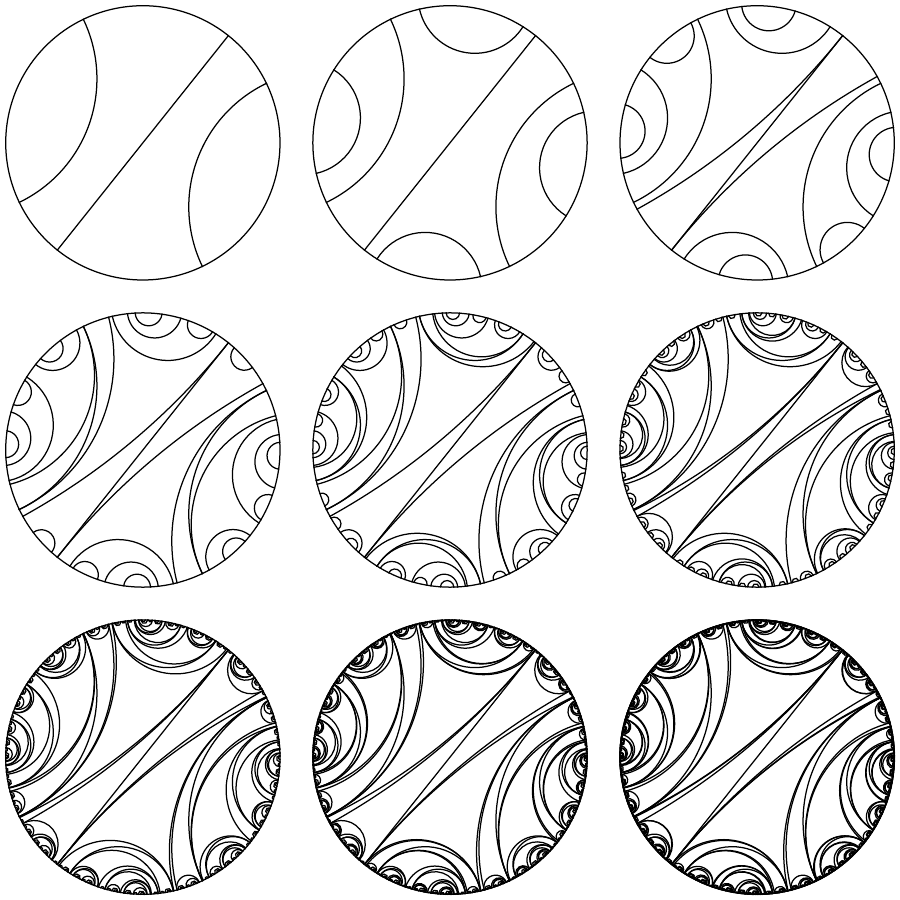}
\includegraphics[width=3in]{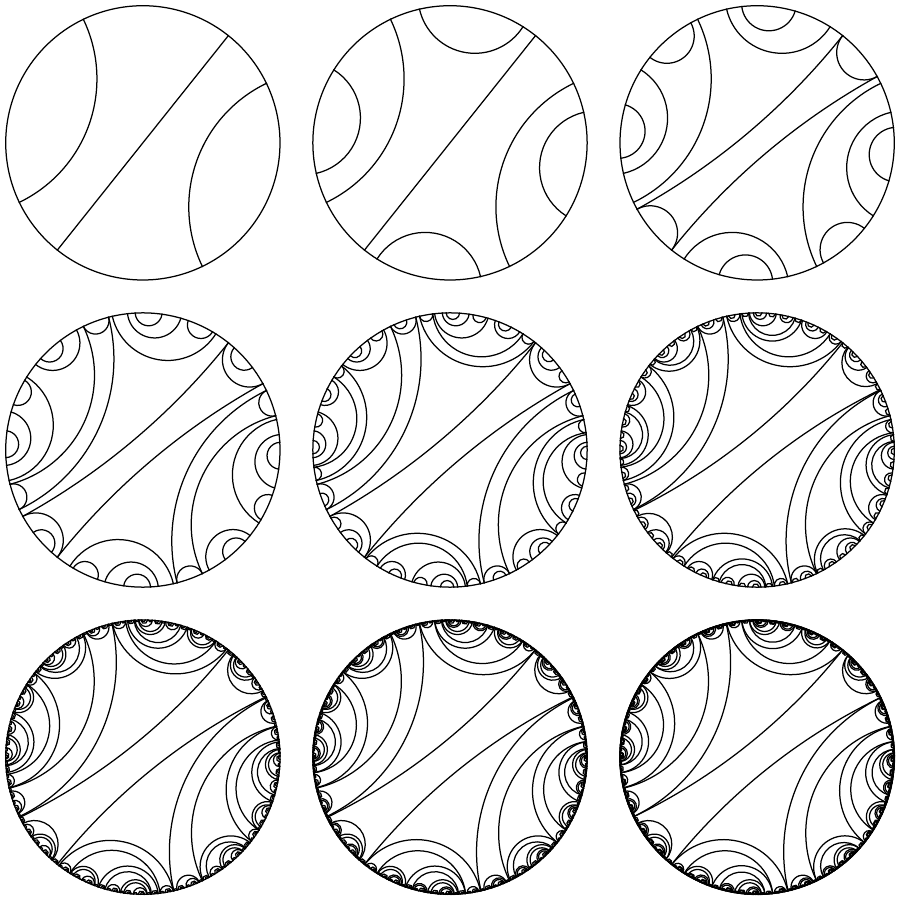}
\caption{
The top nine laminations are successive steps of building an invariant quadratic lamination for the major $\{\frac17,\frac9{14}\}$.  The bottom nine laminations are obtained by performing
$\Lam\compose\Rel$ (at every third step). It gives the lamination of Douady's rabbit.
}
\label{fig:Clean1-7}
\end{figure}

Each finite-stage lamination  $\lambda_k$ can be cleaned up into the standard form $\Lam(\Rel(\lambda_k))$.

\section{Promoting forward-invariant laminations}

A degree-$d$ primitive major is a special case of a lamination that is forward invariant. We have seen how to construct a fully invariant degree-$d$  lamination containing it. How generally can forward invariant laminations be promoted to fully-invariant laminations?

\section{Hausdorff dimension and Growth Rate}

\FloatBarrier

\addtocontents{toc}{\setcounter{tocdepth}{1}} 

 \vspace{1cm}
 \begin{center}
{\bf \sc \Large Part II}
\vspace{.5cm}
\end{center}

Part II consists of the following supplementary sections, which were written by the authors other than W. Thurston:

\tableofcontents

\section{Related work on laminations}
Invariant laminations were introduced as a tool in the study of
complex dynamics by William Thurston \cite{Th}.  Thurston's theory
suggests using spaces of invariant laminations as models for parameter
spaces of dynamical systems defined by complex polynomials.  In degree
$d=2$, Thurston uses \emph{QML} (quadratic minor lamination) to model
the space of $2$-invariant laminations.  Underpinning this approach in
degree $2$ is Thurston's No Wandering Triangle theorem.  Thurston
conjectured that the boundary of the Mandelbrot set is essentially the quotient of the circle by the equivalent relation induced by
QML.  The precise relationship between invariant laminations and
complex polynomials is even less clear for higher degree.

Not all degree $d$-invariant laminations correspond to degree $d$
polynomials. One necessary condition for a lamination to be directly
associated to a polynomial is that it be generated by a tree-like
equivalence relation on $S^1$ (see, for example, \cite{BL,BO1,Ki2,Mimbs}
and Section 2 of Part I).  To distinguish between arbitrary invariant
laminations and those associated to polynomials, we call the invariant
laminations defined by Thurston {\it geometric invariant laminations},
and the smaller class of laminations defined by tree-like equivalence
relations {\it combinatorial invariant laminations}.  (In \cite{BMOV},
such laminations are called $q$-laminations). 

The difference between these two notions can be explained via the following examples. 
Let $a, b, c$ be three distinct points on $S^1$ which form an equivalence class under a tree-like equivalence relation. 
Then in the combinatorial lamination obtained from the given tree-like equivalence, all the leaves $(a, b), (b, c), (a, c)$ are contained. 
But as a geometric lamination, it may have eaves $(a, b), (b, c)$ with out having $(a, c)$ as a leaf. 
Here is another important case to consider. From the ``primitive major"
construction, we can end up with leaves $(a,b), (b,d)$, and
$(c,d)$, where $a, b,c d$ are four distinct points on $S^1$ so that $(a,b,c,d)$ appear in cyclic order. 
On the other hand, if this were a combinatorial lamination, it would only include the ones around
the outside, namely $(a,b), (b,c), (c,d)$, and $(d,a)$. 
One way to understand both of these examples geometrically is to dentify the circle with the ideal boundary of the hyperbolic plane $\mathbb{H}^2$.
Then a combinatorial lamination can be understood as the one consisting of the boundary leaves of the convex hulls of finitely many points on the ideal boundary of $\mathbb{H}^2$.

A fundamental global result in the theory of combinatorial invariant laminations is the existence of {\it locally connected models for connected Julia sets}, obtained by Kiwi \cite{Ki2}.
He associates a combinatorial invariant lamination $\lambda(f)$ to each polynomial $f$ that has no irrationally neutral cycles and whose Julia set is connected.
Then the topological Julia set $J_{\sim_f}:=S^1/_{\sim_f}$ is a locally connected continuum, where $\sim_f$ is the equivalence relation generated by $x\sim_f y$ if $x$ and $y$ are connected by a leaf of $\lambda(f)$, and $f|_{J_f}$ is semiconjugate to the induced map $f_{\sim_f}$ on $J_{\sim_f}$ via a monotone map $\phi: J_f\to J_{\sim_f}$ (by monotone we mean a map whose point preimages are connected).
Kiwi characterizes the set of combinatorial invariant laminations that can be realized by polynomials that have no irrationally neutral cycles and whose Julia sets are connected. In \cite{BCO}, Block, Curry and Oversteegen present a different approach, one based upon continuum theory, to the problem of constructing locally connected dynamical models for connected polynomial Julia sets $J_f$; their approach works regardless of whether or not $f$ has irrational neutral cycles.  These locally connected models yield nice combinatorial interpretations of connected quadratic Julia sets that themselves may or may not be locally connected.

The No Wandering Triangle Theorem is a key ingredient in Thurston's construction (\cite{Th}) of a locally connected model $\mathcal{M}^2_c$ of the Mandelbrot set.  The theorem asserts the non-existence of wandering non-(pre)critical branch points of induced maps on quadratic topological Julia sets.  Branch points of $\mathcal{M}_c$ correspond to topological Julia sets whose critical points are periodic or preperiodic.  Thurston posed the problem of extending the No Wandering Triangle Theorem to the higher-degree case.  In \cite{Le}, Levin showed that for ``unicritical" invariant laminations, wandering polygons do not exist.  Kiwi proved (\cite{Ki1}) that for a combinatorial invariant lamination of degree $d$, a wandering polygon has at most $d$ edges. Block and Levin obtained more precise estimates on the number of edges of wandering polygons in \cite{BL}. Soon after,  Blokh and Oversteegen discovered that some combinatorial invariant laminations of higher-degree ($d\geq3$) do admit wandering polygons (see \cite{BO2}).

Extending Thurston's technique of using invariant laminations to construct a combinatorial model $\mathcal{C}_d$ of the connectedness locus for polynomials of degree $d>2$ remains an area of inquiry.
In \cite{BOPT} and \cite{Pt},  A. Blokh, L. Oversteegen, R. Ptacek and V. Timrion make progress in this direction. They  establish two necessary conditions of laminations from  the polynomials in the {\it Main Cubioid} $CU$, i.e. the boundary of the principal hyperbolic component of the cubic connectedness locus $\mathcal{M}_3$; CU is the analogue of the main cardioid in the quadratic case.  They propose this set of laminations
as the {\it Combinatorial Main Cubioid} $CU^c$, a model for $CU$.

\section{The space of degree $d$ primitive majors}

The next few sections discuss the space $\PM(d)$ of all degree $d$ primitive majors.  In this section, we see various dynamical interpretations of $\PM(d)$. In the following sections, we parametrize $\PM(d)$ and study its topology for small $d$'s. More precisely, we give complete descriptions of $\PM(2)$ and $\PM(3)$, and discuss $\PM(4)$ to some extent.

\subsection{Recalling definitions}
We begin by recalling some concepts from Part I.   A critical class of a degree-d-invariant equivalent relation is subset of $S^1 = \partial \mathbb{D} \subset \mathbb{C}$ that consists of all elements of an equivalence class and that maps under $z \mapsto z^d$ with degree greater than $1$; the associated subsets of $\mathbb{D}$ are the critical leaves and critical gaps of the lamination.   The criticality of a critical class is defined to be one less than the degree of the restriction of the map to that subset.  Per Definition 2.2, the major of a degree-d-invariant lamination (resp. equivalence relation) is the set of critical leaves and critical gaps (resp. equivalence classes corresponding to critical leaves and critical gaps). Such a major is said to be primitive if every critical gap is a `collapsed polygon' whose vertices are identified under the map $z \mapsto z^d$.  (The restriction of $z \mapsto z^d$ to an intact gap of a primitive degree-d-invariant lamination is necessarily injective.)  A critical leaf may be thought of a critical gap defined by polygon with precisely two vertices, and those vertices are identified by $z \mapsto z^d$.  By Proposition , the sum over the critical classes of their criticalities equals $d-1$.

\subsection{Metrizability of $\PM(d)$}

As described in Part I, an element $m \in \PM(d)$ determines a quotient graph $\gamma(m)$ obtained by identifying each equivalence class to a point. The path-metric of $S^1$ defines a path-metric on $\gamma(m)$. In addition, $\gamma(m)$ has the structure of a planar graph, that is, an embedding in the plane well-defined up to isotopy, obtained by shrinking each leaf and each ideal polygon of the lamination to a point.  These graphs have the property that $H^1(\gamma(m))$ has rank $d$, and every cycle has length a multiple of $1/d$.  Every edge must be accessible from the infinite component of the complement, so the metric and the planar embedding together with the starting point, that is, the image of $1 \in \C$, is enough to define the major.

The pseudo-metric  $\met(m)$ on the circle induced by the path-metric on $\gamma(m)$  determines
a continuous function on $S^1\times S^1$. The sup-norm on the space of continuous function on $S^1\times S^1$ induces a metric $\md$ on  $\PM(d)$:
   \[
  \md(m, m') = \sup_{(x, y )\in S^1\times S^1} \left | \met(m)(x,y) - \met(m')(x,y)\right | .
  \]
For the sake of completeness, we give a proof of the fact that $\md$ is indeed a metric.

\begin{lemma} $\md$ is a metric on $\PM(d)$.
\end{lemma}
\begin{proof}
 Non-negativity and symmetry follow automatically from the definition. The rest of the proof is also pretty straightforward.

 Suppose $\md(m, m') = 0$ for some $m, m' \in \PM(d)$. Since $\sup_{x,y} | \met(m)(x,y) - \met(m')(x,y) | = 0$, this means $\met(m)(x,y) - \met(m')(x,y)$ for all $x, y \in S^1$. This implies indeed $m = m'$. In order to see this, suppose $m = \{W_1, \ldots, W_n\}$ and $m' = \{W_1', \ldots, W_k'\}$ are different. Then one can pick two distinct point $p, q$ such that $p, q \in W_i$ for some $i$ but there is no $W_j'$ which contains both $p$ and $q$. Clearly one has $\met(m)(p, q) = 0$ while $\met(m')(p, q) > 0$. This proves that $\md(m, m') = 0$ if and only if $m = m'$.

 It remains to prove the triangular inequality. Let $m_1, m_2, m_3$ be
 three elements of $\PM(d)$. Then
 \begin{align*}
   \md(m_1, m_3) &= \sup_{x, y} \abs{\met(m_1)(x,y) - \met(m_3)(x,y)}\\
                 &\le \sup_{x,y} (\abs{\met(m_1)(x,y) - \met(m_2)(x,y)} + \abs{\met(m_2)(x,y) - \met(m_3)(x,y)})\\
                 &\le \sup_{x,y} \abs{\met(m_1)(x,y) - \met(m_2)(x,y)} + \sup_{x,y} \abs{\met(m_2)(x,y) - \met(m_3)(x,y)}\\
                 &= \md(m_1, m_2) + \md(m_2, m_3).\qedhere
 \end{align*}
\end{proof}

\subsection{A spine for the complement of the discriminant locus}

Let $\mathcal{P}_d$ be the space of monic centered polynomials of degree $d$, and $\mathcal{P}_d^0 \subset \mathcal{P}_d$ the space of polynomials with distinct roots.

There is a natural map $f:\mathcal{P}_d^0 \to \PM(d)$ defined as follows:  for $p\in P^0_d$ consider the meromorphic 1-form
\[
\frac{1}{2d\pi i} \frac{d}{dz} \log p(z) =\frac{1}{2d\pi i} \frac{p'(z)}{p(z)} dz.
\]
Denote by $Z(p)$ the set of roots of $p$. This 1-form gives $\C-Z(p)$ a Euclidean structure. Near infinity, we see a semi-infinite cylinder of circumference $1$  ($\infty$ is a simple pole of $d\log p$ with residue $d$) and near all the zeroes of $p$ we see a semi-infinite cylinder with circumference $1/d$.

We restate Theorem \ref{theorem:pmd_into_polynomialspace}, and give another proof here.

\REFTHM{spine}
\label{thm:Spine}
The map $f:\mathcal{P}_d^0 \to \PM(d)$ is a homotopy equivalence.  More specifically, there exists a section $\sigma: \PM(d) \to \mathcal{P}_d^0$ which is a deformation retract.
\ENDTHM

\begin{proof} For $m\in \PM(d)$ consider $X'_ m=(S^1\times[0,\infty))/ \sim$ where
\[
(\theta_1,t_1)\sim (\theta_2,t_2)\iff (\theta_1,t_1)= (\theta_2,t_2)\quad \text{or}\quad t_1=t_2=0\ \text {and}\ \theta_1\sim_m \theta_2.
\]
The graph quotient of $S^1\times \{0\}$ by $\sim_m$  consists of $d$ closed curves of length $1/d$. Glue to each of these a copy of $(\R/\frac 1d \Z) \times (-\infty,0)$, to construct $X_m$, which is a Riemann surface carrying a holomorphic 1-form $\phi_m$.
The integral of this 1-form around any of the $d$ punctures at $-\infty$ is $1/d$, so we can define a function $p_m:X_m \to \C$
\[
p_m(x)=e^{2\pi i d \int_{x_0}^x \phi_m},
\]
well defined up to post-composition by a multiplicative constant. But the end-point compactification $\overline X_P$ of $X_P$ is  homeomorphic to the 2-sphere, and the complex structure extends to the endpoints, so $\overline X_P$ is analytically isomorphic to $\overline \C$.  With this structure we see that $p_m$ is a polynomial of degree $d$ with distinct roots at the finite punctures; it can be uniquely normalized to be centered and monic, by requiring that $1\times [0,\infty]$ is mapped to a curve asymptotic to the positive real axis.
The map $m\mapsto p_m$ gives the inclusion $\sigma:\PM(d)\to \mathcal{P}_d^0$.

We need to see that $\sigma$ is a deformation retract. For each $p\in \mathcal{P}_d^0$, we consider the manifold with 1-form $(\C-Z(p), \phi_p)$, and adjust the heights of the critical values until they are all $0$.
\end{proof}

\subsection{ Polynomials in the escape locus} Again let $\mathcal{P}_d$ be the space of monic centered polynomials of degree $d$; this time they will be viewed as dynamical systems.  For $p\in \mathcal{P}_d$ let $G_p$ be the {\it Green's function\/} for the filled Julia set $K_p$.

Let $ Y_d(r)\subset \mathcal{P}_d$ be the set of polynomials $p$ such that $G_p(c)=r$ for all critical points of $p$.  The set $Y_d(0)$ is the degree $d$ {\it connectedness locus\/}; it is still poorly understood for all $d>2$.

For $r>0$, there is a natural map $Y_d(r) \to \PM(d)$ that associates to each polynomial  $p\in \mathcal{P}_d$ the equivalence relation $m_p$ on $S^1$ where two angles $\theta_1$ and $\theta_2$ are equivalent if the external rays at angles $\theta_1$ and $\theta_2$ land at the same critical point of $p$.

\REFTHM{equipotential} For $r>0$ and $p\in Y_d(r)$, the equivalence relation $m_p$ is in $\PM(d)$, and the map $p\mapsto m_p$ is a homeomorphism $Y_d(r) \to \PM(d)$.
\ENDTHM

The above theorem is a combination of a theorem of L. Goldberg \cite{Goldberg} and a theorem of Kiwi \cite{Kiwi05}.  To see a more recent proof using quasiconformal surgery, the readers are referred to \cite{Zeng}.

\subsection{Polynomials in the connectedness locus}
For  $m\in \PM(d)$, we geometrically identify it as the unions of convex hulls within $\overline{\mathbb{D}}$ of non-trivial equivalence classes of $m$. Let us denote by $J_1(m), \dots, J_d(m)$ the open subsets of $S^1$ that are intersections with $S^1$ of the components of $\mathbb{D}\setminus m$; each $J_i(m)$ is some finite union of open intervals in $S^1$.  An {\it attribution\/} will be a way of attributing  $J_i(m)\cap J_j(m)$ to  $J_i(m)$, or to $J_j(m)$, or to both.   Call $A$ such an attribution, and denote by $J_i^A$ the interval $J_i$ together with all the points attributed to it by $A$. Define the equivalence relation $\sim_{(m,A)}$ to be
\[
\theta_1\sim_{m,A} \theta_2 \iff \ \text{$d^k\theta_1$ and $d^k\theta_2$ belong to the same $J^A_i$ for all $k\ge 0$.}
\]

Suppose that some $p\in \mathcal{P}_d$ belongs to the connectedness locus without Siegle disks, and that $K_p$ is locally connected, so that there is a Carathéodory loop $\gamma_p:\R/\Z \to \C$.   Then $\gamma_p$ induces the equivalence relation $\sim_p$ on $\R/\Z$ by $\theta_1\sim_p\theta_2$ if and only if  $\gamma_p(\theta_1)=\gamma_p(\theta_2)$.

\REFTHM{existence} There exists $m\in \PM(d)$ and an attribution $A$ such that the equivalence relation $\sim_p$  is precisely $\sim_{m,A}$.
\ENDTHM

This theorem is essentially proved in \cite[Theorem 1.2]{Zeng2}.  Understanding when different $\sim_{m,A}$ correspond to the same polynomial is a difficult problem, even for quadratic polynomials.

\begin{proof} (Sketch)  Choose an external ray landing at each critical value in $J_p$, and an external ray landing at the root of each component of $\overset \circ{K_p}$ containing a critical value if the critical value is attracted to an attracting or parabolic cycle. Then for each critical point $c\in J_p$,  the angles of the inverse images of the chosen rays landing at $c$ form an equivalence class for $(m,A)$.

For each critical point $c \in  \overset \circ  {K_p}$, the angles of the inverse images of the chosen rays landing on the component of $\overset \circ  {K_p}$ containing $c$ form the other equivalence classes. These are the ones that need to be attributed carefully.
\end{proof}

\subsection{Shilov boundary of the connectedness locus}

\begin{conjecture}  All stretching rays through $Y_d(r)$ land on the Shilov boundary of the connected locus. Furthermore, if $r \to 0$, $Y_d(r)$ accumulates to the Shilov boundary.
\end{conjecture}

If this is true, it gives a description of the Shilov boundary of $Y_d(0)$, which is probably the best description of the connectedness locus we can hope for.


\section{Parametrizing primitive majors} \label{s:parametrizingprimitivemajors}

As part of his investigations into core entropy, William P. Thurston wrote numerous Mathematica programs.  This section presents an algorithm found in W. Thurston's computer code which cleverly parametrizes primitive majors using starting angles.

\bigskip

Denote by $I$ the circle of unit length.  We will interpret $I$ as the fundamental domain $[0,1)$ in $\mathbb{R}/\mathbb{Z}$ with the standard ordering on $[0,1)$.  Simultaneously, we will think of $I$ as the space of angles of points in the boundary of the unit disk.

Throughout this section, we will let $m=\{W_1, \cdots, W_s\}\in \PM(d)$ denote a generic primitive major.  By ``generic," we mean that the associated lamination $\mathcal{M}$ consists of $d-1$ leaves and has no critical gaps.  Each leaf $\ell_i$ in $\mathcal{M}$ has two distinct endpoints in $I$.  We will call the lesser of these two points the {\it starting point} of $\ell_i$, and denote it $s_i$, and we will call the greater the {\it terminal point} of $\ell_i$ and denote it $t_i$.  We will adopt the labelling convention that the labels of the leaves are ordered so that $$s_1 < s_2 < \dots < s_{d-1}.$$
Since, for each $i$, $d \cdot s_i \pmod{1}= d \cdot t_i \pmod{1}$, there exists a unique natural number $k_i \in \{1,...,d-1\}$ such that $t_i = s_i + \frac{k_i}{d}$.

Each leaf $\ell_i \in \mathcal{M}$ determines two open arcs of $I$: $I_i=(s_i,t_i)$ and ${I}_i^0 = I \setminus [s_i,t_i]$.  The complement in $\overline{\mathbb{D}}$ of the lamination $\mathcal{M}$ consists of $d$ connected sets.  We will adopt the notation that $C_i$ is the connected component of whose boundary contains $\ell_i$ and has nonempty intersection with the arc $I_i$, for $1 \leq i \leq d$, and $C_0$ is the connected set whose boundary contains an arbitrarily small interval $(1-\epsilon,1) \subset I$.  For each connected set $C_i$, denote by $\mu(C_i)$ the Lebesgue measure of the boundary of $C_i$ in $I = \partial \overline{\mathbb{D}}$.

\begin{lemma}
Let $m \in \PM(d)$ be a generic primitive major.  Then $\mu(C_i) = 1/d$ for all $i$.
\end{lemma}

\begin{proof}
Since for every leaf $\ell_i$ the lengths of the arcs $I_i$ and $I_i^0$ are both integer multiples of $1/d$, $\mu(C_i)$ is also an integer multiple of $1/d$.  Thus, we can write $\mu(C_i) = m_i/d$ for a unique natural number $m_i$.  Then, since the $C_i$ are pairwise disjoint, $$1 = \sum_{i=0}^{d-1} \mu(C_i) = \frac{1}{d} \sum_{i=0}^{d-1} m_i.$$
Consequently, $m_i = 1$ for all $i$.
\end{proof}

\begin{lemma} \label{l:lastleaf}
Let $m \in \PM(d)$ be a generic primitive major.  Then $s_{d-1} < \frac{d-1}{d}$ and $t_{d-1}=s_{d-1}+\frac{1}{d}$.
\end{lemma}

\begin{proof}

First, observe that $s_{d-1}< \frac{d-1}{d}$.
To see this, suppose $s_{d-1} \in [\frac{d-1}{d},1)$.  We know $t_{d-1}$ is the fractional part of $(s_{d-1} + \frac{k_{d-1}}{d} \pmod{1})$ for some $k_{d-1} \in \mathbb{N}$.  Consequently, $t_{d-1} \not \in (\frac{d-1}{d},1)$, contradicting the fact that $s_{d-1}<t_{d-1}$.

Now, suppose $s_{d-1}+\frac{1}{d} < t_{d-1}$.  Since $s_{d-1}$ is the biggest of the $s$'s, there is no leaf in $\mathcal{M}$ whose starting point lies in the arc $I_{d-1}$.  Therefore the boundary of $C_{d-1}$ contains the entire arc $I_{d-1}$, and so $\mu(C_{d-1}) > 1/d$, a contradiction.
\end{proof}

\begin{definition} Let $m \in \PM(d)$ be a generic primitive major.  The \emph{derived primitive major} $m^{\prime}$ is an equivalence relation on $I$ that is the image of $m$ under the following process:  collapse the interval $[s_{d-1},t_{d-1}]$ in $I$ to a point, and then affinely reparametrize the quotient circle so that it has unit length, keeping the point $0$ fixed.
\end{definition}

\begin{lemma}
For any generic primitive major $m \in \PM(d)$, the derived primitive major $m^{\prime}$ is in $\PM(d-1)$.
\end{lemma}

\begin{proof}
By Lemma \ref{l:lastleaf}, the arc $[s_{d-1},t_{d-1}]$ has length $\frac{1}{d}$, so the reparametrization affinely stretches the quotient circle by a factor of $\frac{d}{d-1}$.  For $i \leq d-2$, denote the image of $s_i$ and $t_i$ in $\mathcal{M}^{\prime}$ by $s_i^{\prime}$ and $t_i^{\prime}$.  If $t_i - s_i = \frac{k_i}{d}$, then
$$t_i^{\prime} - s_i^{\prime} = \frac{k_i^{\prime}}{d} \cdot \frac{d}{d-1} = \frac{k_i^{\prime}}{d-1},$$ where $k_i^{\prime} = k_i-1$ if $[s_{d-1},t_{d-1}] \subset [s_i,t_i]$ and $k_i^{\prime}=k_i$ otherwise.  In either case, $(d-1)\cdot(t_i^{\prime}-s_i^{\prime}) = 0 \pmod{1}$. Hence $m'=\{(s_i', t_i') \mid i=1,\cdots, d-2\}$ is in $\PM(d-1)$.
\end{proof}

\begin{lemma} Let $m \in \PM(d)$ be a generic primitive major.  Then $s_i < \frac{i}{d}$ for all $i$.
\end{lemma}

\begin{proof}
 Repeatedly deriving the major $m$ yields a sequence of primitive majors $$m^{(0)} (:=m), m^{(1)} (:= m^{\prime}), m^{(2)}, \dots, m^{(d-3)}.$$
The major $m^{(j)}$ consists of $d-1-j$ leaves, $\ell^{(j)}_1,\dots,\ell^{(j)}_{d-1-j}$, that are the images under $j$ derivations of the leaves $\ell_1^{(0)}, \dots,\ell_{d-1-j}^{(0)}$ of the original major $m$.  We will denote the starting point of the leaf $\ell_k^{(j)}$ by $s_k^{(j)}$.

We wish to show, for any fixed $i$, that $s_i^{(0)} < \frac{i}{d}$.  The major $m^{(d-1-i)}$ consists of $i$ leaves, of which $\ell_i^{(d-1-i)}$ has the largest starting point.  Hence by Lemma \ref{l:lastleaf}, we have $$s_i^{(d-1-i)} < \frac{i}{i+1}.$$
When deriving $m^{(j-1)}$ (which is in $\PM(d-j+1)$) to form $m^{(j)}$, for any $j$, we collapse an interval of length $\frac{1}{d-j+1}$ and rescale by a factor of $\frac{d-j+1}{d-j}$.  Thus,

\begin{align*}
s_i^{(d-1-i)} &= s_i^{(0)} \cdot \frac{d}{d-1} \cdot \frac{d-1}{d-2} \cdot ... \cdot \frac{d-(d-1-i)+1}{d-(d-1-i)} \\
& = s_i^{(0)} \cdot \frac{d}{i+1}. \\
\end{align*}
Hence $$s^{(0)}_i \cdot \frac{d}{i+1} < \frac{i}{i+1},$$ implying $$s^{(0)}_i < \frac{i}{d} .$$

\end{proof}

\begin{theorem} \label{t:StartsToMajor}
Given any increasing sequence $0 \leq s_1 < s_2 < ... < s_{d-1} < 1$ such that $s_i < \frac{i}{d}$ for all $i$, there is a unique degree-$d$ invariant primitive major $m$ whose starting points are $s_1,...,s_{d-1}$ and there is an algorithm to find $m$.
\end{theorem}

\begin{proof}
Let $m \in \PM(d)$ be any primitive major whose leaves have starting points $s_1,\dots,s_{d-1}$; we will show that the terminal points of each leaf of $m$ are uniquely determined.

When we derive $m$ $n$ times, we collapse a union of arcs, namely $\bigcup_{i=1}^{n} cl(I_{d-i})$.  The starting point $s^{(n)}_{d-1-n}$ is the biggest starting point of the resulting major, $m^{(n)}$.  By Lemma \ref{l:lastleaf}, $t^{(n)}_{d-1-n}-s^{(n)}_{d-1-n} = \frac{1}{d-n}$.  Reversing the rescaling process which at each derivation rescales the quotient circle to have unit length, an interval of length $\frac{1}{d-n}$ in $\mathcal{M}^{(n)}$ corresponds to an interval of length $\frac{1}{d}$ in the original major $\mathcal{M}$ when we measure the arcs which get collapsed as having length $0$.

Thus, for each natural number $n<d-1$, $t_{d-1-n}$ is the smallest number in $[0,1)$ such that
\begin{equation} m \left( [s_{d-1-n},t_{d-1-n}] \setminus \bigcup _{i=1}^n I_{d-i} \right) = \frac{1}{d}, \label{eq:algorithm} \end{equation}
where $m$ is Lebesgue measure.  For $n=0$, we have $t_{d-1}=s_{d-1}+\frac{1}{d}$, and thus $I_{d-1}=(s_{d-1},s_{d-1}+\frac{1}{d})$, by Lemma \ref{l:lastleaf}. Inductively, if $t_{d-1},...,t_{d-n}$ are known, Equation \ref{eq:algorithm} gives $t_{d-n-1}$.
\end{proof}

Theorem \ref{t:StartsToMajor} describes an algorithm which associates a primitive major to any ordered sequence of points $\{s_i\}_{i=1}^{i=d-1}$ such that $s_i < \frac{i}{d}$ and $s_i < s_{i+1}$ for all $i$.  We now describe an algorithm used in W. Thurston's code for constructing such sequences of points from arbitrary collections of points.

\begin{definition}
For any sequence $X = \{x_1,\dots,x_{d-1}\}$ of $d-1$ distinct points $x_i$ in $I$, define $A$ to be the map $$A:X \mapsto Y$$ where $Y=\{y_1,\dots,y_{d-1}\}$ is the sequence of $d-1$ points $y_i$ in $I$ defined by the following process:

\begin{enumerate}
\item Reorder and relabel (if necessary) the elements of the sequence $X$ so that $$x_1 < x_2 < \dots < x_{d-1}.$$

\item Set
$$y_i =  \begin{cases} x_i &\mbox{if } x_i < \frac{i}{d} \\
x_i - \frac{1}{d}  & \mbox{if } x_i  \geq \frac{i}{d}.
 \end{cases} $$
\end{enumerate}
\end{definition}

\begin{theorem}
Let $X = \{x_1,\dots,x_{d-1}\}$ be any sequence of $d-1$ points $x_i$ in $I$. Then there exists $n \in \mathbb{N}$ such that $A^n(X) = A^{m}(X)$ for all $m \geq n$.   For such an $n$, $A^n(X) = \{a_1,\dots,a_{d-1}\}$ is a sequence of $d-1$ distinct points $a_i$ in $I$ such that $$a_1 < a_2 < \dots < a_{d-1}$$ and $a_i < \frac{i}{d}$ for all $i$.

\end{theorem}

\begin{proof} For convenience, denote by $B^j = \{b_1^j,\dots,b_{d-1}^j\}$ the result of applying Step 1 of the definition of the map $A$ to $A^j(X)$.  Thus, $B^j$ is just the sequence $A^j(X)$ reordered and relabeled so that $b_1^j < \dots < b_{d-1}^j$.  Notice that when passing from $A^j(X)$ to $A^{j+1}(X)$ (with $B^j$ as an intermediate step), if we do not subtract $1/d$ from at least one of the elements $b^j_i$, then $B^j = A^m(X)$ for all $m > j$.

Since we only subtract $1/d$ from $b_i^j$ if $b_i^j \geq i/d \geq 1/d$, every $b^j_i$ is nonnegative.  Since $\sum_i b^0_i < \infty$, we can only subtract $1/d$ finitely many times from elements of $B^0$ without some $b^j_i$ becoming negative.   Hence, there are only finitely many integers $j$ such that we subtract $1/d$ from at least one element of $B^j$ when passing from $A^j(X)$ to $A^{j+1}(X)$.   Hence, there exists $n_0 \in \mathbb{N}$ such that $B^{n_0} = A^m(X)$ for all $m > n_0$.  Set $n = n_0+1$.  Then $B^{n_0} = A^n(X) = A^m(X)$ for all $m \geq n$.

By assumption, $b_1^{n_0} < \dots < b_{d-1}^{n_0}$.  Since we do not subtract $1/d$ from any $b^{n_0}_i$ when passing from $A^{n_0}(X)$ to $A^{n}(X)$, this means $b^{n_0}_i < i/d$ for all $i$.
\end{proof}

\section{Understanding $\PM(3)$}
\subsection{Topology of $\PM(3)$}
We now explicitly describe $\PM(3)$. One can represent points on the
circle by their angle from the positive real axis. This angle is
measured as the number of turns we need to get to that point, i.e., as
a number between 0 and 1.
Let $m \in \PM(3)$.  In the generic case, $m$ has two leaves, each of which bounds one third of the circle. Assume that we start with a generic choice of a cubic primitive major and rotate it counterclockwise. We get a new cubic primitive major at each angle until we make one full turn. But there are two types of special cases to look at. For brevity of writing, we will say simply \emph{a major} to denote a cubic primitive major throughout this section.

One case is when two major leaves share an end point. Putting an extra
leaf connecting the non-shared end points of the major leaves, we get
a regular triangle. In fact, which two sides of this regular triangle
you choose as the major does not matter. Therefore, we get an extra
symmetry in this case; if we rotate such a major, then it does not
take a full turn to see the same major again: only a $1/3$-turn is
enough. Let's call the set of all majors of this type the
\emph{degeneracy locus}. One can think of this space as the space of regular
triangles inscribed in the unit circle.

Another special case is when two major leaves are parallel (if drawn
as straight lines). We call the set of all majors of this type
the \emph{parallel locus}. In this case, after making a half-turn, we
see the same major again.

Starting from these two types of singular loci allows us to
understand the topology of the space $\PM(3)$. Note that both the
degeneracy locus and the parallel locus are topological circles. Let's
see what a neighborhood of a point on the degeneracy locus looks
like. From a major $m$ on the
degeneracy locus, there are three different ways to move into the
complement: remove one of the sides of the regular triangle and open
up the shared end point of the remaining two, as illustrated in Figure
\ref{fig:threenearbymajors}. They are all nearby, but
there is no short path connecting two of $A, B, C$ without crossing
the degeneracy locus.

\begin{figure}	
	\centering
	\includegraphics[scale=0.18]{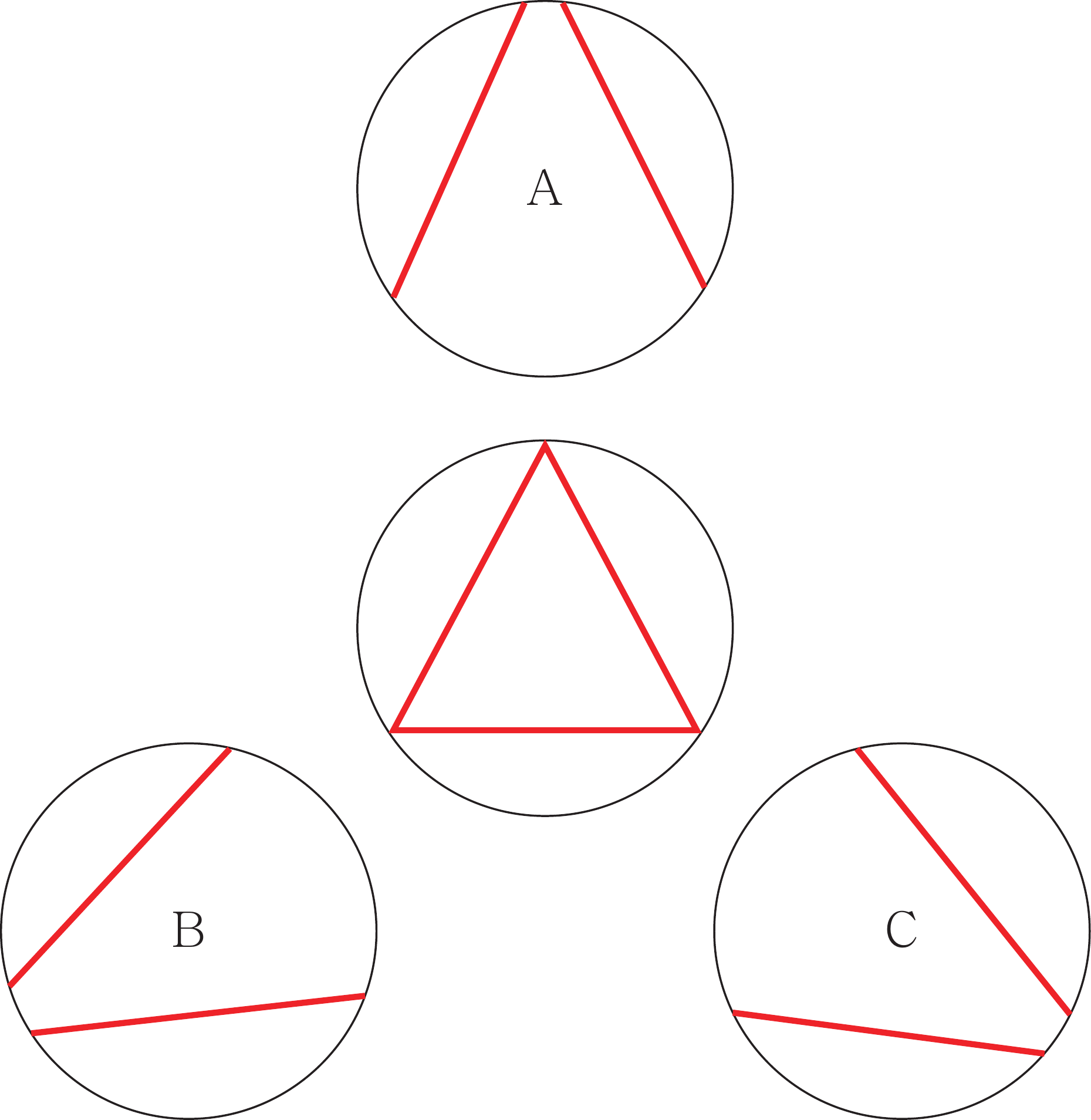}
	\caption{The three generic majors $A,B$ and $C$ are close to the degenerate major $m$ shown in the center.}
	\label{fig:threenearbymajors}
\end{figure}

On the other hand, if you move along a locus consisting of majors
that are distance $\epsilon$ away from the degeneracy locus for a
small enough positive number $\epsilon$ and look at the closest
degenerate major at each moment, then you will see each degenerate
major exactly three times, as in Figure~\ref{fig:nearbymajors}.

\begin{figure}	
	\centering
	\includegraphics[scale=0.35]{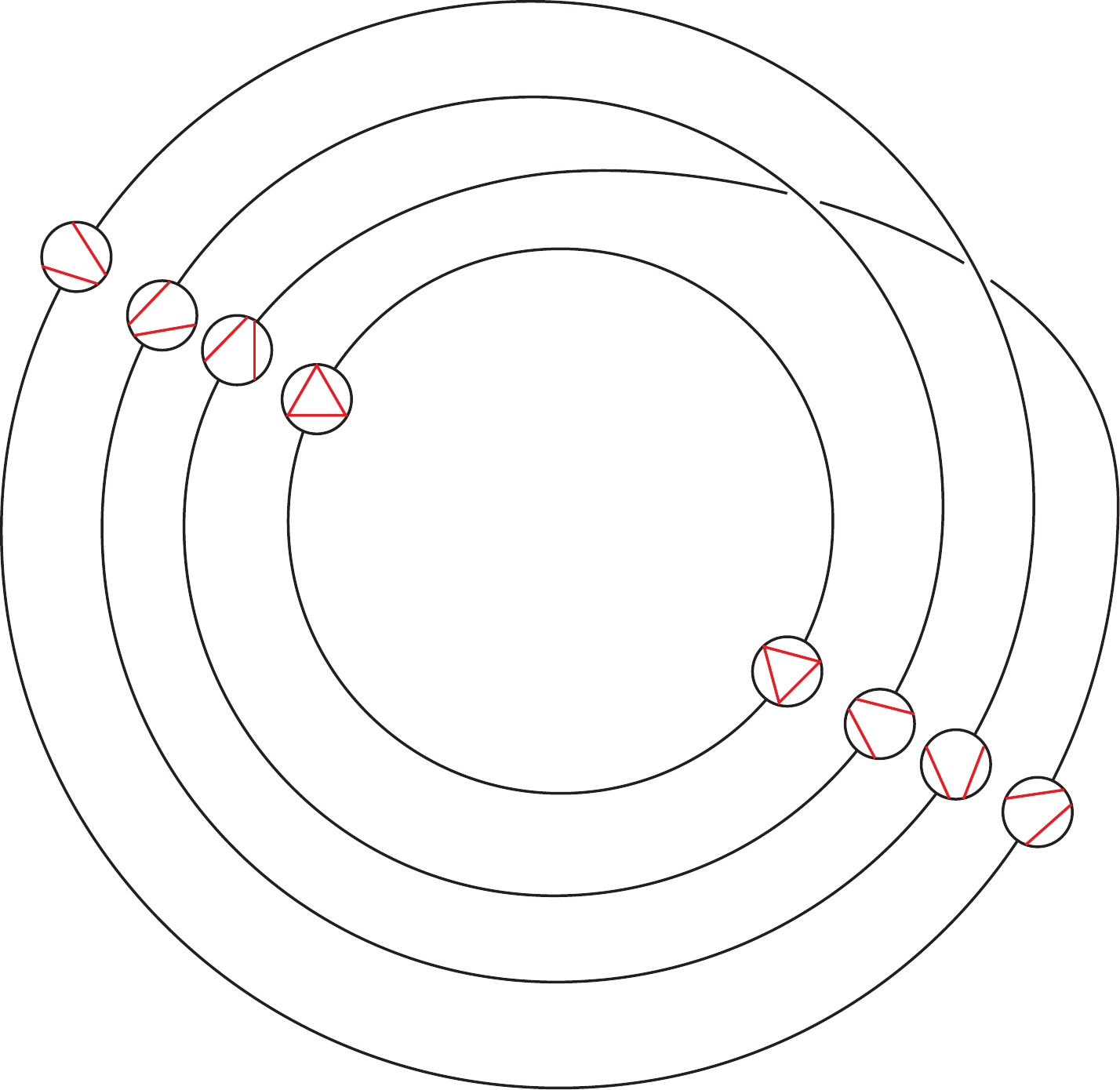}
	\caption{ The degeneracy locus (the inner-most circle) and the
          locus of majors which are distance~$\epsilon$ away from it.}
	\label{fig:nearbymajors}
\end{figure}

Hence, a neighborhood of the degeneracy locus is obtained from a tripod
cross an interval by gluing the tripod ends with a 2/3-turn. The
boundary is again a topological circle which is embedded in $\R^3$ as
a trefoil knot (see the image on the right side of Figure \ref{fig:neighborhoods}).

Now, we move on to the local picture around the parallel
locus. Almost the same argument works, except now the situation is
somewhat simpler and the neighborhood is homeomorphic to a M\"{o}bius
band. One can embed this space into $\R^3$ so that the boundary is
again a trefoil knot. (See the image on the left side of Figure \ref{fig:neighborhoods}.) Now the whole space $\PM(3)$ is obtained from
these two spaces by gluing along the boundary. Figure \ref{figure:3dimensionalSpine} illustrates what the space looks like after this gluing.

Visualizing $\PM(3)$ by
dividing into neighborhoods of two singular loci in this way allows one
to see that $\PM(3)$ is a $K(B_3, 1)$-space.

We first construct the universal cover
of $\PM(3)$. Write $\PM(3)$ as $A \cup B$, where $A$ is the closure of the
neighborhood of the degeneracy locus and $B$ is the closure of the
neighborhood of the parallel locus. We glue
them in a way that $A \cap B = \partial A = \partial B$.

Now it is easy to see that $\widetilde{A}$ is just the product of a tripod
with $\R$, and $\widetilde{B}$ is simply an infinite
strip. $\widetilde{A}$ has three boundary lines each of which is glued
to a copy of $\widetilde{B}$, and each end of each copy of
$\widetilde{B}$, one needs to glue a copy of $\widetilde{A}$, and so
on. To get $\widetilde{\PM(3)}$, we need to do this infinitely many
times and finally get
the product of an infinite trivalent tree with $\R$, which is
obviously contractible. Therefore, all the higher homotopy groups of
$\PM(3)$ vanish.

On the other hand, the Seifert--van-Kampen theorem says that
\begin{align*}
  \pi_1(\PM(3)) &= \pi_1(A) *_{\Z} \pi_1(B) \\
     &= \langle \alpha, \beta \mid \alpha^3 = \beta^2 \rangle
\end{align*}
which is one presentation for~$B_3$.

\begin{figure}
  \[
  \includegraphics[width=.4\textwidth]{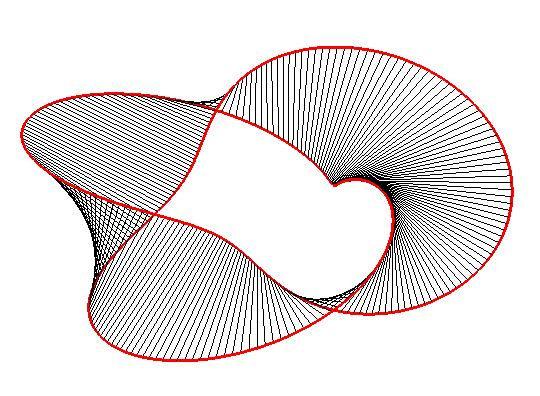}\qquad
  \reflectbox{\includegraphics[width=.4\textwidth]{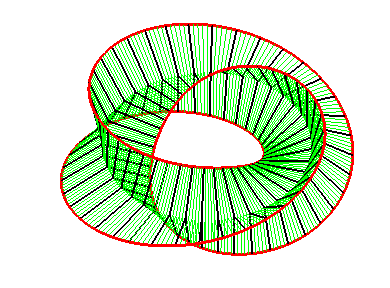}}
  \]
  \caption{Neighborhoods of the parallel locus and the degeneracy locus}
  \label{fig:neighborhoods}
\end{figure}

\subsection{Parametrization of $\PM(3)$ using the angle bisector}
Let $m$ be a non-degenerate cubic major. The end points of leaves of $m$ divide the
circle into four arcs, two of them with length 1/3 and the other
two have length between 0 and 1/3. Call these other two arcs $I$
and~$J$. Draw a line $L$ passing through the midpoint of $I$ and the
midpoint of $J$, and let $\theta$ be the angle from the positive real
to~$L$. Relabeling $I$ and $J$ if necessary, let
$I$ be the interval that $L$ meets at the angle $\theta$, and let
$a$ be the length of the interval $I$. Our parameters are $a$ and
$\theta$, as illustrated in Figure~\ref{fig:cubicmajor}. Note that we can
choose $\theta$ from $[0, 1/2]$, since $(a, \theta)$ represents the same major as
$(1/3 - a, \theta - 1/2 \mod 1)$. Also note that $a$ runs from 0
to 1/3.
Hence, the set $\{(a,\theta) : 0 \le a \le 1/3, 0 \le \theta \le 1/2\}$ with appropriate identifications on the boundary gives a parameter space of $\PM(3)$ (see Figure \ref{fig:paramforcubicmajor}).

When $a$ is either $0$ or $1/3$, either $I$ or $J$ becomes a single point, and this corresponds to the degeneracy locus. The locus where $a = 1/6$ is the parallel locus (the red line in Figure \ref{fig:paramforcubicmajor}).

\begin{figure}
  \[
  \includegraphics[scale=0.3]{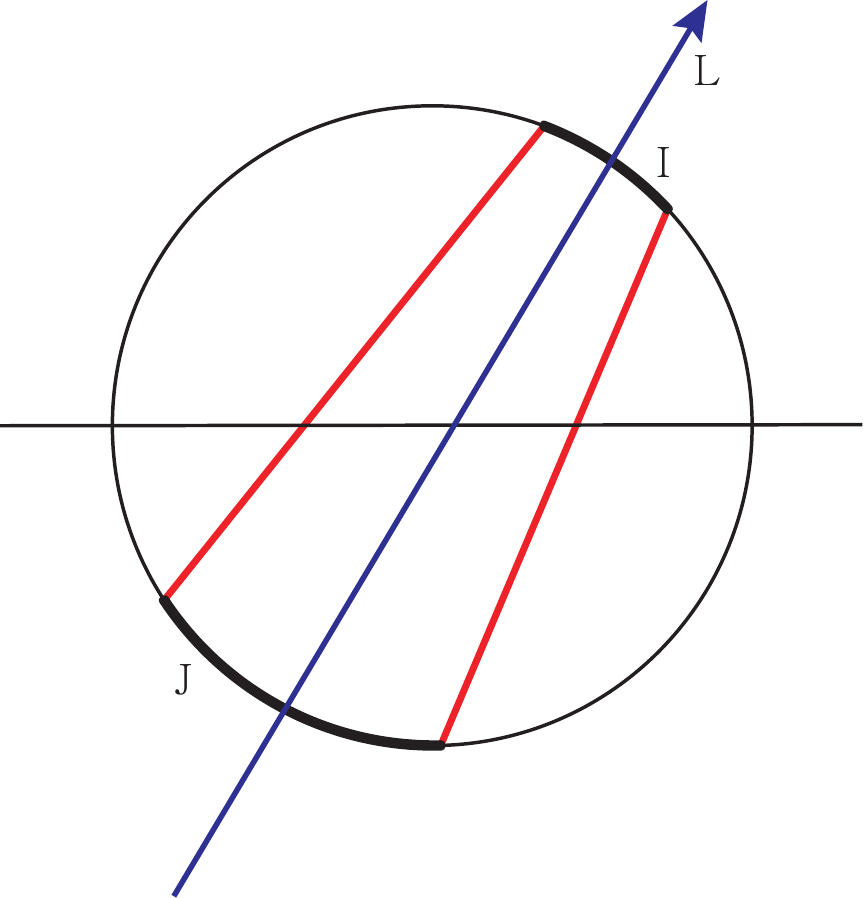}
  \]
  \caption{The red leaves represent a cubic major and the blue line is the angle bisector.}
  \label{fig:cubicmajor}
\end{figure}

\begin{figure}
    \[
  \includegraphics[scale=0.35]{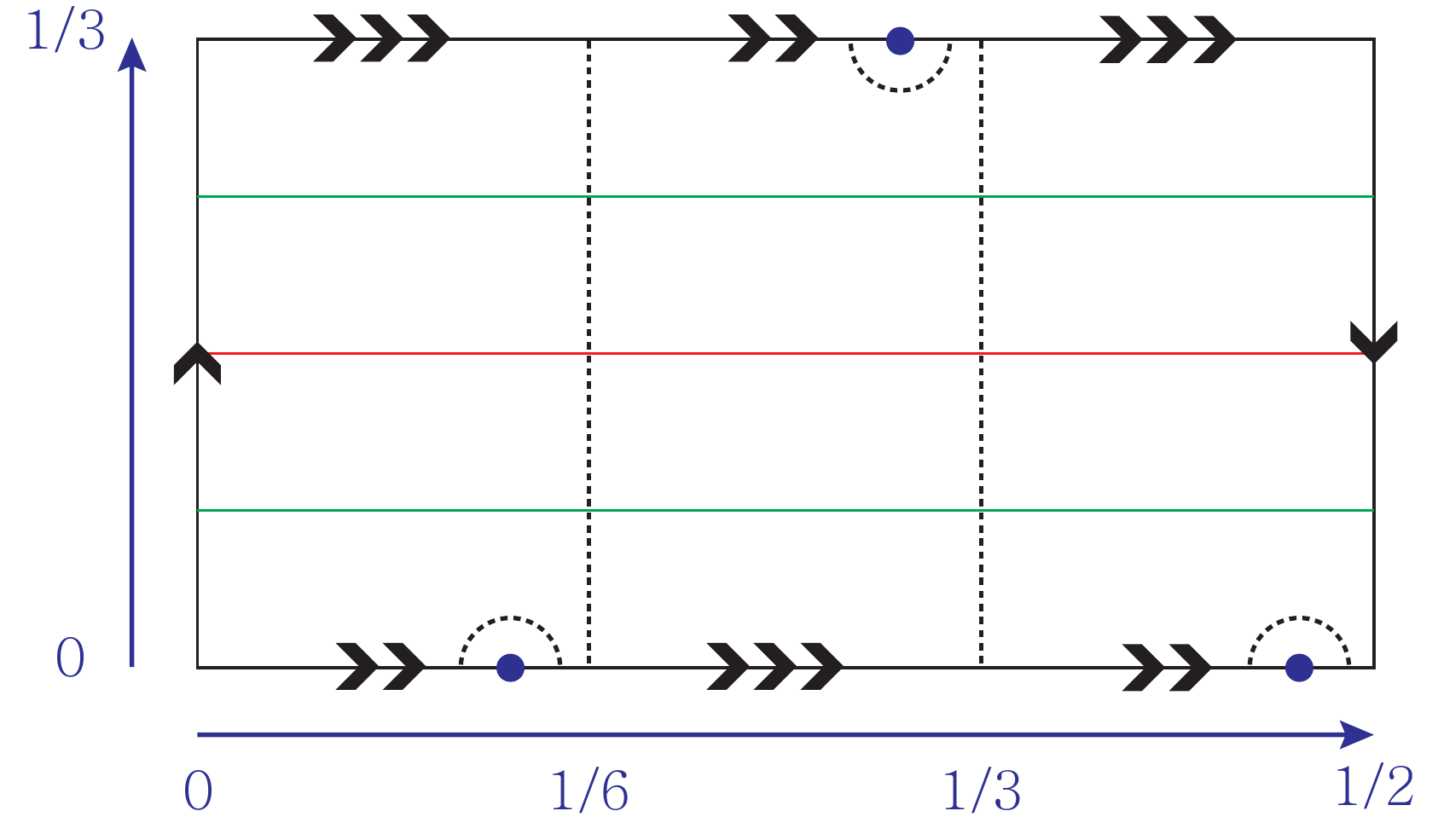}
  \]
		\caption{The parameter space for cubic majors using the angle bisector.}\label{fig:paramforcubicmajor}		
\end{figure}

It easy to see that $(a, 0)$ and $(1/3 - a, 1/2)$ represent the same major. Also observe that, for any $0\le \theta \le 1/6$, the pairs
$(0, \theta)$, $(1/3, \theta + 1/6)$, and $(0, \theta + 1/3)$ are just
different choices of two edges of the regular triangle whose vertices
are at $(\theta, \theta + 1/6, \theta + 1/3)$. Hence they must
be identified. Similarly, for $1/6 \le \theta \le 1/3$,
$(1/3, \theta - 1/6)$, $(0, \theta)$, and $(1/3, \theta + 1/6)$ represent the
same point. For instance, the three blue dots in Figure \ref{fig:paramforcubicmajor} should be identified.

\subsection{Embedding of $\PM(3)$ into $S^3$}

We can visualize how $\PM(3)$ embeds into $S^3$. Consider the
decomposition of $S^3$ into two solid tori glued along the boundary.
Put the degeneracy and parallel locus as central circles of the solid tori.
Seeing $S^3$ as $ \R^3$ with a point at infinity, one may assume that the parallel locus coincides with the unit circle
on $xy-$plane and the degeneracy locus is the $z-$axis with the point
at infinity. Then one can view $\PM(3)$ with one point removed as a 2-complex in~$\R^3$.

We already know that $\PM(d)$ embeds as a spine of the
complement of the discriminant locus, so it would be instructive to
see the discriminant locus in this picture. A monic centered cubic
polynomials is written as $z^3 + az + b$ for some complex numbers $a,
b$. Hence the space of all such polynomials can be seen as $\C^2$. The
unit sphere is the locus $|a|^2 + |b|^2 = 1$, and the discriminant
locus is $4a^3 + 27b^2 = 0$. The intersection of these two loci is a
trefoil knot. We will embed $\PM(3)$ into $S^3$ so that it forms a
spine of the complement of this trefoil knot.

Consider the stereographic projection $\phi: S^3 \setminus \{
(0,0,0,1) \} \to  \R^3$ defined by
 $$\phi(x_1, x_2, x_3, x_4) =
\left(\dfrac{x_1}{1-x_4}, \dfrac{x_2}{1-x_4},
  \dfrac{x_3}{1-x_4}\right).$$
   Instead
of taking the line segment connecting a point on the unit circle of
$xy$-plane and the $z$-axis, we first take the preimages of these two
points under $\phi$ and consider the great circle passing through them
in $S^3$. Then we take the image of this great circle under
$\phi$. While one wraps up the parallel locus twice and the degeneracy
locus three times, we construct a surface as the trajectory of the
image of the great circle passing through the preimages of the points
on the parallel and the degeneracy locus. Now it is guaranteed to be
an embedded 2-complex (not a manifold, since the degeneracy
locus is singular) by construction.

Let's return to the space of normalized cubic polynomials
\[
\bigl\{ \, z^3 + az + b \bigm| a, b \in \C, \abs{a}^2 + \abs{b}^2 = 1 \, \bigr\}.
\]
We can identify the degeneracy locus inside this space as the subset
cut out by $a = 0$, and the parallel locus as the subset cut out by $b
= 0$. To connect the degeneracy locus to the parallel locus by
spherical geodesics, running
three times around the degeneracy locus while running twice around the
parallel locus, we look at the subset
\begin{align*}
\PM(3) &= \bigl\{ \, z^3 + s \theta^2 + t \theta^3 \bigm|
  s, t \in \R_{\ge 0}, \theta \in S^1, \abs{s\theta^2}^2 +
  \abs{t\theta^3}^2 = 1\, \bigr\}\\
   &= \Bigl\{\,z^3 + a z^2 + b z^3 \Bigm|
       a, b \in \C, \abs{a}^2 + \abs{b}^2 = 1,
     \frac{a^3}{b^2} \in [0,\infty] \,\Bigr\}.
\end{align*}
This is clearly disjoint from the discriminant locus.

Recall that in the proof of Theorem \ref{thm:Spine}, we constructed  a section $\sigma: \PM(d) \to \mathcal{P}_d^0$
where  $\mathcal{P}_d^0$ is the space of all monic centered polynomials of degree $d$ with distinct roots.  The above embedding of $\PM(3)$ gives another
section into $\mathcal{P}_3^0$, but it is not exactly the same as $\sigma$ in Theorem \ref{thm:Spine}. For instance, $\sigma$ has the property that
for each polynomial $f$ in the image of $\PM(3)$ under $\sigma$, all the critical values of $f$ have the same modulus, while the above embedding
does not have this property.

\subsection{Other Parametrizations}

It is not apriori clear that the angle-bisector parametrization of $\PM(3)$ can be
generalized to the parametrization of $\PM(d)$ for higher~$d$. In this
section, we briefly survey several different ways of parametrizing $\PM(3)$
and see the advantages and disadvantages of each method.

\subsubsection{Starting point method}\label{sec:starting-point}
Here we start with the most native way to parametrize the space of primitive majors. Start from angle $0$ and walk around the circle until you meet an end of a major leaf, and say the angle is $x_1$. Keep walking until you meet an end of different major leaf, say the angle is $x_2$. The numbers $x_1, x_2$ are regarded as the starting points of the major leaves of given cubic major. We have two cases: either $x_2 - x_1 < 1/3$ and the leaves are $\{(x_1, x_1 - 1/3 \mod 1), (x_2, x_2+1/3)\}$ or $x_2 - x_1 \ge 1/3$ and the leaves are $\{(x_1, x_1 + 1/3), (x_2, x_2 + 1/3)\}$. As you see, it is fairly easy to get a neat formula for leaves. $x_1$ has the range from $0$ to $1/3$ and $x_2$ as the range from $0$ to $2/3$. But not every point in the rectangle $[0,1/3] \times [0, 2/3]$ is allowed. First of all, there is a restriction $x_2 \ge x_1$, and sometimes even $x_2 \ge x_1 + 1/3$. So, this parametrization method is not as neat as the rectangular parameter space obtained by the angle bisector method.
Another issue is that there are many more combinatorial possibilities in higher degrees (compare with \S  \ref{s:parametrizingprimitivemajors}). We will see this in more detail while we discuss the next method.

\subsubsection{Sum and difference of the turning number}
Given a cubic lamination, start from angle 0 and walk around the circle until the first time you meet two consecutive ends $x < y$ belonging to distinct leaves. We call these numbers $x, y$ turning numbers of the given lamination. Let $S = y + x$ and $D = y - x$ (S stands for the `sum' and D stands for the `difference'). Then one can easily get $x=(S-D)/2$ and $y=(S+D)/2$, and the leaves are $\{x, x-1/3 \mod 1\}, \{y,y+1/3 \mod 1\}$. Note that $D$ runs from $0$ to $1/3$ and for a given $D$, $S$ runs from $D$ to $4/3 - D$. Hence we get a trapezoid shape domain for the parameter space and one can figure out which points on the boundary are identified as we did for the other models.

It is pretty clear what each parameter means and one gets a neat formula for the leaves. On the other hand, it has some drawbacks when one tries to generalize to higher degree cases. Even for degree 4, the complement of the degeneracy locus in $\PM(4)$ is not connected. (This is what we postponed discussing in the last subsection. See Figure \ref{fig:pm4}).
Hence, however one defines the turning numbers, it is hard to determine which configuration one has. One can divide the domain into pieces, each of which represents one combinatorial configuration, and give a different formula for each such piece. This requires understanding the different possible configurations. In particular, one can start with counting the number of connected components of the complement of the degeneracy locus in $\PM(d)$. Tomasini counted the number of components in his thesis (see Theorems 4.3.1 and 4.3.2., pp.
118-121, in \cite{Toma}).

\subsubsection{Avoiding the M\"{o}bius band}
One way to avoid the M\"{o}bius band is this:  if you take the quotient of the set of majors by the symmetry $z \leftrightarrow -z$
that conjugates a cubic polynomial to a dynamically isomorphic polynomial, then the M\"{o}bius band folds in half to an annulus.   One boundary component of the annulus is wrapped three times around a circle (laminations with a central lamination), and the other boundary component consists of laminations whose majors are parallel. The parameter transverse to the annulus is the shortest distance between endpoints of the majors, in the interval [0,1/6]. The other parameter is the most clockwise endpoint of this shortest distance interval.

\section{Understanding $\PM(4)$}
\label{sec:pm4}

We now give some brief comments on the shape of $\PM(4)$. Unlike in $\PM(2)$ and $\PM(3)$,
there is more than one way that a generic primitive major can be
arranged topologically: the complement of the degeneracy locus is not
connected. The two possibilities are illustrated in
Figure~\ref{fig:pm4}.

\begin{figure}
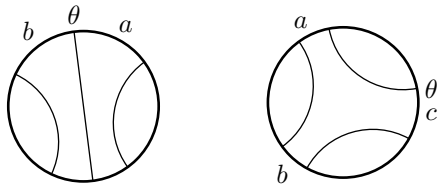

  \[\includegraphics{majors-0}
  \qquad\qquad
  \includegraphics{majors-1}
  \]
  \caption{The two topological types of generic primitive majors in
    $\PM(4)$, with the parameters marked.}
  \label{fig:pm4}
\end{figure}

$\PM(4)$ is a 3-complex, and each of these topological types of
generic majors contributes a top-dimensional stratum of the
3-complex. The \emph{parallel
  stratum} is the piece corresponding to the stratum on the left,
containing the part where the three leaves are
parallel. Topologically, the parallel stratum may be parametrized by
triples $(\theta,a,b)$ with $\theta$ in the circle and $a,b \in (0,1/4)$.
Here $\theta$ is the angle
to one endpoint of the central leaf and $\theta - a$ and $\theta + b$
are the angles to adjacent endpoints of the other two leaves, so that
the three leaves have endpoints
\[
(\theta-a, \theta-a-1/4),\quad(\theta, \theta + 1/2),
  \quad\text{and}\quad(\theta+b,\theta+b+1/4)
\]
(with coordinates interpreted modulo~$1$).
There is an
equivalence relation:
\[
(\theta, a, b) \equiv (\theta+1/2, 1/4 - b, 1/4 - a).
\]
Therefore, the parallel stratum topologically is a square cross an interval,
with the top glued to the bottom by a half-twist. As a manifold with
corners, this piece has 2 codimension-1 faces, each an annulus.

The other stratum, the \emph{triangle stratum}, can be parametrized by
quadruples $(\theta,a,b,c)$ with $a,b,c > 0$, $a+b+c = 1/4$, and
$\theta$ in the circle. Here $a$, $b$, and $c$ are the lengths of the
three intervals on the boundary of the central gap, and $\theta$ is
the angle to the start of one leaf, so that the three leaves have endpoints
\[
(\theta, \theta + 1/4),\quad(\theta+1/4+a, \theta+1/2+a),
  \quad\text{and}\quad(\theta-1/4-c, \theta-c).
\]
(The last leaf is also $(\theta+1/2+a+b,\theta+3/4+a+b)$.)
Again, there is an equivalence relation:
\[
(\theta,a,b,c) \equiv (\theta+1/4+a,b,c,a) \equiv (\theta+1/2+a+b,c,a,b).
\]
This stratum is therefore topologically a triangle cross an interval,
with the top glued to the bottom by a $1/3$ twist. As a manifold with
corners, it has only one codimension-1 face, an annulus.

We next turn to the codimension-1 degeneracy locus. Here there is only
one topological type (the stratum is connected), consisting of a
triangle and another leaf. As illustrated in
Figure~\ref{fig:pm4-codim1}, the major can be uniquely parametrized by
an angle~$\theta$, the angle to the vertex of the triangle opposite
the leaf, and a number $a \in (0,1/4)$, the length of one of the
intervals on the boundary of the gap between the triangle and the
leaf. The codimension-1 degeneracy locus is therefore an annulus.

There are three ways to perturb a codimension-1 degenerate major into
generic majors. Note that two of the generic majors are in the
parallel stratum and one is in the triangle stratum. Thus, all three
annuli that we found on the boundary of the top-dimensional strata are
glued together.

\begin{figure}
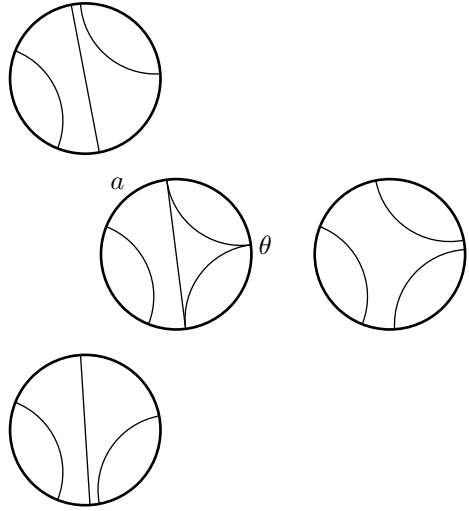

  \[
  \begin{tikzpicture}
    \node at (0,0) {\includegraphics{majors-10}};
    \node at (0:2.7cm) {\includegraphics{majors-12}};
    \node at (120:2.7cm) {\includegraphics{majors-13}};
    \node at (240:2.7cm) {\includegraphics{majors-11}};
  \end{tikzpicture}
  \]
  \caption{A codimension-1 primitive major in $\PM(4)$ (in the center)
    and three perturbations to generic majors.}
  \label{fig:pm4-codim1}
\end{figure}



\section{Thurston's entropy algorithm and entropy on the Hubbard tree}

The purpose of this section is to define the notions in the following diagram and establish the equality on the right column (Theorem \ref{Thurston-algorithm}):

 \[
 \begin{tikzpicture}
   \matrix[row sep=0.8cm,column sep=4cm] {
     \node (Gammai) {$ \text{rational }\theta$}; &
       \node (Gamma) {$\log\rho(A_{\theta})$}; \\
     \node (S2i) {$f_{\theta}$}; &
       \node (S2) {$h(\mathcal{H}_{\theta}, f_{\theta})$.}; \\
   };
   \draw[double equal sign distance] (Gamma) to node[auto=left,cdlabel] {\text{ equality by Thm.\ \ref{Thurston-algorithm}} } (S2);
   \draw[->] (S2i) to node[auto=right,cdlabel] {\text{entropy on Hubbard tree}} (S2);
   \draw[->] (Gammai) to node[auto=left,cdlabel] {\text{Thurston's  entropy algorithm}} (Gamma);
   \draw[->] (Gammai) to node[auto=right,cdlabel] {(DH)} (S2i);
 \end{tikzpicture}
 \]

More precisely, to any rational angle $\theta$ (mod 1), Douady-Hubbard \cite{DH}  associated a unique postcritically finite
quadratic polynomial $ f_\theta: z\mapsto z^2+c_\theta$. This polynomial induces a Markov action on its Hubbard tree.
The topological entropy of the polynomial $f_\theta$ on its Hubbard tree is called the \emph{core entropy} of $f_\theta$.

 In order to combinatorially encode  and effectively compute  the core entropy, W. Thurston developed an algorithm that takes $\theta$ as its
 input, constructs a non-negative matrix $A_\theta$ (bypassing $f_{\theta}$), and outputs its Perron-Frobenius leading eigenvalue $\rho(A_\theta)$.
 We will prove that the logarithm of this eigenvalue is  the core entropy of the quadratic polynomial $f_\theta$.

\subsection{Thurston's entropy algorithm }\label{algorithm2}

Set $ \mathbb{S}=\R/\Z$.   All angles in this section are considered to be  mod 1, i.e., elements of $\mathbb{S}$.

Let $\tau: \mathbb{S} \to \mathbb{S}$ denote the angle doubling map. An  angle $\theta$ is \emph{periodic} under the action of $\tau$ if and only if it is rational with odd denominator,
and  (strictly) \emph{preperiodic} if and only if it is rational with even denominator.

Fix  a rational angle $\theta \in \mathbb{S} \setminus \{0\}$.
 If $\theta$ is periodic, exactly one of $(\theta+1)/2$ and $\theta /2$ is periodic and the other is preperiodic. If $\theta$ is preperiodic, both $(\theta+1)/2$ and $\theta/2$ are preperiodic.
Set $2^{-1}\theta$ to be the periodic angle, if exists, among $(\theta+1)/2$ and $\theta /2$, and otherwise set it to be $\theta /2$.  Define the set
 \[O_\theta:=\Bigl\{\{2^n \theta, 2^l \theta\}\Bigm|l, n\geq  -1\text{ and }2^n\theta\not=2^l\theta\Bigr\},\]
  with the convention that the 
pairs $\{2^n\theta,2^l\theta\}$ that constitute $O_\theta$ are
unordered sets.
We divide  the   circle  $\mathbb{S}$ at the points
$\big\{{\theta}/2, (\theta+1)/2\big\}$, forming  two closed half circles, with the boundary points belonging to both
halves.

Define $\Sigma_\theta$ to be the abstract linear space over $\R$ generated by the elements of $O_\theta$.
Define a linear map $\mathcal{A}_\theta: \Sigma_\theta \to \Sigma_\theta$ as follows. For any basis vector $\{a,b\}\in O_\theta$, if $a$ and $b $ are in a common  closed half-circle, set $\mathcal{A}_\theta(\{a,b\} )=\{2a,2b\}$; otherwise set $\mathcal{A}_\theta(\{a,b\} )=\{2a,  \theta\} +  \{\theta, 2b\}$.
 Denote by $A_\theta$ the matrix of $\mathcal{A}_\theta$ in the basis  $O_\theta$; it is a non-negative matrix. Denote its leading eigenvalue, which exists by the Perron-Frobenius theorem, by  $\rho(A_\theta)$. It is easy to see that $A_\theta$ is not nilpotent, so $\rho(A_\theta)\ge1$.

\begin{definition}\label{alto} \emph{Thurston's entropy algorithm} is the map
$$(\mathbb{Q} \cap \mathbb{S} \setminus \{0\}) \owns \theta  \mapsto \log \rho(A_\theta).$$
 \end{definition}
 
\noindent We will relate the output of Thurston's entropy algorithm, $ \log \rho(A_\theta)$, to quadratic polynomials  in subsection \ref{ss:RelatingAlgorithmToPolynomials}.

\begin{example}
Set $\theta=\frac{1}{5}$. The abstract linear space $\Sigma_\theta$ has basis
\[O_{\frac{1}{5}}=
\bigl\{\{\tfrac{1}{5},\tfrac{2}{5}\},\{\tfrac{1}{5},\tfrac{3}{5}\},\mystrut\{\tfrac{1}{5},\tfrac{4}{5}\},\{\tfrac{2}{5},\tfrac{3}{5}\},\{\tfrac{2}{5},\tfrac{4}{5}\},\{\tfrac{3}{5},\tfrac{4}{5}\}\bigr\}.
\]
We divide the circle $\mathbb{S}$ by the pair $\{\frac1 {10}, \frac 35\}$. The linear map $\mathcal{A}_{\frac{1}{5}}$ acts on the basis vectors as follows:
\begin{align*}
  \Big\{\dfrac{1}{5},\dfrac{2}{5}\Big\}&\mapsto\Big\{\dfrac{2}{5},\dfrac{4}{5}\Big\},
    &\Big\{\dfrac{1}{5},\dfrac{3}{5}\Big\}&\mapsto\Big\{\dfrac{2}{5},\dfrac{1}{5}\Big\},
      &\Big\{\dfrac{1}{5},\dfrac{4}{5}\Big\}&\mapsto\Big\{\dfrac{1}{5},\dfrac{2}{5}\Big\}+\Big\{\dfrac{1}{5},\dfrac{3}{5}\Big\},\\
  \Big\{\dfrac{2}{5},\dfrac{3}{5}\Big\}&\mapsto\Big\{\dfrac{4}{5},\dfrac{1}{5}\Big\},
    &\Big\{\dfrac{2}{5},\dfrac{4}{5}\Big\}&\mapsto\Big\{\dfrac{4}{5},\dfrac{1}{5}\Big\}+\Big\{\dfrac{1}{5},\dfrac{3}{5}\Big\},
      & \Big\{\dfrac{3}{5},\dfrac{4}{5}\Big\}&\mapsto\Big\{\dfrac{1}{5},\dfrac{3}{5}\Big\}.
\end{align*}
We then compute $\log\rho(A_\theta)=0.3331$.
\end{example}

\begin{figure}[htpb]
\centering
\includegraphics[width=3.5in]{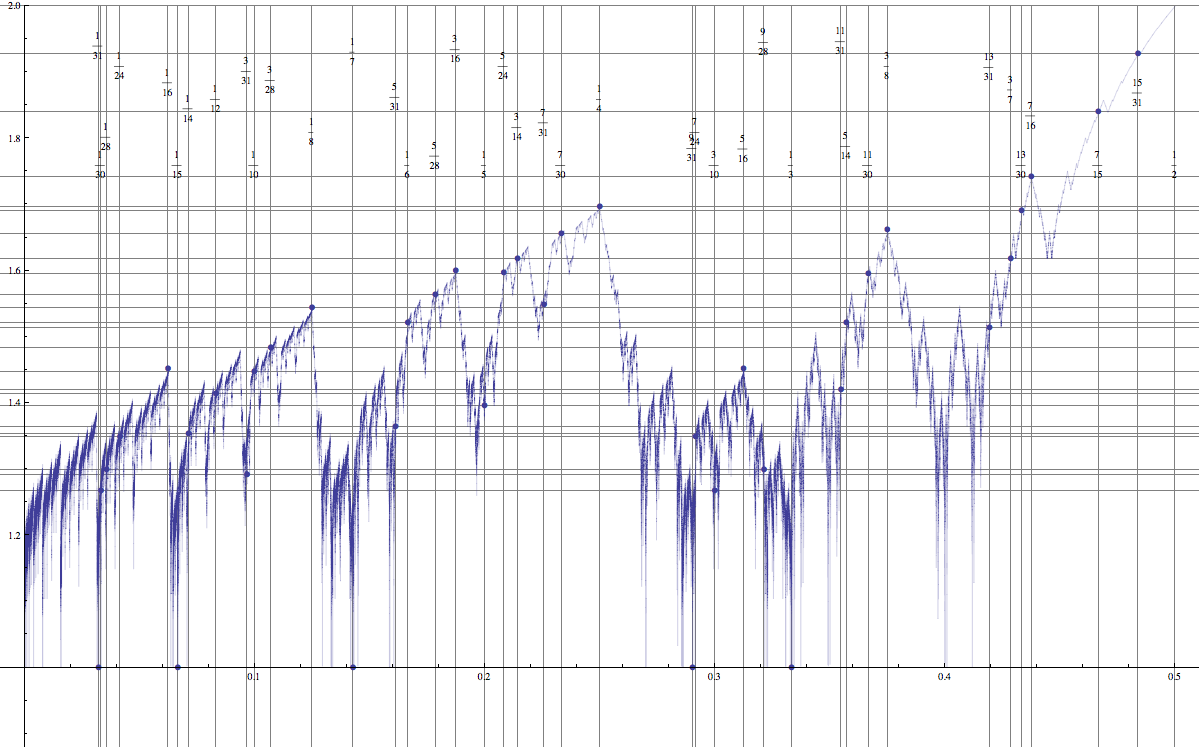}
\caption{W. Thurston's first plot of core entropy. The horizontal axis is the $\theta$-axis, for rational $\theta \in [0,1/2]$ (half of the unit circle), and the vertical axis is core entropy, or $\log \rho(A_{\theta})$.
 } \label{fig:Bills-plot}
\end{figure}

\subsection{Hubbard trees following Douady and Hubbard}
This subsection presents background material about Hubbard trees, which will be used in later sections to justify Thurston's entropy algorithm as computing core entropy.  (See, for example, \cite{DH,Poi2} for additional information about Hubbard trees.) 

Let $f$ be a postcritically finite polynomial, i.e., a polynomial all
of whose critical points have a finite (and hence periodic or preperiodic) orbit under $f$.
By classical results of Fatou, Julia, Douady and Hubbard,  the \emph{filled Julia  set} 
$$\mathcal{K}_f=\{z\in \C\mid f^n(z)\not\to \infty\}$$
  is  compact, connected, locally connected  and locally arc-connected. These conditions also hold for the \emph{Julia set} $\mathcal{J}_f:=\partial \mathcal{K}_f$.
The Fatou set $\mathcal{F}_f :=\cbar \setminus \mathcal{J}_f$ consists of one component $U(\infty)$ which is the basin of attraction of $\infty$, and at most countably many bounded components constituting the interior of $\mathcal{K}_f$. Each of the sets $\mathcal{K}_f, \mathcal{J}_f, \mathcal{F}_f$ and $U(\infty)$ is fully invariant by $f$; each Fatou component is (pre)periodic (by Sullivan's non-wandering domain theorem, or by hyperbolicity of the map); and each periodic cycle of Fatou components contains at least one critical point of $f$ (counting $\infty$).

There is a system of Riemann mappings $$\Big\{\phi_U: \D\to U\Bigm| U \text{ Fatou component}\Big\}$$
each extending to a continuous map on the closure $\overline{\D}$, so that  the following diagram commutes for all $U$:
\[
 \begin{tikzpicture}
   \matrix[row sep=0.8cm,column sep=2.4cm] {
     \node (Gammai) {$ \overline \D $}; &
       \node (Gamma) {$\overline \D$}; \\
     \node (S2i) {$\overline U$}; &
       \node (S2) {$\overline{f(U)}$.}; \\
   };
   \draw[->] (Gamma) to node[auto=left,cdlabel] {\phi_{f(U)}} (S2);
   \draw[->] (S2i) to node[auto=right,cdlabel] {f} (S2);
   \draw[->] (Gammai) to node[auto=left,cdlabel] {z \mapsto z^{d_U}} (Gamma);
   \draw[->] (Gammai) to node[auto=right,cdlabel] {\phi_U} (S2i);
 \end{tikzpicture}
 \]

\noindent In particular, on every periodic Fatou component $U$, including $U(\infty)$, the map $\phi_U$ realizes a conjugacy between a power map and the first return map on
$U$. The image in $U$ under $\phi_U$ of radial lines in $\D$ are, by
definition, \emph{internal rays} on $U$ if $U$ is bounded and
\emph{external rays} if $U=U(\infty)$.
Since a power map sends a radial line to a radial line, the polynomial $f$ sends an internal/external ray to an internal/external ray.

 If $U$ is a bounded Fatou component, then  $\phi_U:\overline \D\to \overline U$ is a homeomorphism, and thus every boundary point of $U$ receives exactly one internal ray from $U$. This is in general not true for $U(\infty)$, where several external rays may land at a common boundary point.

\begin{definition}[supporting rays]\label{def:support-ray} We say that
  an external ray $R$ \emph{supports} a bounded Fatou component~$U$
 if
 \begin{enumerate}
 \item the ray lands at a boundary point $q$ of $U$, and
 \item there is a sector based at $q$ delimited by $R$ and the
   internal ray of $U$ landing at $q$ which does not
   contain other external rays landing at~$q$.
 \end{enumerate}
\end{definition}

\begin{figure}[htpb]
\centering
\includegraphics[width=3.5in]{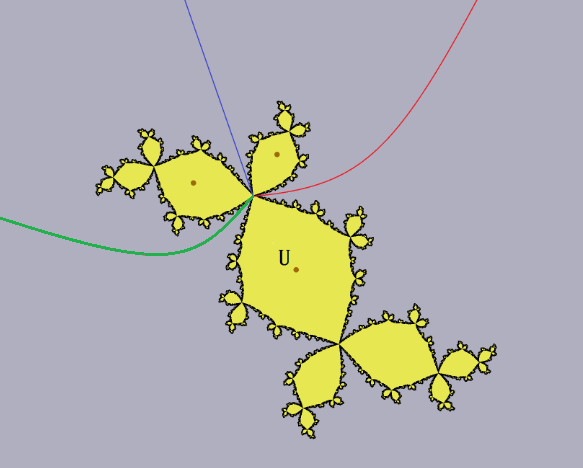}
\caption{The green (left) and red (right) rays are supporting rays of the Fatou component $U$ , but the blue (middle) one is not.} \label{fig:supporting-ray}
\end{figure}
It follows from  Definition \ref{def:support-ray} that for any bounded Fatou component $U$ and point $z\in\partial U$ there are at most two external rays which support $U$ and land at $z$.
Start from the internal
ray in $U$ which lands at  $z$ and turn in the counterclockwise
direction centered at $z$.
The first (resp.\ last) encountered external ray landing at $z$  is
called the \emph{right supporting ray} (resp.\ \emph{left supporting
  ray}) of $U$ at $z$.
 See Figure \ref{fig:supporting-ray}.

The system of internal/external rays does not depend on the possible choices of $\phi_U$. If $f: z\mapsto z^2+c$
and $f$ is postcritically finite, there is actually a unique choice of $\phi_U$ for each Fatou component $U$. In particular,   $\phi_{U(\infty)}$ conjugates $z^2$ to $f$ and $\phi_{U}$ conjugates $z^2$ to $f^p$ if $U$ is a  bounded periodic
Fatou component and $p$ is the minimal integer such that $f^p(0)=0$ (if no such $p$ exists, $\mathcal{K}_f=\mathcal{J}_f$).
In this case, for any $x\in \mathbb{S}$ we use $\mathcal{R}_f(x)$ or simply $\mathcal{R}(x)$ to denote the image under $\phi_{U(\infty)}$ of the ray $\{re^{2\pi i x}, 0<r<1\}$ and will call it  the \emph{external ray of angle $x$}. Angles of internal rays can be defined similarly. We also use $\g(x)=\phi_{U(\infty)}(e^{2\pi i x})$ to denote the \emph{landing point} of the ray $\mathcal{R}(x)$.

Any pair of points in the closure of a bounded Fatou component  can be joined in a unique way by a Jordan arc consisting of (at most two) segments of internal rays. We call such arcs \emph{regulated} (following Douady and Hubbard).
Since $\mathcal{K}_f$ is arc-connected, given two points $z_1, z_2\in \mathcal{K}_f$, there is an arc $\gamma: [0,1]\to \mathcal{K}_f$ such that $\gamma(0)=z_1$ and $\gamma(1)=z_2$. In general, we will not distinguish  between the map  $\g$ and its image. It is proved in \cite{DH} that such arcs can be chosen in a unique way so that  the intersection with the closure of a Fatou component is regulated. We still call such arcs regulated and denote them by $[z_1,z_2]$.
We say that a subset $X\subset \mathcal{K}_f$ is \emph{allowably
  connected} if for every $z_1,z_2\in X$ we have $[z_1,z_2]\subset
X$.

\begin{definition}
 We define the \emph{regulated hull}
of  a subset $X$ of $ \mathcal{K}_f$ to be  the minimal closed allowably connected subset of $\mathcal{K}_f$ containing $X$.
\end{definition}

\begin{proposition}
For a collection of $z_1,\ldots,z_n$ finitely many points in $\mathcal{K}_f$, their regulated hull   is a finite tree.
\end{proposition}

\begin{definition}
Let $f$ be a  postcritically finite polynomial. The \emph{postcritical set}
$\mathcal{P}_f$ is defined to be%
\footnote{For the purpose of this part, we include critical points in the postcritical set.}
$$ \mathcal{P}_f = \left\{f^n(c)\Bigm| f'(c)=0, n\ge 0\right\}.$$
The \emph{Hubbard tree} $\mathcal{H}_f$ is defined to be the regulated hull of the finite set $\mathcal{P}_f$.\end{definition}

The \emph{vertex set} $V(\mathcal{H}_f)$ of $\mathcal{H}_f$ is the union of $ \mathcal{P}_f$ together with the branching points of $\mathcal{H}_f$, namely the points  $p$ such that $\mathcal{H}_f\setminus\{p\}$ has  at least three connected components. The closure of a connected component of $\mathcal{H}_f\setminus V(\mathcal{H}_f)$ is called an \emph{edge}.
\begin{lemma} For a postcritically finite  polynomial $f$,
the set $\mathcal{H}_f$ is a tree with finitely many edges. Moreover
$f(\mathcal{H}_f)\subset \mathcal{H}_f$  and
$f:\mathcal{H}_f\to \mathcal{H}_f$ is a Markov map
(as defined in Appendix A).
\end{lemma}

Using Proposition \ref{entropy-formula}, we may relate the topological entropy of $f$ on $\mathcal{H}_f$ to the spectral radius of a transition matrix $D_f$ constructed from $f$ by  $h(\mathcal{H}_f,f)=\log\rho(D_f)$.

\subsection{Relating Thurston's entropy algorithm to polynomials} \label{ss:RelatingAlgorithmToPolynomials}

Thurston's entropy algorithm effectively computes the topological  entropy $h(\mathcal{H}_{f}, f)$ for any postcritically finite polynomial without actually computing the Hubbard tree. We will see how to relate
the quadratic version of the algorithm given in Section \ref{algorithm2} to quadratic polynomials.

 On one  hand,  Thurston's entropy algorithm  produces a quantity, $\log \rho(A_\theta)$,  from any given rational angle $\theta$. On the other hand,
Douady-Hubbard defined a finite-to-one map $\mathbb{Q}\ni \theta\mapsto c_\theta$  so that the quadratic polynomial  $ z\mapsto z^2+{c_\theta}$ is
postcritically finite.
More precisely:

\begin{theorem}[Douady-Hubbard]\label{notation1} If $\theta\in \mathbb{Q}$ is  preperiodic (resp.\ $p$-periodic) under the angle doubling map, there is a unique parameter $c_\theta$   such that for $f: z\mapsto z^2+c_\theta$ both  external rays $\mathcal{R}(\frac{\theta}{2})$ and $\mathcal{R}(\frac{\theta+1}{2})$ land at $0$, and $0$ is preperiodic (resp.\ support the Fatou component  containing $0$, and $0$ is $p$-periodic). Furthermore,  every postcritically finite quadratic polynomial arises in this way.
\end{theorem}

\noindent Our objective is to establish:

\begin{theorem}\label{Thurston-algorithm}
For $\theta$ a rational angle,  $\log\rho(A_\theta)=h(\mathcal{H}_f, f )$ for $f:z\mapsto z^2+{c_\theta}$.
\end{theorem}

\begin{proof}
 Fix a rational angle $\theta$ and consider $f:z\mapsto z^2+{c_\theta}$. The idea of the proof is the following, inspired by \cite{Gao,Jung}:
\begin{enumerate}
\item Construct a topological graph $G$ and a Markov action $L:G\to G$   such that the spectral radius
of the transition matrix $D_{(G,L)}$ is equal to the spectral radius of $A_\theta$.
\item Construct a continuous, finite-to-one, and surjective semi-conjugacy $\Phi$ from $L:G\to G$ to $ f_{}:\mathcal{H}_f\to \mathcal{H}_f$.
\end{enumerate}
Then we may conclude that
\[\log\rho(A_\theta)\overset{\text{same matrix}}=\log\rho(D_{(G,L)})\overset{\text{Prop. \ref{entropy-formula}}}=h(G,L)\overset{\text{Prop. \ref{Do4}}}=h(\mathcal{H}_f, f_{}).\]

Let $G$  be a topological complete  graph whose
vertex set is the forward orbit  $\frac{\theta}2$ and $ \frac{\theta+1}{2}$ after identifying the diagonal angles: 
$$V_{G}=\left\{2^n\theta, n\ge -1\right\}/{\left (\frac{\theta}2\sim \frac{\theta+1}{2} \right)}.$$
 The set of edges is denoted by
 $$E_{G}=\left\{\,e(x,y)|\ x\not=y\in V_{G}\right\}$$ 
 (with $e(x,y)=e(y,x)$). Being a topological graph means that $G$ is a topological space and each edge $e(x,y)$
 is homeomorphic to a closed interval with ends $x$ and $y$.
Clearly $E_G$ is in bijection with $O_\theta$.

Mimicking the action of the linear map $\mathcal{A}_\theta$, we can define a piecewise monotone map $L:G\to G$ as follows.
Let $x,y$ be two distinct vertices in $V_{G}$. If $x,y$ belong to the same closed half circle, i.e., the closure of a complete component of $\frac{\theta}{2}, \frac{\theta+1}{2}$, then $\tau(x)\ne \tau(y)\in V_{G}$. In this case, let $L$ map the edge $e(x,y)$ homeomorphically onto the edge $e(\tau (x),\tau (y))$. If $x,y$ belong to distinct open half circles, subdivide the edge $e(x,y)$ into non-trivial arcs $e(x,z)$ and $e(z,y)$, and let $L$ map the arc $e(x,z)$ (resp.\  $e(z,y)$)  homeomorphically onto $e(2x,\theta)$ (resp.\ $e(\theta,2y))$.

It is easy to see that the transition matrix of $(G,L)$ is exactly $A_\theta$. By Proposition \ref{entropy-formula}, the topological entropy  $h(G,L)$ is always equal to $\log\rho(A_\theta)$, regardless  the
precise choices of $L$ as  homeomorphisms on the edges.

However, to relate $h(G,L)$ to $h(\mathcal{H}_f,f)$, we will redefine the homeomorphic action of $L$ on each edge by lifting corresponding actions of $ f_{}$ via suitably
defined conjugacies.

We will treat the periodic and  preperiodic cases separately.

\begin{itemize}
\item $\theta$ is periodic.

Then the rays of angles $\frac{\theta}{2}$ and $\frac{\theta+1}{2}$ both land at the boundary of the Fatou component $U$ containing $0$ and support the component.
If $\frac{\theta}{2}$ is periodic then the rays are right supporting rays, and if it is  $\frac{\theta+1}{2}$ that is periodic then
the rays are left supporting rays. For simplicity we will only treat the former case and thus right supporting rays.
\begin{figure}[htpb]
\centering
\includegraphics[width=3.8in]{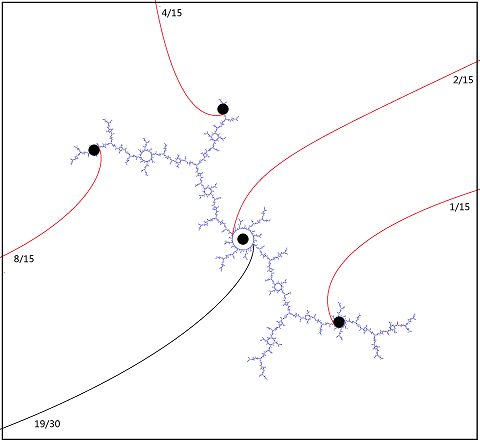}
\caption{In this example, $\theta=4/15$ and $\theta/2=2/15$ is periodic. All the rays in this figure are right-supporting rays.} \label{fig:supporting-ray-2}
\end{figure}

The angles in $V_{G}$ form a periodic cycle. For each $x\in V_{G}$, the external ray $\mathcal{R}(x)$ of angle $x$  may support one or two Fatou components, but it right-supports a unique periodic Fatou component, denoted by $U_x$. Define a map $\Phi:V_{G}\to \mathcal{P}_f$ such that $\Phi(x)$ is the center of the Fatou component $U_x$. Since   external rays with distinct angles of $V_{G}$ right-support  distinct Fatou components, we have $U_x\not=U_y$ if $x\not=y\in V_{G}$. Thus the map $\Phi:V_{G}\to \mathcal{P}_f$ is bijective.

 Extend $\Phi$ to a map, also denoted by $\Phi$,  from $G$ to $\mathcal{H}_f$ such that $\Phi$ maps the edge $e(x,y)$ homeomorphically to the regulated arc $[\Phi(x),\Phi(y)]\subset \mathcal{H}_f$. We now assert and justify several facts about $\Phi$:
\begin{enumerate}
\item $\Phi$ is finite-to-one.

It follows directly from the fact that $\Phi:V_{G}\to \mathcal{P}_f$ is a bijection.

\item $\Phi$ is surjective.

Since $\Phi(V_{G})=\mathcal{P}_f$ and $G$ is a complete graph, for any $p,q\in \mathcal{P}_f$, $[p,q]\subset \Phi(G)$. So we only need to evoke the fact that each edge of $\mathcal{H}_f$ is contained in a regulated arc $[p,q]$ with $p,q\in \mathcal{P}_f$.

\item $ f_{}\circ\Phi=\Phi\circ L$ after suitable modification of $ L$ on each edge.

If two distinct vertices $x,y\in V_{G}$ belong to the same closed half circle, the interior of the regulated arc $[\Phi(x),\Phi(y)]$ does not contain the critical point of $ f_{}$. Its $f$-image is
\begin{align*} f_{}([\Phi(x),\Phi(y)])& =[ f_{}(\Phi(x)), f_{}(\Phi(y))]=[\Phi(2x),\Phi(2y)]\\
&=
\Phi(e(2x,2y))=\Phi(L(e(x,y)). \end{align*}
So we can redefine $L$ on the edge by lifting $ f_{}$, i.e. by setting   
$$ L=\Phi^{-1}\circ f_{}\circ\Phi$$ on the edge of $e(x,y)$.

If two vertices $x,y\in V_{G}$ belong to distinct open half circles, the interior of the regulated arc $[\Phi(x),\Phi(y)]$ contains the critical point $0$. Its $f$-image is
\begin{align*}
  f_{}( [\Phi(x),\Phi(y) ])&= [f_{}(\Phi(x)),c_\theta]\cup\ [c_\theta, f_{}(\Phi(y)) ] \\ 
   &= [ \Phi(2x),\Phi(\theta)\ ]\cup\ [\Phi(\theta),\Phi(2y)]\\
   &= \Phi(e(2x,\theta)\cup e(\theta,2y))= \Phi(L(e(x,y)).
\end{align*}
So we can redefine  $L$ on  $e(x,y)$ such that  $\Phi\circ L= f_{}\circ\Phi$ on each of the
two segments of $e(x,y)$ subdivided by $\Phi^{-1}(0)$.
\end{enumerate}
Thus the maps $L,\Phi$ and $ f_{}$ have been shown to satisfy the properties of Proposition \ref{Do4}, so the equation
$$h(G,L)=h(\mathcal{H}_f, f_{})$$ holds.

\item $\theta$ is preperiodic.

In this case, the filled Julia set is equal to the Julia set. We can also define a map $\Phi:V_{G}\to \mathcal{P}_f$ such that $\Phi(x)$ is the landing point of $\mathcal{R}(x)$. It is easy to see that $\Phi$ is surjective. However, if we extend $\Phi$ piecewise monotonically on the edge of $G$ as what we did in the periodic case, the map
 $\Phi$ will lose the property of being finite-to-one, because some rays with distinct angles in $V_{G}$ may land at the same point, which means that $\Phi$ collapses some edges of $G$ to points. So we may no longer apply Proposition \ref{Do4} directly.   To overcome this difficulty, let us define a subgraph $\G$ of  $G$ as follows:
 \[V_{\G}:=V_{G}\quad\text{and}\quad E_{\G}:=\{e(x,y)\in E_{G}|\ \Phi(x)\not=\Phi(y)\}\]
 We will check that the graph $ \G$ satisfies the following properties.
\begin{enumerate}
\item  $ \G$ is connected.

This is because  for any $x\in V_{\G}$, the edge $e(\theta/2,x)$ belongs to $\G$.

\item $\G$ is $L$-invariant.

First, we observe that for any two distinct vertices $x,y\in V_G$ belonging to the same closed half circle, if $\mathcal{R}(x)$ and $\mathcal{R}(y)$ land at distinct points, then $\mathcal{R}(2x)$ and $\mathcal{R}(2y)$ also land at distinct points.

Let $e(x,y)$ be an edge of $\G$. Then the rays $\mathcal{R}(x)$ and $\mathcal{R}(y)$ land at distinct points $\Phi(x)$ and $\Phi(y)$ respectively. If $x,y$ belong to the same  closed half circle,  by this observation  the image edge 
$$L(e(x,y))=e(2x,2y)$$ belongs to $\G$. If $x,y$ belong to distinct open half circles, then $x$ (resp.\ $y$), $\theta/2$ belong to the same closed half circle  and the image 
$$L(e(x,y))=e(2x,\theta)\cup e(\theta,2y)$$ belongs to $\G$.

\item $h(G,L)=h(\G,L)$.

We claim that the set $G\setminus \G$ is $L$-invariant. In fact, an edge $e(x,y)$ belongs to $ E_{G}\setminus E_{\G}\overset{\text{definition}}{\iff}\Phi(x)=\Phi(y)\iff $ the rays $\mathcal{R}(x),\mathcal{R}(y)$ land at a common periodic point in $\mathcal{P}_f$. To see the last implication,  we only need to show the ``$\Rightarrow$''
part. Since $x,y$ are both in the forward orbit of $\theta$, we may assume $y=2^k x,k\geq1$. Then
$$\Phi(y)=\Phi(2^kx)=f^k(\Phi(x)). $$ So $\Phi(x)$ is a periodic point, receives both rays
$\mathcal{R}(x),\mathcal{R}(y)$,  and belongs to $\mathcal{P}_f$.
 In this case the angles $x,y$ must belong to the same half circle and $\Phi(2x)=\Phi(2y)$. It follows that 
 $$L(e(x,y))=e(2x,2y)$$ also belongs to $E_{G}\setminus E_{\G}$.

 The argument above also shows that $L$ maps an edge in $E_{G}\setminus E_{\G}$ homeomorphically onto an edge in $E_{G}\setminus E_{\G}$. According to the definition of the topological entropy, we have
\[h(G\setminus \G, L)=0.\]
So by Proposition \ref{Do2}, we obtain
\[h(G,L)=\max\Big\{h(\G,L),h(G\setminus \G, L)\Big\}=h(\G,L)\]
\end{enumerate}

By these properties, it is enough to prove that $h(\G,L)=h(\mathcal{H}_f, f_{})$. In this case, the map $\Phi\big|_{\G}:\G\to \mathcal{H}_f$ is finite-to-one. By the same argument as in the periodic case, we also obtain that $\Phi\big|_{\G}$ is surjective and 
$$ f_{}\circ\Phi=\Phi\circ L$$ on $\G$. So the equality 
$$h(\G,L)=h(\mathcal{H}_f, f_{})$$ holds.
\end{itemize}
   \end{proof}
Thurston's entropy algorithm and Theorem  \ref{Thurston-algorithm} are generalized to higher degree maps in \cite{Gao2}.

\section{Combinatorial laminations and polynomial laminations}

The action of a degree $d$ postcritically finite polynomial on its Julia set can be combinatorially encoded by
the action of a degree $d$ expanding map on an invariant lamination on the circle $\mathbb{S}=\R/\Z$ or on the
torus $\mathbb{T}=\mathbb{S}\times \mathbb{S}$. Here we illustrate this connection.

\subsection{Combinatorial laminations}

We denote the diagonal $\Delta$  of $\mathbb{T}$ by $$\De=\{(x,x), x\in \mathbb{S}\}.$$

For $x\ne y\in \mathbb{S}$, we use $\overline{xy}$ to denote the closure in $\overline \D$ of the hyperbolic chord in $\D$ connecting
$e^{2\pi ix}$ and $e^{2\pi iy}$, which is called a \emph{leaf}. This leaf is represented (twice) on the torus $\mathbb{T}\setminus\De$ as $(x,y)$ and $(y,x)$.

Two  leaves $\overline{xy}$ and  $\overline{x'y'}$ are said to be \emph{compatible} if   they are either equal or do not
cross inside $\D$. We say also that the two points $(x,y)$ and $(x',y')$ on the torus are compatible.
In this case the four points $(x,y), (y,x), (x',y'), (y',x')$ are pairwise compatible.

A \emph{lamination in $\D$} is a collection of leaves which are pairwise compatible and whose union is closed in $\D$. A \emph{lamination on the torus} is a subset $L$ closed in $\mathbb{T}\setminus\De$ and symmetric with respect to the diagonal $\De$
so that points in $L$ are pairwise compatible.

\subsection{Polynomial laminations}
Let $f$ be a monic  degree $d$ polynomial with a connected and locally connected Julia set. There is a unique Riemann mapping $\phi :\cbar\setminus \overline \D\to \cbar\setminus \mathcal{K}_f$ tangent to the identity at $\infty$; it extends continuously to the closure by the Carath\'{e}odory theorem. External rays are parametrized by external angles. The set of pairs of  external rays   landing
at a common point gives another combinatorial characterization of the polynomial dynamics.

Two distinct external rays $\mathcal{R}(x)$ and $\mathcal{R}(y)$ landing at the same point  are said to be a \emph{ray-pair of $f$}. We also say that $\{x,y\}$ is an angle-pair of $f$. If the rays land at $z$, we says that $\{x,y\}$ is an angle-pair at $z$.

 Let $\{x,y\}$ be an angle-pair  at $z$. Then the union of the two rays $\mathcal{R}(x)$ and $\mathcal{R}(y)$ together with $\{z\}$ divides the plane into two regions. If at least one of the two regions
does not contain other rays landing at $z$, we say that $\{x,y\}$ is an \emph{adjacent angle-pair}.

 A polynomial $f$  induces a  lamination $L(f)$ whose leaves (i.e. points in $\mathbb{T}$) consist of geodesics connecting   adjacent angle-pairs. $L(f)$ is said to be the \emph{polynomial lamination} of $f$.

\subsection{Good and excluded regions of a lamination}

We recall here some notions introduced in Section 4.
If a lamination on $\mathbb{T}$ contains a point $l = (x,y)$, then a certain set $X(l)$ of other points are excluded from the lamination because they are not compatible with  $(x,y)$.  If you draw the horizontal and vertical circles through
the two points $(x,y)$ and $(y,x)$, they divide the torus into four rectangles having the same vertex set; the remaining two common vertices are $(x,x)$ and $(y,y)$. The two rectangles  bisected by the diagonal are actually squares,  of side lengths $a-b \mod 1$
and $b-a \mod 1$. Together, they form the compatible region $G(l)$. The
leaves represented by points in the interior of the remaining two
rectangles constitute the excluded region $X(l)$. See
Figure~\ref{fig:ExcludedLeaves}, where the blue region is $X(l)$ and
the tan region is $G(l)$.

Given a  set  $S$ of leaves, the \emph{excluded region} $X(S)$ is the union of the excluded regions $X(l)$ for $l \in S$, and the \emph{good region} $G(S)$  is the intersection of the good regions $G(l)$ for $l \in S$.  If $S$ is a finite lamination, then $G(S)$ is a finite union of closed rectangles that are disjoint except for corners.

\subsection{Invariant laminations from majors}

Define the map  $F: \mathbb{T}\to \mathbb{T}$ by 
$$(x,y)\mapsto (\tau(x), \tau(y)),$$ where $\tau$ is the angle-doubling map defined in $\S$ 1.1.  A lamination $L\subset \mathbb{T}\setminus\De$ is said to be \emph{$F$-invariant} if, for every $(x,y)\in L $,
\begin{enumerate}
\item  either $F(x,y)\in\De$ or $F(x,y)\in L$;
\item there exist two preimages of $(x,y)$ in $L$, which have different $x$ components and different $y$ components.
\end{enumerate}

The polynomial lamination $L(f)$ for a postcritically finite quadratic polynomial $f:z\mapsto z^2+c$ is an example of an $F$-invariant lamination.

Given an angle $\theta\in \mathbb{S}$, there are several more or less natural ways to define an $F$-invariant lamination, mimicking the polynomial lamination of $ f:z\mapsto z^2+c_\theta$. And as we shall see, they differ by at most a countable set. Here we choose one that mimics the fact that preimages of the critical point of $f$ accumulates on every Julia point. Other definitions will be given later (see Section \ref{sec:entropy}).

Set, inductively,
 \begin{align*} b_0 &:= \Big\{\Big(\frac{\theta}{2},\frac{\theta+1}{2}\Big),\Big(\frac{\theta+1}{2}, \frac{\theta}{2}\Big) \Big\}\quad \text{the major leaves } \\
b_{i+1} &:=\Big(F^{-1}(b_i )\cap G(b_i )\Big)\cup b_i  \\
\text{\rm Pre}_\theta&:= \bigcup_{i\geq0}b_i\quad  \text{the set of pre-major leaves}\\
\text{\rm cluster}(\text{\rm Pre}_\theta)&:=\text{the set of cluster points in $\mathbb{T}$ of the pre-major leaves}.\\
L_\theta&:=\text{the  set  of cluster points in $\mathbb{T}\setminus \De$ of the pre-major leaves.}
\end{align*}

\subsection{Relating combinatorial laminations to polynomial laminations}

\begin{proposition}\label{polynomial-lamination}
For a rational angle $\theta$, the set $L_\theta$ is equal to the polynomial lamination of $ f:z\to z^2+c_\theta $. In particular,
 $L_\theta$ is an $F$-invariant lamination.
\end{proposition}

\begin{proof} This says that every accumulation point $(x,y)$, $x\ne
  y$, of $\text{\rm Pre}_\theta$ is an adjacent angle-pair for $
  f:z\mapsto z^2+c_\theta$. And conversely
every adjacent angle-pair is accumulated by a sequence $(x_n,y_n)\in \text{\rm Pre}_\theta$.

If  $\theta$ is preperiodic, the ray-pair of angles $\frac\theta 2$ and $\frac{\theta+1}2$ land at the critical point $0$. So  each point in $\text{\rm Pre}_\theta$ is a  (not necessarily adjacent) angle-pair at a pre-critical point.
If a sequence of distinct leaves $(x_n,y_n)\in \text{\rm Pre}_\theta$ converges to $(x,y)$, one can find a subsequence converging from one side
of the limit leaf $(x,y)$. It follows that $(x,y)$ is  an angle-pair, and is an adjacent angle-pair. Conversely,  one just needs to consider  an adjacent angle-pair
at a point $z$ of the Hubbard tree, and then use expansions on the tree to see that $z$ is approximated by preimages of $0$
on the tree. Details are presented in  Appendix B.

The case that $\theta$ is periodic is considerablely more complicated. Ends of a leaf in $\text{\rm Pre}_\theta$ is
 not an angle-pair, but
corresponds to a pair of rays supporting a common Fatou component with dyadic internal angles.  They would form a cutting line if we add the two internal
rays. See details  in Appendix B.
\end{proof}

\section{Combinatorial Hubbard tree for a rational angle $\theta$}

We have now seen the combinatorial encodings of Julia sets in $\mathbb{S}$ and $\mathbb{T}$. Let us turn now to
the corresponding encodings of Hubbard trees. There are actually two characterizations of a Hubbard tree: one
by looking at  ray-pairs separating the postcritical set, the other by looking at forward images of ray-pair landing points.

\subsection{Combinatorial Hubbard tree}
Note that the Julia points of the Hubbard tree $\mathcal{H}_f$ can be considered as the set of
Julia points  that `separate' the postcritical set $\mathcal{P}_f$. Such Julia points receive necessarily
at least two external rays.
One can therefore define a combinatorial counterpart of  the postcritical set and the Hubbard tree for quadratic maps as follows.

Let $\theta\in\mathbb{S}$ be a rational angle. Define the post-major angle set
 to be:
\[P^{\mathbb{S}}_\theta=\left\{\tau^n\theta\mid n\geq -1\right\}  \quad \text{and}\quad P_\theta=\left\{(\alpha,\alpha)\in \mathbb{T}\mid\ \alpha\in P^{\mathbb{S}}_\theta\right\}.\]
  As $\theta$ is  rational, the set $P_\theta$ is a finite $F$-forward-invariant set in the diagonal of the torus
(also $P^{\mathbb{S}}_\theta$ is a finite $\tau$-forward-invariant set). \footnote{We should have set $P^{\mathbb{T}}_\theta$ in place of $P_\theta$. Since we work more often on the torus model, we  neglect the superscript $\mathbb{T}$ for convenience. }

For $(x,y)\in\mathbb{T}  $, we say that $(x,y)$ \emph{separates} $P_\theta$ if  $x\ne y$ and both components of  $\De\setminus  \{(x,x),(y,y)\}$ contain points of $P_\theta$,
or equivalently, both components of $\mathbb{S}\setminus\{x,y\}$
contain points of $P^{\mathbb{S}}_\theta$. Visually, this is
equivalent to assert that both  open tan squares in $\mathbb{T}$  in
Figure~\ref{fig:ExcludedLeaves}
generated by $(x,y)$ contain points of $P_\theta$ in the interior.
We say that $(x,y)$ \emph{intersects} $P_\theta$ if either $(x,x)$ or $(y,y)$ is in $P_\theta$.

Recall that
\[ L_\theta=\text{\rm cluster}(\text{\rm Pre}_\theta)\setminus\De=\text{the cluster set in $ \mathbb{T} \setminus\De$ of the pre-major leaves.}\] Set
 \begin{equation*}
   X_\theta=\Big\{\,(x,y)\in L_\theta\mid  (x,y)\text{ intersects or separates }P_\theta\Big\}.
   \end{equation*}
 \begin{definition}[Combinatorial Hubbard tree]
The set $H_\theta:=X_\theta\cup P_\theta$ is called the {combinatorial Hubbard tree} of $\theta$.
\end{definition}

It is known that the Hubbard tree of a postcritically finite polynomial is an  attracting core of the landing points of ray-pairs, in the sense that some forward iterations of any such a point will be in the Hubbard tree and stay there under further iterations. Here
is the combinatorial counterpart.

\begin{proposition}\label{attractor}
Let $\theta\in \mathbb{S}$ be a rational angle, The set $X_\theta$  is an attracting core of $L_\theta$, in the sense that  $X_\theta\subset L_\theta$ and  for any $(x,y)\in L_\theta$, there is $n\ge 0$ so that $F^n(x,y)\in X_\theta$.
  \end{proposition}
\begin{proof} The inclusion $X_\theta\subset L_\theta$ is obvious.

For a pair of points $x,y\in \mathbb{S}$, let us define the distance of $x,y$,  denoted by $|x-y|$, to be the  arc-length of the shortest of the two arcs in $\mathbb{S}$ determined by $x,y$. Let $(x,y)\in L_\theta$. Note that if $|x-y|\leq \frac{1}{4}$, then 
$$|\tau(x)-\tau(y)|=2|x-y|.$$
 So there exists a minimal $n\geq0$ such that $|\tau^n(x)-\tau^n(y)|> \frac{1}{4}$. In this case, $$\dfrac12\ge |\tau^n(x)-\tau^n(y)|> \frac{1}{4}.$$
It follows that
the shorter closed arc between $\tau^{n+1}(x)$, $\tau^{n+1}(y)$ must contain the point $\theta$. Since the leaf $\overline{\tau^{n+1}(x)\tau^{n+1}(y)}$ belongs to the closure of a component of 
$$\overline \D\setminus\overline{\frac{\theta}{2}\frac{\theta+1}{2}},$$ it follows that the leaf $\overline{\tau^{n+1}(x)\tau^{n+1}(y)}$ separates or intersects $$P^{\D}_\theta:=\Big\{e^{2\pi i t}\mid t\in P^{\mathbb{S}}_\theta\Big\}.$$
\end{proof}

\subsection{Relation between combinatorial trees and pre-major leaves}

Fix a rational angle $\theta\in \mathbb{S}$. It is relatively easy to find in the countable set $\text{\rm Pre}_\theta$ of points that separate or intersect
the post-major angle set. We will see in Proposition \ref{compact-invariant} that, starting from these points, we can recover the combinatorial Hubbard tree. Set
\begin{align*}
S_\theta&:=\Big\{ (x,y)\in \text{\rm Pre}_\theta\mid (x,y)\text{ separates or intersects  }P_\theta\Big\} \\
&=\Big\{\text{pre-major leaves separating or intersecting post-major angles}\Big\} \\
\text{\rm cluster}(S_\theta)&:=\text{the set of cluster points of $S_\theta$ in $\mathbb{T}$.}
\end{align*}

\begin{lemma}\label{limit-separate}
Any $(x,y)\in \text{\rm cluster}(S_\theta)$  intersects or separates $P_\theta$.
\end{lemma}
\begin{proof} This is an easy consequence of the fact that $P_\theta$ is finite.
\end{proof}

\begin{proposition}\label{compact-invariant}
The combinatorial Hubbard tree $H_\theta$ and the set  $\text{\rm cluster}(S_\theta)$
are both compact and $F$-forward-invariant.
Moreover the two sets differ by a finite set.
\end{proposition}
\begin{proof}
The compactness of  $\text{\rm cluster}(S_\theta)$ is obvious. We proceed to prove that it is $F$-forward-invariant.

We show first that the union of the post-major set $P_\theta$ with the set $S_\theta$  is $F$-forward-invariant.  As $\text{\rm Pre}_\theta$ is the set of pre-major leaves, if $(x,y)\in \text{\rm Pre}_\theta$, either $F(x,y)\in \text{\rm Pre}_\theta$ or $F(x,y)\in P_\theta$. So we are left to prove that if  $(x,y)$ separates or intersects $P_\theta$,
so does  $(\tau(x),\tau(y))$.

Denote, as before, the set $$P_\theta^{\D} =\{e^{2\pi it}\mid t\in P_\theta^{\mathbb{S}}\},$$ which is the post-major angle set in the unit
disc model.
Note that
$  (x,y)$ separates or intersects $P_\theta$ if and only if $\overline{xy}$ separates or intersects $P^{\D}_\theta$  in $\overline{\D}.$
If $\overline{xy}$ intersects $P^{\D}_\theta$, the leaf $\overline{\tau(x)\tau(y)}$ must intersect $P^{\D}_\theta$. If $\overline{xy}$ separates (in $\overline \D$) but does not  intersect $P^{\D}_\theta$, the leaf $\overline{xy}$ belongs to a component $W$ of 
$$\overline\D\setminus \overline{\frac{\theta}{2}\frac{\theta+1}{2}}.$$ Then,  $\overline{xy}$ must separate a point $e^{2\pi i t}$ of $P^{\D}_\theta$ and the point $e^{2\pi i \frac{\theta}{2}}$. It follows that  $\overline{\tau(x)\tau(y)}$ separates the two points $e^{2\pi i\tau(t)},e^{2\pi i \theta}\in P^{\D}_\theta$. This shows that $S_\theta\cup P_\theta$ is $F$-forward-invariant.

Second, we will show  that $\text{\rm cluster}(S_\theta)$ is $F$-forward invariant. Take $\{(x_n,y_n)\}$ a sequence of distinct points in $S_\theta$ such that $$\ds \lim_{n\to \infty}(x_n,y_n)=(x,y)\in \text{\rm cluster}(S_\theta).$$
According to the $F$-invariant property of $S_\theta\cup P_\theta$ proved above, the sequence of points $\{F(x_n,y_n)\}$ belong to $S_\theta\cup P_\theta$. As $P_\theta$ has only finitely many $F$-preimages, countably many of $\{F(x_n,y_n)\}$ must belong to $S_\theta$. Since 
$$\ds \lim_{n\to \infty}F(x_n,y_n)=F(x,y),$$ we have $F(x,y)\in \text{\rm cluster}(S_\theta)$.

The proof that  $H_\theta$ is compact and $F$-forward-invariant is similar and left as an exercise.

By Lemma \ref{limit-separate} we have $\text{\rm cluster}(S_\theta) \subset  H_\theta$.

Let us show that $H_\theta\setminus \text{\rm cluster}(S_\theta)$ is also finite.
  Note that if a point $(x,y)\in  X_\theta$
  separates $P_\theta$, then $(x,y)\in L_\theta$, which is the cluster set in $\mathbb{T}\setminus \De$ of $\text{Pre}_\theta$. One can thus find a sequence $\{(x_n, y_n)\} $ in $\text{ Pre}_\theta$
  converging to $(x,y)$. So for all large $n$, the point $(x_n, y_n)$ must separate the finite post-major set $P_\theta$. It follows by definition that $(x_n, y_n)\in S_\theta$, and therefore $(x,y)\in
\text{\rm cluster}(S_\theta)$. So a point of $  X_\theta\setminus \text{\rm cluster}(S_\theta)$ must intersect but not separate $P_\theta$.

 On the other hand, since $L_\theta$ is the polynomial lamination of $ f_\theta$ (by Proposition \ref{polynomial-lamination}), each
 angle belongs to at most two adjacent angle-pairs.  In particular, for any point $(p,p)\in P_\theta$, there are at most two points of $L_\theta$  intersecting $(p,p)$.
So $ X_\theta\setminus \text{\rm cluster}(S_\theta)$, and therefore $ H_\theta \setminus \text{\rm cluster}(S_\theta)$, is a finite set.
\end{proof}

\section{Entropies}\label{sec:entropy}

The objective is to prove that several quantities in the combinatorial and dynamical worlds  all encode the core entropy of a quadratic polynomial.
\begin{definition}
For a polynomial $f$ such that its Hubbard tree exists and is a finite tree, the \emph{core entropy} of $f$ is the topological entropy of the restriction of $f$ to its Hubbard tree,  $h(\mathcal{H}_f,f)$.  \end{definition}

We remark that if a polynomial $f$ has a connected and locally connected filled Julia set such that the postcritcal set lies on a finite topological tree, then its  Hubbard tree exists.  This condition is called “topologically finite” in \cite{Ti1}, and is more general than postcritically finite.
\subsection{Various entropies and Hausdorff dimensions}

Let $\theta\in \mathbb{S}$ be a rational angle.
As   $H_\theta$ is  compact and $F$-forward invariant, we can define its  topological entropy   $h(H_\theta,F)$.

Recall that for the polynomial  $ f_\theta:z\mapsto z^2+c_\theta$, we use $\g(\eta)$ to denote the
landing point of the external ray of angle $\eta$, and $\mathcal{H}_\theta$ to denote its Hubbard tree.

The polynomial $ f_\theta$ induces two more combinatorial sets\footnote{ these sets are referred to  in some literature as the set of bi-accessible angles  }
\begin{itemize}
\item the angle-pair set  
$$\text{\rm B}_\theta=\{ (\beta,\eta)\in \mathbb{T} \mid \beta\not=\eta\text{ and }\g(\beta)=\g(\eta)\}$$
\item and its projection to the circle 
$$\text{\rm B}^{\mathbb{S}}_\theta=\{ \beta\in\mathbb{S}\mid \exists\ \eta\not=\beta\text{ s.t }\g(\beta)=\g(\eta)\}.$$
\end{itemize}

The  set $\NE$ in the following table will be defined in the next subsection.

$$\begin{array}{|r|cr|c| r| c|}
\hline
  \text{the actions}& \multicolumn{2}{|c|}{F:(\eta,\zeta)\mapsto (2\eta, 2\zeta)} &\mystrut f_\theta & \tau:\eta\mapsto 2\eta\ \qquad& \text{on a graph}\mystrut \\ \hline
   \begin{array}{r}  \text{ entropies}\\ \text{on\  trees} \end{array} &&h(H_\theta)  &\mystrut h(\mathcal{H}_\theta) &&\mystrut  \\
  \begin{array}{r}  \text{ dimensions}\\ \text{ on trees} \end{array}& & \HD (H_\theta) & & &\mystrut  \\
 \begin{array}{r}  \text{dimensions}\\ \text{on bi-acc sets} \end{array}&\begin{array}{l}  \HD (L_\theta)\\ \HD (\NE\setminus\De) \end{array}  &\HD(\text{\rm B}_\theta)&& \HD (\text{\rm B}^{\mathbb{S}}_\theta)&\mystrut \\ \hline
\text{algorithm}&\multicolumn{4}{c}{} & \rho(A_\theta)\mystrut \\
\hline
\end{array}$$
\vspace{0.5cm}

\begin{theorem}\label{Main}
For a rational angle $\theta\in \mathbb{S}$, all the quantities in the above table (as well as several other quantities detailed in the proof)  are related by
$$\log \rho(A_\theta) = {h(\cdot)}={\HD(*)}\cdot \log 2.$$
\end{theorem}
\begin{proof}  We will establish the equalities following the schema:
\[\begin{array}{ll} & \log  \rho(A_\theta)\mystrut \\
&\ \Big\| \mystrut \text{\ by Thm \ref{Thurston-algorithm}}\\
 h(H_\theta,F)= \mystrut  & h(\mathcal{H}_\theta, f_\theta) \qquad (\text{entropies})  \\
 \  \Big\uparrow\mystrut \cdot  \log 2   &\  \Big\uparrow\mystrut \cdot  \log 2\\
 \HD(H_\theta) \mystrut  & {\HD(\text{\rm B}^{\mathbb{S}}_\theta)}  \qquad (\text{dimensions})  \\
\ \mystrut \Big\| \\
 \multicolumn{2}{l}{{\HD(L_\theta)} ={\HD(\NE\setminus\De)}=\mystrut\HD(\text{\rm B}_\theta) }\\
  \multicolumn{2}{c} {(\text{by Prop. \ref{non-escape} below })}
\end{array}\]

For simplicity set $f=f_\theta$ and $\mathcal{H}=\mathcal{H}_\theta$.
 \begin{itemize}
 \item  $h(H_\theta,F)=h(\mathcal{H}, f)$.

By Proposition \ref{compact-invariant},   the set $  H_\theta $ is  compact $F$-invariant.
 By Proposition \ref{polynomial-lamination} the set $L_\theta$ is equal to the polynomial lamination of $ f$. By the definition of $  H_\theta$,  there is a continuous surjective and finite-to-one map $\varphi:   H_\theta\to \mathcal{H}\cap \mathcal{J}$ realizing a semi-conjugacy: $ f\circ \varphi=\varphi\circ F$.
By Proposition \ref{Do4}, we have that 
$$h(  H_\theta,F)=h(\mathcal{H}\cap \mathcal{J}, f).$$ If $\theta$ is preperiodic,
$\mathcal{H}\cap \mathcal{J}=\mathcal{H}$ so 
$$h(  H_\theta,F) =h(\mathcal{H}, f).$$
If $\theta$ is periodic,   the points in $\mathcal{H}\setminus \mathcal{J}$  are attracted by the critical orbit $\mathcal{P}_f$, on which $ f$ has $0$ entropy.
So \begin{align*} h(\mathcal{H}, f) &= h\Big( (\mathcal{H}\cap \mathcal{J})\cup \mathcal{P}_f , f\Big)\text{\  \ by Proposition \ref{Do3}}\\
&= \max\Big( h(\mathcal{H}\cap \mathcal{J}, f), h(\mathcal{P}_f , f)\Big)\text{\  \ by Proposition \ref{Do2}}\\
&=h(\mathcal{H}\cap \mathcal{J}, f) = h(  H_\theta,F)\text{\  \ by Proposition \ref{Do4}}\end{align*}
Note that this entropy is also equal to 
$$ h(\text{\rm cluster}(S_\theta),F),$$  as $\text{\rm cluster}(S_\theta)$ and $H_\theta$  differ by a finite set (Proposition \ref{compact-invariant}).
\item $\HD(H_\theta)\cdot \log 2=h(H_\theta,F)$.

This follows directly from  Proposition \ref{compact-invariant} and Lemma \ref{dimension-formula}.

\item ${h(\mathcal{H}, f)} ={\HD(\text{\rm B}^{\mathbb{S}}_\theta)}\cdot \log 2$.

This fact was first noticed by Tiozzo \cite{Ti1,Ti2}. It is based on the fact that the set $\text{\rm B}^{\mathbb{S}}_\theta$ has an attracting core
which is the set $H^{\mathbb{S}}_\theta $ of angles  whose rays land at the Hubbard tree, so has the same dimension. One can then relate the dimension of $H^{\mathbb{S}}_\theta$ to entropy of $\tau$ on $H^{\mathbb{S}}_\theta$ by Lemma \ref{dimension-formula}, and then relates to the entropy on the Hubbard tree via the ray-landing semi-conjugacy which is finite-to-one (so one can apply Proposition  \ref{Do4}).

\item $ \HD(L_\theta)= \HD(X_\theta)$  since $X_\theta$ is an attracting core of $L_\theta$ (Proposition \ref{attractor}).

\item $\HD(L_\theta)={\HD(\NE\setminus\De)}=\HD(\text{\rm B}_\theta)$, see Proposition \ref{non-escape} below.
 \end{itemize} \end{proof}

\subsection{Non-escaping sets on the torus}\label{sec:no-escaping}

Let us fix  an angle $\theta\in \mathbb{S}$ (not necessarily rational). We define  two invariant subsets mimicking
filled Julia sets  as follows:

\medskip

{\bf\noindent Definition} (non-escaping sets)\textbf{.}
Set inductively   \begin{align*} \Omega _0 &:=\text{ the interior of the good region $G\Big(\dfrac\theta 2,\dfrac{\theta+1}2\Big)$}\\
 \Omega _{i+1} &:=F^{-1}(\Omega _i )\cap \Omega _0 \\
\NE_1(\theta)&:=\bigcap_{n\geq0}\Omega _n\\
\NE_2(\theta)&:=\bigcap_{n\geq0}\overline \Omega _n
\end{align*}
In other words, $\NE_1(\theta)$ consists of the points in $\mathbb{T}$ whose orbits never escape  $\Omega _0$.

\begin{lemma}
We have $\De\subset \overline{\NE_1}(\theta)\subset \NE_2(\theta)$.
\end{lemma}
\begin{proof} We omit  $\theta$ for simplicity.

For each $n\geq0$, the diagonal $\De$ is contained in $ \overline{\Omega _n}\mbox{}$, and  the set  $\De\setminus \Omega _n\mbox{}$ is  finite. It follows that $\De\subset \NE_2\mbox{}$ and $\De\setminus \NE_1\mbox{}$ is a countable set and hence $\De\subset \overline{\NE_1}\mbox{}$. The inclusion $\overline{\NE_1}\mbox{}\subset \NE_2\mbox{}$ is obvious.   \end{proof}


For $\theta\not=0$, the diameter $\overline{\dfrac{\theta}{2}\dfrac{\theta+1}{2}}$ subdivides the unit circle into two half open halves. Denote by $S_0$ the half circle containing the angle $0$ and by $S_1$ the other one. Set
\[\text{\rm Pre}^{\mathbb{S}}_\theta=\bigcup_{n\geq1}\tau^{-n}(\theta).\]
For any angle $\alpha\in\R/\Z\setminus\text{\rm Pre}^{\mathbb{S}}_\theta$, we can assign a sequence $\epsilon_0\epsilon_1\ldots\in \{0,1\}^{\text{\small $\N$}}$ such that $\epsilon_i=\delta$ if and only if $\tau^i(\alpha)\in S_\delta$ ($\delta=0$ or $1$). This sequence is
said to be  the \emph{itinerary of $\alpha$ relative to $\theta$}, denoted by $\iota_\theta(\alpha)$.

Let $\theta$ be a rational angle. There is a parallel description  in the dynamical plane of $ f: z\mapsto z^2+c_\theta$.

 If $\theta$ is preperiodic,   the external rays $\mathcal{R}\left(\frac{\theta}{2}\right)$ and $\mathcal{R}\left(\frac{\theta+1}{2}\right)$ both land at the critical point $0$. In this case, we set
\[\mathcal{R}\left(\frac{\theta}{2},\frac{\theta+1}{2}\right) := \mathcal{R}\left(\frac{\theta}{2}\right)\cup\{0\}\cup \mathcal{R}\left(\frac{\theta+1}{2}\right).\]
This set is called the \emph{major cutting line}.

If $\theta$ is periodic, say of period $p$,  the rays $\mathcal{R}\left(\frac{\theta}{2}\right)$ and $\mathcal{R}\left(\frac{\theta+1}{2}\right)$ support the Fatou component $U$ that contains $0$. The period of $U$ is also $p$. The first return map $f^p: \overline U\to \overline U$ is conjugated to the action of $z^2$ on $\overline \D$, and
thus admits a unique fixed point on the boundary, called the root of $U$.  This root point has internal angle $0$, and its unique $f^{p}$-preimage on $\partial U$
has internal angle $1/2$.
Denote by $r_U(0)$ and $r_U(1/2)$ the corresponding internal rays in $U$.  They  land at $\g\left(\frac{\theta}{2}\right)$, $\g\left(\frac{\theta+1}{2}\right)$ respectively (or vice versa, depending on which of $\frac{\theta}{2}$, and $\frac{\theta+1}{2}$ is periodic). In this case, we define the \emph{major cutting line} to be
\[\mathcal{R}\left(\frac{\theta}{2},\frac{\theta+1}{2}\right):=\mathcal{R}\left(\frac{\theta}{2}\right)\cup \overline {r_U(0)\cup r_U(1/2)} \cup  \mathcal{R}\left(\frac{\theta+1}{2}\right)\]

In  both cases, $\mathcal{R}\left(\frac{\theta}{2},\frac{\theta+1}{2}\right)$ is a simple curve which subdivides the complex plane into two open regions.
Assume $\theta\ne 0$.  Denote by $V_0$ the region containing the ray $\mathcal{R}(0)$ and by $V_1$ the other region. For an angle 
$$\alpha\in\R/\Z\setminus\text{\rm Pre}^{\mathbb{S}}_\theta,$$ it is easy to see that $\iota_\theta(\alpha)=\epsilon_0\epsilon_1\ldots$ if and only if 
$$ f_\theta^{i}(\g_{c_\theta}(\alpha))\subset V_{\epsilon_i},$$ for all $i\geq0$.

Set
\[\text{Pre-cr}=\bigcup_{n\geq1}f^{-n}(\g(\theta))\]
For $z\in \mathcal{J}\setminus \text{Pre-cr}$, the orbit of $z$ does not meet the major cutting line. So we can define its itinerary relative to this cutting line:
\[\iota_\theta(z)=\epsilon_0\epsilon_1\epsilon_2\ldots\qquad\text{if\quad }f^i(z)\in V_{\epsilon_i}\]
It is easy to see that if $z\in \mathcal{J}\setminus \text{Pre-cr}$ has a external angle $\alpha$, then $\iota_\theta(\alpha)=\iota_\theta(z)$.

The following known result will be useful  for us (see Appendix C for a proof):
\begin{lemma}\label{land together}
Let $\theta$ be a rational angle, and let $\alpha\not=\beta\in\R/\Z\setminus\text{\rm Pre}^{\mathbb{S}}_\theta$. Then the rays $\mathcal{R}(\alpha)$ and $\mathcal{R}(\beta)$ land together if and only if $\iota_\theta(\alpha)=\iota_\theta(\beta)$.
\end{lemma}

\begin{corollary}\label{separate}
For $z_1,z_2\in \mathcal{J}\setminus \text{Pre-cr}$, then $z_1\not=z_2$ if and only if $\iota_\theta(z_1)\not=\iota_\theta(z_2)$.
\end{corollary}


\begin{proposition}\label{non-escape}
For $\theta$ a rational angle, the four sets $L_\theta$, $\NE_1(\theta)\setminus\De$,  $\NE_2(\theta)\setminus\De$ and $\text{\rm B}_\theta$ are pairwise different by a set that is at most countable. Consequently, they have the same dimension.
\end{proposition}
\begin{proof} Fix a rational angle $\theta$. Let us introduce
 a new set
\begin{multline*}Z =\{(x,y)\in\mathbb{T}\mid\{x,y\}\textrm{ is an angle-pair (not necessarily adjacent)} \\
 \textrm{ at a point in } \mathcal{J}\setminus \text{Pre-cr} \}.\end{multline*} \\
We will compare this set with the four sets $L_\theta$, $\NE_1 \setminus\De$, $\NE_2 \setminus\De$ and $\text{\rm B}_\theta$. The notation ``a.e'' below on top of the equal sign means that the symmetric difference between the two sets has measure zero.
\begin{enumerate}
\item \textbf{$\NE_1 \setminus\De=Z $}.

By the definition of $\NE_1 $, a point $(x,y)\in \NE_1 \setminus\De\iff x\not= y$ and the orbit of $(x,y)$ stays in $\Omega _0 $ $ \iff x\not=y\text{ and }\iota_\theta(x)=\iota_\theta(y)\overset{\text{Lemma }\ref{land together}}{\iff}\text{the rays }\mathcal{R}(x)$ and $\mathcal{R}(y)$ land at a common point $z\in \mathcal{J} \setminus \text{Pre-cr} \iff (x,y)\in Z $.

\item $\NE_1 \setminus\De\overset{a.e}{=}\NE_2 \setminus\De$.

We already know $\NE_1 \setminus\De\subset \NE_2 \setminus\De$.
Let us fix a point $(x,y)$  in $\NE_2 \setminus (\NE_1 \cup\De)$. By the definition  of $\NE_2 $ and $\NE_1 $, we have that $$(x,y)\in\left(\bigcap_{n\geq0}\overline{\Omega}_n \right)\setminus\left(\bigcap_{k\geq0}\Omega_k \right)$$ so $$ (x,y)\in \overline{\Omega}_n \ \ \forall\, n\text{\ \ and\ \ } (x,y)\notin  \Omega_N \text{\  for some $N\ge 0$}.$$ But $\Omega_k $ is decreasing with respect to $k$ (it is the set of points that remain in $\Omega_0 $ up to $k$ iterates). So
  for  every $n\ge N$, 
  $$(x,y)\in\overline{\Omega}_n \setminus\Omega_n =\partial \Omega_n.$$
    We can then select a point $(x_n,y_n)\in \Omega _n $ such that $\{(x_n,y_n),n\geq1\}$ converges to $(x,y)$ as $n$
goes to infinity. Set
\[z_n=\g(x_n),\quad w_n=\g(y_n),\quad z=\g(x),\quad w=\g(y).\]
We obtain that
\[\lim_{n\rightarrow\infty}z_n=z \quad\text{and}\quad \lim_{n\rightarrow\infty}w_n=w .\]
It is known that $ f_\theta$ is uniformly expanding on a  neighborhood  of $\mathcal{J}$ with respect to the orbiford metric of $f_\theta$ (see \cite{DH,Mc,Mil}). Since $(x_n,y_n)\in \Omega _n $, the first $n$-entries of $\iota_{\theta}(x_n)$ and $\iota_\theta(y_n)$ are identical. By the argument in the proof of Lemma \ref{land together}, the orbifold distance of $z_n$ and $w_n$ is less than $C\lambda^{-n}$ for some constant $C$ and $\lambda>1$. It follows
that $z=w$ and thus $(x,y)$ is an angle-pair.

 On the other hand, by an inductive argument, one can see that for a point $(p,q)\in\partial \Omega_n$,  either $p$ or $q$ belongs to the set $\bigcup_{i=1}^{n+1}\tau^{-i}(\theta)$.
So, in our case, either $x$ or $y$ belongs to $ \text{\rm Pre}^{\mathbb{S}}_\theta$. It follows that $$\g(x)=\g(y)\in \text{Pre-cr}.$$ It is known that the  angle-pairs  at points in Pre-cr   form a countable set, so we obtain that $\NE_1 \setminus\De\overset{a.e}{=}\NE_2 \setminus\De$.

 \item $L_\theta\overset{a.e}{=}Z $.

 By Proposition \ref{polynomial-lamination}, the set $L_\theta$ is the lamination of $ f_\theta$. Then a point $(x,y)\in  L_\theta\setminus Z$ represents an angle-pair at a common point of $\text{Pre-cr}$, so 
 $$\text{\rm cluster}(\text{\rm Pre}_\theta)\setminus(\De\cup \setminus Z )$$ is a countable set. On the other hand, the set $Z \setminus L_\theta$ consists of angle-pairs that are not adjacent and that land
 at non-precritical points. Then the landing points each
  receives at least 4 external rays and the number of rays remain constant under forward iteration. But each point receiving  at least two rays
  will be mapped eventually  into the Hubbard tree and stay there. So a landing point of angle-pairs in $Z \setminus L_\theta$
  is mapped eventually  to a branching point of the Hubbard tree. As the tree is finite, there are only finitely many branching points forming finitely
  many forward invariant orbits. So the set of such landing points is countable, and each point receives finitely many ray-pairs. It follows that
  the set $Z \setminus L_\theta$  is countable.

 \item Finally, $\text{\rm B}_\theta\supset Z $ and $\text{\rm B}_\theta\setminus Z $ is countable as it concerns only angle-pairs whose rays land at the countable set Pre-cr.
\end{enumerate}
\end{proof}



\appendix
\section{Basic results about entropy}

The following can be found in \cite{Do,DV}.
Let $X$ be a compact topological space, $f:X\rightarrow X$  a continuous map.

\REFPROP{Do2}
If $X=X_1\cup X_2$, with $X_1$ and $X_2$ compact, $f(X_1)\subset X_1$ and $f(X_2)\subset X_2$, then 
$$h(X,f)=\sup\bigl(h(X_1,f),h(X_2,f)\bigr).$$ \ENDPROP

\REFPROP{Do3}
Let $Z$ be a closed subset of $X$ such that $f(Z)\subset Z$. Suppose that for any $x\in X$, the distance of $f^n(x)$ to $Z$ tends to $0$ uniformly on any compact set in $X \setminus Z$. Then $h(X,f)=h(Z,f)$. \ENDPROP

\REFPROP{Do4} Assume that $\pi$ is a  surjective semi-conjugacy $$\begin{array}{rcl}Y &\xrightarrow[]{\ q\ }  &Y\\ \pi\Big\downarrow &&\Big\downarrow \pi \vspace{-0.1cm} \\ X &  \xrightarrow[]{\ f\ } & X\vspace{-0.1cm}.\end{array}$$
We have $h(X,f)\leq h(Y,q)$. Furthermore, if  $\ds \sup_{x\in X}\#\pi^{-1}(x)< \infty$, then  $h(X,f)=h(Y,q)$.\ENDPROP

The following can be found in \cite{Fu}.
\begin{lemma}\label{dimension-formula}
Let $A$ be a compact $\tau$-invariant subset in $\mathbb{S}$, resp.\ a compact $F$-invariant subset in $\mathbb{T}$. We have
\[e^{h(A,\tau)}=2^{\HD(A)}\quad\text{resp.}\quad e^{h(A,F)}=2^{\HD(A)}.\]
\end{lemma}

\begin{definition}[finite connected graph] A { finite connected graph} is a connected topological space $G$ consisting of the union of the finite {vertex set} $V_G$ and the {edge set} $E_G$ with the properties that
\begin{itemize}
\item each edge $e\in E_G$ is homeomorphic to a closed interval and connects two points of $V_G$. These two vertices are called the ends of $e$.\vspace{-2pt}
\item the interior of an edge does not contain the vertices  and two edges can intersect only at their ends.
\end{itemize}
\end{definition}
A \emph{finite tree}  is a finite connected graph without cycles.

Let $T$ be a finite tree. A point $p$ in $T$ is called an \emph{end point} if $T\setminus \{p\}$ is connected and is called a \emph{branching\ point} if  $T\setminus \{p\}$ has at least $3$ connected components.
An important property of the finite tree is that for any two points $p,q\in T$, there is a unique  arc in $T$ that connects $p$ and $q$. This arc is denoted by $[p,q]_T$. For a subset $X\subset T$, denote by $[X]_T$ the connected  hull of $X$ in $T$.

Let $X$ be a finite connected graph. We call $f:X \to X$  a \emph{Markov map} if each edge of $X$ admits a finite
subdivision into segments and $f$ maps each segment continuously monotonically onto  some edge of $X$.

Enumerating the edges of  $X$  by $\g_i$ , $i=1,\cdots,k$, we obtain a \emph{transition matrix} $D_f=(a_{ij})_{k\times k}$ of $(X,f)$
 such that  $a_{ij}=\ell$ ($\ell\geq1$) if $f(\g_i)$ covers $\ell$ times $ \g_i$ and $a_{ij}=0$ otherwise. Note that for different enumerations of the edges, the obtained transition matrices are pairwise similar. Denote by $\rho(D_f)$ the leading eigenvalue of $D_{f}$. By Perron-Frobenius Theorem, $\rho(D_f)$ is a non-negative real number and it is also
the growth rate of $\|D_{f}^n\|$ for any matrix norm.

 Since the matrix  $D_f$ has integer coefficients, we have $\rho(D_f)\ge 1$ unless there exists $n$ such that $D_f^n= 0$ (nilpotent). But the latter case cannot happen because every edge is mapped onto at least one edge. It follows that we have always $\rho(D_f)\ge 1$.

The following is classical:
\begin{proposition}\label{entropy-formula}
The topological entropy $h(X,f)$ is equal to $\log\rho(D_f)$.
\end{proposition}

\section{Proof of Proposition \ref{polynomial-lamination}}

For a rational angle $\theta\ne 0$, we want to prove that $L_\theta$, the cluster set in $\mathbb{T}\setminus \De$ of the pre-major lamination is equal to
the polynomial lamination of $f: z\mapsto z^2+c_\theta$.

\begin{proof}
To get a good geometric intuition of an invariant lamination, we choose to use the unit disk model. Let 
$$\wp:\mathbb{T} \to \{\text{closures of hyperbolic geodesics in } {\D}\}$$ map a point $(x,y)$ to   $\overline{xy}$, which is called a \emph{leaf} (or a point in the case $x=y$). Denote, respectively, by $\text{Pre}_\theta^{\D} $ and  $L_\theta^{\D}$ the image of $\text{\rm Pre}_\theta$ and $L_\theta$. For each $i\geq0$, define
 \[L_0 =\wp(b_0 )\qquad\text{and}\qquad L_i =\wp(b_{i} \setminus b_{i-1} ),\quad i\geq1\]
 Then we have 
 $$L_0 =\left\{\overline{\frac{\theta}{2}\frac{\theta+1}{2}}\right\}$$ and $\tau:L_{i+1} \to L_{i} $ is at least $2$ to $1$ (here by abuse of notation we set $\tau (\overline{xy})=\overline{\tau(x)\tau(y)}$.)

The leaf $\overline{\frac{\theta}{2}\dfrac{\theta+1}{2}}$ corresponds in   the dynamical plane to the major cutting line $\mathcal{R}\left(\frac{\theta}{2},\frac{\theta+1}{2}\right)$.

For any $n\geq1$, let $\overline{x_ny_n}$ be a leaf of $L_n $. Set  
$$(x_{n-i},y_{n-i}):=F^{i}(x_n,y_n), \quad i\in[0,n].$$
 Then $\overline{x_iy_i}$ is a leaf of $L_i $. Since each $i$,
the set $\{x_i,y_i\}$  belongs to  $\overline{S_0}$ or $\overline{S_1}$ (recall that $S_0$ denotes the half open circled containing $0$ bounded by $\frac{\theta}{2},\frac{\theta+1}{2}$), and the ray-pair $\mathcal{R}(x_i)\cup\mathcal{R}(y_i)$ belongs to the corresponding  component of $\overline{V_0}$, $\overline{V_1}$. Inductively, it follows that
 \begin{itemize}
 \item If $\theta$ is preperiodic,  each $\{x_{i},y_{i}\}$
 is an angle-pair at a point $z_{i}$ with $ f_\theta(z_i)=z_{i-1}$. In this case, we define a cutting   curve corresponding to $\overline{x_iy_i}$ by
 \[\mathcal{R}(x_i,y_i):=\mathcal{R}(x_i)\cup\{z_i\}\cup \mathcal{R}(y_i)\]
 \item If $\theta$ is periodic,  each  $\{x_{i},y_{i}\}$ is a pair of angles  whose rays support a common Fatou component $U_{i}$, and $f(U_i)=U_{i-1}$. In this case, we define a cutting curve corresponding to $\overline{x_iy_i}$ by
   \[\mathcal{R}(x_i,y_i):=\mathcal{R}(x_i)\cup\overline{ r_{U_i}(\alpha_i)\cup r_{U_i}(\beta_i) } \cup \mathcal{R}(y_i)\]
   where $r_{U_i}(\alpha_i)$ (resp.\ $r_{U_i}(\beta_i)$) is the internal ray in $U_i$ landing at $\g(x_i)$ (resp.\ $\g(y_i)$).
\end{itemize}

 For $i\geq0$, set
 \[\Gamma_i =\left\{\mathcal{R}(x,y)\ \Big|\ \overline{xy}\in L_i \right\}\quad\text{and}\quad \Gamma_\infty =\bigcup_{i\geq0}\Gamma_i \]
 Then there is a natural bijection $l_i:L_i \to\Gamma_i $ which maps a leaf $\overline{xy}\in L_i $ to the curve $\mathcal{R}(x,y)$ and satisfies the commutative diagram:
\begin{equation*} \begin{array}{c c c c}
&L_{i} & \overset{\tau}{\longrightarrow } &L_{i-1}  \\
&&&\\[-4pt]
&l_{i} \Big\downarrow & &\Big \downarrow l_{i-1}  \\
&&&\\[-6pt]
&\Gamma_{i} &\overset{ f_\theta}{\longrightarrow}& \Gamma_{i-1}
\end{array}\end{equation*}

We will talk a little more about the curves in $\Gamma_\infty $ for a periodic angle $\theta$. In this case, there is the unique Fatou component cycle of defgree $2$ which contains the critical point $0$. Then the internal angle for each Fatou component can be uniquely defined. Using the fact that the two internal rays contained in $\mathcal{R}\Big(\frac{\theta}{2},\frac{\theta+1}{2}\Big)$ have angles $0$ and $1/2$, together with the construction of $L_i $ ($i\geq1$), the correspondence between $L_i $ and $\Gamma_i $, and the definition of internal angles of a Fatou component, it is not difficult to check the following results:
\begin{itemize}
\item Every curve $\mathcal{R}(x,y)\in\Gamma_\infty $ passes through a single Fatou component $U$, and the rays $\mathcal{R}(x)$, $\mathcal{R}(y)$ support the component $U$.
\item The two internal rays of $U$ contained in $\mathcal{R}(x,y)$ have angles $\frac{j}{2^n},\frac{j+1}{2^n}$ for some $n\geq1$, $0\leq j\leq 2^{n}-1$.
\item For any Fatou component $U$, any integer $n\geq1$ and $0\leq j\leq 2^n-1$, there exist  a unique curve $\mathcal{R}(x,y)\in\Gamma_\infty $ such that it contains the two internal rays $r_U\left(\frac{j}{2^n}\right)$ and $r_U\left(\frac{j+1}{2^n}\right)$. In fact, there exists a minimal $k$ such that $f_\theta^k$ sends $r_U(\frac{j}{2^n})$ and $r_U(\frac{j+1}{2^n})$ to the internal rays contained in $\mathcal{R}(\frac{\theta}{2},\frac{\theta+1}{2})$. Suppose that $f_\theta^k(r_U(\frac{j}{2^n}))$ is the periodic one. Then $\mathcal{R}(x,y)$ is a lift by $f^k_\theta$ of $\mathcal{R}(\frac{\theta}{2},\frac{\theta+1}{2})$ based at $r_U(\frac{j}{2^n})$.
\end{itemize}

With these preparations, we are ready to prove the Proposition. First, we will show that each  leaf in $L_\theta^{\D} $ represents an adjacent angle-pair.

Let $\overline{xy}$ be a leaf in $L_\theta^{\D} $. Then there exists a sequence of leaves $\overline{x_iy_i}\in L_i $ such that $\{\overline{x_iy_i},i\geq1\}$ converges to $\overline{xy}$ in the Hausdorff topology. Without loss of generality, we may assume that
 \[\lim_{n\rightarrow\infty}x_n=x,\qquad\lim_{n\rightarrow\infty}y_n=y\]
  To see the dynamical explanation of the point $(x,y)$, we need to discuss the set of accumulation points $R$ of the corresponding curves $\mathcal{R}(x_i,y_i)$.
At first, we only consider the case that $\overline{x_iy_i}$ converges to $\overline{xy}$ from one side.
We distinguish $3$ basic cases.
\begin{enumerate}
\item No curves in  $\{\mathcal{R}(x_i,y_i) \mid i\geq1\}$ pass through Fatou components. In this case, the external rays $\mathcal{R}(x_i)$ converge to $\mathcal{R}(x)$ and the rays $\mathcal{R}(y_i)$ converge to $\mathcal{R}(y)$. Consequently, the common landing point $z_i$ of $\mathcal{R}(x_i)$ and  $\mathcal{R}(y_i)$ converge to a point $z\in \mathcal{J}_f$, which must  be the common landing point of $\mathcal{R}(x)$ and $\mathcal{R}(y)$. Then we have
    \[R:=\lim_{i\rightarrow\infty}\mathcal{R}(x_i,y_i)=\mathcal{R}(x)\cup\{z\}\cup \mathcal{R}(y)\]
    We claim that $\{x,y\}$ is an adjacent angle pair.   Otherwise, each component of $\C\setminus\mathcal{R}(x,y)$ contains at least two components of $\mathcal{K}_f\setminus z$. Since $\ds\lim_{i\rightarrow\infty}(x_i,y_i)=(x,y)$, for sufficiently large $i$, the points $z_i$ belong to a common component of $\mathcal{K}_f\setminus z$. Set
 \[\{z_i\}_{i\geq1}\subset K\subset U,\]
 where $K$ is a component of $\mathcal{K}_f\setminus z$ and $U$ is a component of $\C\setminus\mathcal{R}(x,y)$. Let $\{v,w\}$ be the  angle-pair at $z$ bounding $K$. On one hand, we know 
 $$[v,w]\subsetneqq[x,y],$$ since, by assumption, the pair $\{x,y\}$ is not adjacent.
 On the other hand,   the angles $x_i,y_i\ (i\geq1)$ belong to $ [v,w]$ and hence the interval $[x,y]$ is a subset of $ [v,w]$. This leads to a contradiction.

\begin{example}
Set $\theta=1/4$. The red curves in Figure \ref{case1} are cutting curves $\{\mathcal{R}(x_i,y_i)\}$ satisfying case (1) such that $x_i\to 2/7,y_i\to 4/7$ as $n\to\infty$. It is seen that  $\{\mathcal{R}(2/7),\mathcal{R}(4/7)\}$ is an adjacent ray-pair landing at the $\alpha$-fixed point of $f_{1/4}$ and $R=\mathcal{R}(2/7)\cup\{\alpha\}\cup\mathcal{R}(4/7)$.
\end{example}
\begin{figure}[htpb]
\centering
\includegraphics[width=4.3in]{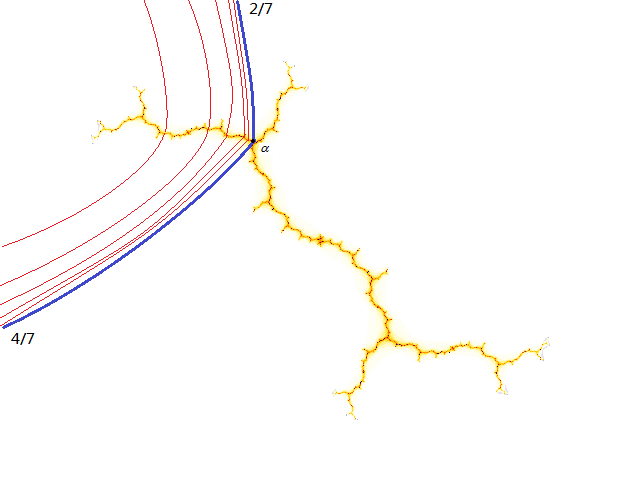}
\caption{} \label{case1}
\end{figure}\vspace{-8pt}
 \item For each $i\geq1$, the curve $\mathcal{R}(x_i,y_i)$ passes through a Fatou component $U_i$  and the diameters of $U_i$  ($i\geq1$) converge to $0$. In this case, the external rays $\mathcal{R}(x_i)$ converge to $\mathcal{R}(x)$, the  rays $\mathcal{R}(y_i)$ converge to $\mathcal{R}(y)$ and the Fatou components $U_i$ converge to a point $z\in \mathcal{J}_f$. This point  must  be the common landing point of $\mathcal{R}(x)$ and $\mathcal{R}(y)$. Then we have
    \[R=\lim_{i\rightarrow\infty}\mathcal{R}(x_i,y_i)=\mathcal{R}(x)\cup\{z\}\cup \mathcal{R}(y)\]

 Using the same argument as  in $(1)$ above, we see that the angle-pair $\{x,y\}$ is an adjacent angle-pair.
 \item For each $i\geq1$, the curve $\mathcal{R}(x_i,y_i)$ passes through a Fatou component $U_i$  and the infimum of the diameters of $U_i$ is bounded from $0$. In this case, for sufficiently large $i$, and passing to a subsequence if necessary, all $U_i$ coincide, and hence may be  denoted by $U$.  It follows that  the two internal rays of $U$ contained in the curves $\mathcal{R}(x_i,y_i)$ tend to a single internal ray as $i\rightarrow\infty$. So  the two limit external rays $\mathcal{R}(x)$ and $\mathcal{R}(y)$ must land together. In fact,
     \[R=\lim_{i\rightarrow\infty}\mathcal{R}(x_i,y_i)=\mathcal{R}(x)\cup \mathcal{R}(y)\cup r_U(\alpha)\cup\{z\}\]
     where $z=\g(x)=\g(y)$ and $r_{U}(\alpha)$ is the internal ray in $U$ landing at $z$.
     Since each pair of rays $\mathcal{R}(x_i)$, $\mathcal{R}(y_i)$ support the component $U$, the angle-pair $\{x,y\}$ must be an adjacent angle-pair.

\begin{example} Set $\theta=4/15$. The blue curves in Figure
  \ref{case23} are cutting curves $\{\mathcal{R}(x_i,y_i)\}$
  satisfying case (2) such that $x_i\to 2/7,y_i\to 4/7$ as
  $n\to\infty$. It is seen that
  $\{\mathcal{R}(2/7),\mathcal{R}(4/7)\}$ is an adjacent ray-pair
  landing at the $\alpha$-fixed point of $f_{4/15}$ and
  $R=\mathcal{R}(2/7)\cup\{\alpha\}\cup\mathcal{R}(4/7)$.  The red
  curves in Figure \ref{case23} are cutting curves
  $\{\mathcal{R}(2/15,v_i)\}$ satisfying case  (3) such that
  $v_i\to 3/5$ as $n\to\infty$. It is seen that
  $\{\mathcal{R}(2/15),\mathcal{R}(3/5)\}$ is an adjacent ray-pair
  landing at $z$ and
  $R=\mathcal{R}(2/15)\cup\overline{r_U(0)}\cup\mathcal{R}(3/5)$.
\end{example}
     \end{enumerate}
\begin{figure}[htpb]
\centering
\includegraphics[width=4.7in]{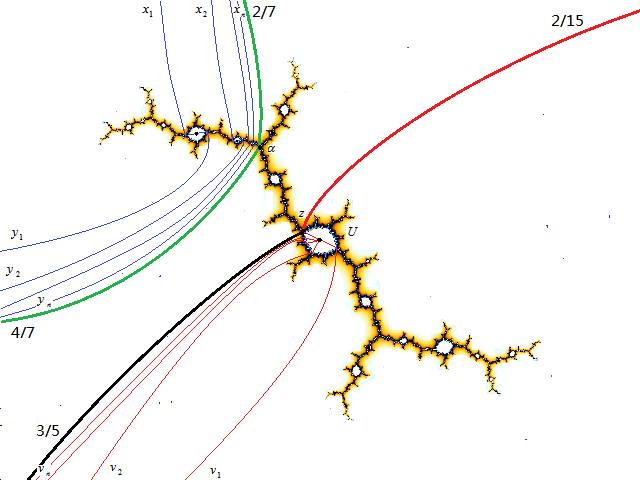}
\caption{} \label{case23}
\end{figure}

Generally, we allow the leaves $\overline{x_iy_i}$ to converge to $\overline{xy}$ from two sides. Then the sequence $\{(x_i,y_i)\}_{i\geq1}$ may contain  subsequences satisfying $1$, $2$ or $3$ of the basic cases described above from either sides of $\overline{xy}$. So the set of accumulation points $R$ is formed by  the possible combinations of the set of accumulation points in the  basic cases from two sides. That is,  $R$  consists of an adjacent ray-pair $\{\mathcal{R}(x),\mathcal{R}(y)\}$ together with eventual $0$, $1$ or $2$ internal rays. Anyway, a leaf $\overline{xy}\in L_\theta^{\D} $ corresponds to the unique adjacent angle-pair $\{x,y\}$.

Next, we will show the opposite implication. That is, any adjacent angle-pair $\{x,y\}$ is contained in $L_\theta^{\D} $ as a leaf $\overline{xy}$.

Let $\{x,y\}$ be an adjacent angle-pair at $z$. Let $V$ be one of  regions bounded by $\mathcal{R}(x)$ and $\mathcal{R}(y)$ without other rays landing at $z$, set $K=V\cap \mathcal{K}$.
Choose any point $w\in K$, denote by $[w,z]_K$ the regulated arc connecting $z$ and $w$ in $K$. We distinguish two cases:
 \begin{enumerate}
\item  $[w,z]_K\cap \mathcal{J} $ clusters at $z$. Then two further cases may happen.
\begin{itemize}
\item $[w,z]_K$ passes through an ordered infinite sequence  of Fatou components $\{U_i,i\geq1\}$ from $w$ to $z$ so that $U_i$ converges to $z$ as $i\rightarrow\infty$. We claim that for any $i\geq2$, we can pick up a curve $\mathcal{R}(x_i,y_i)\in\Gamma_\infty $ such that it separates $U_{i-1}$ and $z$ (in Figure \ref{case23}, we have $\{x,y\}=\{2/7,4/7\}$ and the blue curves are what we want). By this Claim, the sequence of curves $\mathcal{R}(x_i,y_i)$ ($i\geq2$) converge to the ray-pair of $\{x,y\}$ in the Hausdorff topology, so $\overline{xy}\in L_\theta^{\D} $.

\smallskip
{\em Proof of this Claim:}
Fix $U_i$. The intersection of $U_i\cap [w,z]_K$ consists of two internal rays $r_{U_i}(\alpha)$ and $r_{U_i}(\beta)$ of $U_i$. We assume that $r_{U_i}(\alpha)$ lies in between  $z$ and $r_{U_i}(\beta)$. If $\alpha$ is not equal to $\frac{j}{2^n}$ for any $n\geq1,j\geq0$, we can choose integers $n_0$ and $j_0$  so that 
$$\alpha\in\left(\frac{j_0}{2^{n_0}},\frac{j_0+1}{2^{n_0}}\right) \quad \textrm{and} \quad \beta\not\in\left[\frac{j_0}{2^{n_0}},\frac{j_0+1}{2^{n_0}}\right].$$ Then the curve
$\mathcal{R}(x_i,y_i)\in\Gamma_\infty $ that contains $r_{U_i}\left(\frac{j_0}{2^{n_0}}\right)$ and $r_{U_i}\left(\frac{j_0+1}{2^{n_0}}\right)$ separates $U_{i-1}$ and $z$.
If $\alpha=\frac{j_0}{2^{n_0}}$ for some integers $n_0,j_0$, we can choose a sufficiently large $n$ and a suitable integer $j$ such that 
$$\alpha=j/2^n \quad \textrm{and} \quad \beta\not\in\left[\frac{j-1}{2^{n}},\frac{j+1}{2^{n}}\right].$$  
In this case, either the curve (in $\Gamma_\infty $) containing $r_{U_i}\left(\frac{j-1}{2^{n}}\right)$, $r_{U_i}\left(\frac{j}{2^{n}}\right)$ or the curve (in $\Gamma_\infty $) containing $r_{U_i}\left(\frac{j-1}{2^{n}}\right)$, $r_{U_i}\left(\frac{j}{2^{n}}\right)$ separates $U_{i-1}$ and $z$.
Denote this curve by $\mathcal{R}(x_i,y_i)$.
\smallskip

\item No subsequence of the Fatou components (if any) passed through by $[w,z]_K$  converges to $z$.  Then, replacing $w$ by a point closer to $z$ if necessary, we may assume $[w,z]_K\subset \mathcal{J}$. In this case, we can pick up a sequence of points 
$$\{z_i,\ i\geq1\}\subset[w,z]_K\cap(\mathcal{J}\setminus \text{Pre-cr})$$ from $w$ to $z$ such that $z_i$ ($i\geq1$) converges to $z$ as $n\rightarrow\infty$.
    We claim that for any $i\geq1$, we can pick up a  curve $\mathcal{R}(x_i,y_i)\in\Gamma_\infty $ so that it  separates $z_i$ and $z$. Thus, we also obtain that $\overline{xy}\in L_\theta^{\D} $ (in Figure \ref{case1}, we have $\{x,y\}=\{2/7,4/7\}$ and the red curves are what we want).

\smallskip

{\em Proof of this Claim:}
  We only need to prove that each $\mathcal{R}(x_i,y_i)$ separates $z_i$ and $z_{i+1}$. Fix  $i\geq1$. By Corollary \ref{separate}, $\iota_\theta(z_i)\not=\iota_\theta(z_{i+1})$. Denote by  $p_i$ the index of the first distinct entries of $\iota_\theta(z_i)$ and $\iota_\theta(z_{i+1})$. Then the points $f^{p_i}(z_i)$ and $f^{p_i}(z_{i+1})$ are separated by the  curve $\mathcal{R}\left(\frac{\theta}{2},\frac{\theta+1}{2}\right)$. Lifting this curve along the orbit of the pair $\{z_i,z_{i+1}\}$, we obtain a curve  $\mathcal{R}(x_i,y_i)\in\Gamma_\infty $ that  separates $z_i$ and $z_{i+1}$.
  \smallskip
\end{itemize}

\item The set $ [w,z]_K\cap \mathcal{J}$ does not cluster $z$. In this case, the point $z$ is on the boundary of a Fatou component $U\subset V$ and the angle-pair $\{x,y\}$ bounds $U$. Denote by $r_{U}(\alpha)$ the internal ray of $U$ landing at $z$. If $\alpha$ is not equal to $\frac{j}{2^n}$ for any $n\geq1,j\geq0$, we can choose a sequence of integers $\{j_n,n\geq1\}$ such that $\alpha\in\left(\frac{j_n}{2^{n}},\frac{j_n+1}{2^{n}}\right)$ for any $n\geq1$. Let
$\mathcal{R}(x_n,y_n)$ be the curve in $\Gamma_\infty $ that contains $r_{U}\left(\frac{j_n}{2^{n}}\right)$ and $r_{U}\left(\frac{j_n+1}{2^{n}}\right)$. Then this sequence of curves  converge to the ray-pair of $\{x,y\}$ as $n\rightarrow\infty$. If $\alpha=j_0/2^{n_0}$ for some integers $n_0,j_0$,  then  $\alpha$ can be expressed as $j_n/2^n$ for every $n\geq n_0$. Note that one of $x$ and $y$, say $x$,  belongs to the set $\text{\rm Pre}^{\mathbb{S}}_\theta$. Then we obtain two sequences of curves,
 $$\left\{\mathcal{R}(x,y_n)\right\}_{n\geq n_0} \quad \textrm{and} \quad \left\{\mathcal{R}(x,s_n)\right\}_{n\geq n_0},$$
  belonging to $\Gamma_\infty $ such that each ray $\mathcal{R}(x,y_n)$ contains the internal rays $r_{U}\left(\frac{j_n-1}{2^n}\right),r_U\left(\frac{j_n}{2^n}\right)$ and each ray $\mathcal{R}(x,s_n)$ contains the internal rays $r_{U}\left(\frac{j_n+1}{2^n}\right),r_U\left(\frac{j_n}{2^n}\right)$.  By the construction of the curves in $\Gamma_\infty $, we can see that one sequence of curves converge to a single ray $\mathcal{R}(x)$ and the other sequence converges to the ray-pair of $\{x,y\}$. So $\overline{xy}\in L_\theta^{\D} $ (in Figure \ref{case23}, we have $x=2/15,y=3/5$ and the red curves are what we want).
\end{enumerate}\end{proof}

\section{Proof of Lemma \ref{land together}}

For  a rational angle $\theta\in \mathbb{S}\setminus\{0\}$, and $f:z\mapsto z^2+c_\theta$,  this lemma claims that two non pre-major angles $\alpha$ and $\beta$
form an angle pair (i.e. the external rays $\mathcal{R}(\alpha)$ and $\mathcal{R}(\beta)$
land together) if and only if they have the same itinerary relative to the major leaf.

The necessity is obvious because $\alpha$ and $\beta$ have the same itineraries as that of $z$, where $z$ is the common landing point of $\mathcal{R}(\alpha)$ and $\mathcal{R}(\beta)$.

For the sufficiency, we only need to prove the following result: if the first n-th entries of $\iota_\theta(\alpha)$ and $\iota_\theta(\beta)$ are identical (see Subsection \ref{sec:no-escaping} for the definition of $\iota_\theta(\alpha)$),
then the distance of $z,w$ is less than $C\cdot \lambda^{-n}$ with some metric, where $\lambda>1,\ C$ are constants and $z,w$ are the landing points of $\mathcal{R}(\alpha), \mathcal{R}(\beta)$ respectively.

We will prove this by distinguishing between the following two cases: i) the case that $c_\theta$ is (strictly) periodic and ii) the case that $c_{\theta}$ is (strictly) preperiodic.
\medskip

\noindent i) Case $c_\theta$ is periodic:

In this case, $\mathcal{P}_f$ consists of the orbit of $c_\theta$. For a point $p\in \mathcal{P}_f$, set
\[B_p=\Big\{z\in U_p\,\Big|\, |\phi_{U_p}(z)|<e^{-1}\Big\}\]
and set
\[B_\infty=\Big\{z\in U_\infty\,\Big|\, |\phi_{U_\infty}(z)|>e\Big\}\]
Define $\mathcal{L}$ the complement of  the union 
$$B_\infty \cup \left( \bigcup_{p\in \mathcal{P}_f}B_p \right)$$ in $\C$. Then $\mathcal{L}$ is a compact connected neighborhood of the Julia set $\mathcal{J}$
on which $ f$ is uniformly expanding with respect to the hyperbolic metric $\rho$ in a neighborhood $U$ of $\mathcal{L}$, i.e. $\exists \lambda>1$ such that if $\gamma\subset \mathcal{L}$ is an arc and $ f:\gamma\to f(\gamma)$ is a homeomorphism, then
\[\textrm{diam}_\rho( f(\gamma))>\lambda\cdot \textrm{diam}_\rho(\gamma).\]

 Define
\[\mathcal{R}(P _\theta):=\bigcup_{n\geq1}f^n\left(\mathcal{R}\left(\frac{\theta}{2},\frac{\theta+1}{2}\right)\right)\]
For any $p,q\in \mathcal{L}$, we can find an arc in $\mathcal{L}$ connecting $p,q$ such that it doesn't cross the rays in $\mathcal{R}(P _\theta)$. So we can
define a number
\begin{multline*}
M_{p,q}=\inf\Big\{\textrm{length}_\rho(\gamma)\mid \gamma\subset \mathcal{L} \textrm{ is an arc connecting $p$ and $q$ } \\
\textrm{ but not crossing  the rays in $\mathcal{R}(P _\theta)$}\Big\}
\end{multline*}
It is not difficult to check that $M_{p,q}$ is uniformly bounded for $p,q\in \mathcal{L}$, i.e. $\exists\ C>0$ s.t $M_{p,q}<C$ for any $p,q\in \mathcal{L}$.

Now, suppose the first $n$-th entries of $\iota_\theta(\alpha)$ and $\iota_\theta(\beta)$ are identical.
Set $z,w$ to be the landing points of $\mathcal{R}(\alpha),\mathcal{R}(\beta)$, respectively, and $z_i,w_i$ the $i$-th iteration of $z,w$ by $ f$.
Choose an arc $\gamma_n\in \mathcal{L}$ with 
$$\textrm{length}_{\rho}(\gamma_n)<C$$ such that it connects $z_n,w_n$ and doesn't cross the rays in $\mathcal{R}(P _\theta)$.
Lift $\gamma_n$ to the arc $\gamma_{n-1}$ with starting point $z_{n-1}$. 
Since $\gamma_n$ doesn't cross the rays in $\mathcal{R}(P _\theta)$, then so doesn't $\gamma_{n-1}$ and it implies that $\gamma_{n-1}$ is contained in the same component of 
$$\mathcal{L} \setminus\mathcal{R}\left(\dfrac{\theta}{2},\dfrac{\theta+1}{2}\right).$$

Note that the $i$-th entries of $\iota_\theta(\alpha)$ and $\iota_\theta(\beta)$ are equal if and only if $z_i,w_i$ belongs to the same component of $\mathcal{L} \setminus\mathcal{R}\left(\frac{\theta}{2},\frac{\theta+1}{2}\right)$. Then $w_{n-1}$ is the unique preimage of $w_n$ contained in the component of $\mathcal{L} \setminus\mathcal{R}\left(\frac{\theta}{2},\frac{\theta+1}{2}\right)$ that contains $z_{n-1}$. It follows that $\gamma_{n-1}$ connects $w_{n-1}$ and $z_{n-1}$. As 
$$\gamma_{n-1}\subset f^{-1}(\mathcal{L})\subset \mathcal{L},$$ by the uniform expansion of $ f$ in $\mathcal{L}$, we have
\[\textrm{length}_{\rho}(\gamma_{n-1})<\lambda^{-1} \textrm{length}_{\rho}(\gamma_n).\]

Repeating this process $n$ times, we obtain an arc $\gamma_0$ which is the $n$-th lift of $\gamma_n$ with the starting point $z$. With the same argument as before, the ending point of $\gamma_0$ is $w$ and
\[\textrm{dist}_\rho(z,w)\leq \textrm{length}_{\rho}(\gamma_{0})<\lambda^{-n} \textrm{length}_{\rho}(\gamma_n)< C\lambda^{-n}.\]

\bigskip

\noindent ii) Case $c_\theta$ is preperiodic:

 In this case, $\mathcal{J}=\mathcal{K}$ and $ f$ is uniformly expanding with respect to an {\it admissible metric (orbifold metric)} $\rho$ in a compact neighborhood of $\mathcal{J}$ (see \cite{DH,Mc,Mil}).

Since $\mathcal{J}$ is compact, the length of the regulated arc $[z,w] $ (with respect to the metric $\rho$) is uniformly bounded by a constant $C$
for any $z,w\in \mathcal{J}$. Using the same notions as the periodic case, if the first $n$-th entries of $\iota_\theta(\alpha)$ and $\iota_\theta(\beta)$ are equal, then for $0\leq i\leq n-1$, the polynomial $ f$ maps the regulated arc $[z_i,w_i] $ homeomorphically to the regulated arc $[z_{i+1},w_{i+1}] $. By the uniformly expansion of $ f$ on $\mathcal{J}$, we have that
\[\textrm{dist}_\rho(z,w)\leq \textrm{length}_\rho([z,w] )\leq \lambda^{-n} \textrm{length}_\rho([z_n,w_n] )<C\cdot\lambda^{-n}.\]

We note that the result of Lemma \ref{land together} is covered by a more general theorem in \cite{Zeng2}.

\section{Additional Images}

\begin{figure}[http]
  \[
  \includegraphics[width=.95\linewidth]{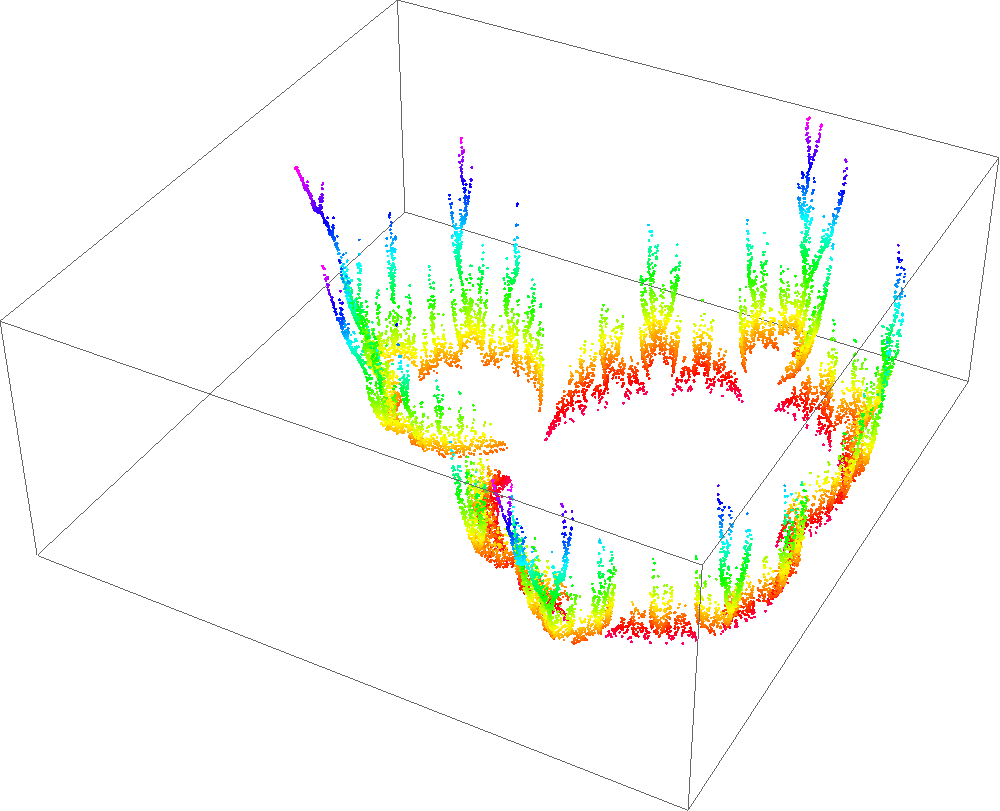}
  \]
  \caption{A plot of core entropy, using data computed by Wolf Jung.  The horizontal plane is $\mathbb{C}$ and the vertical axis is the interval $[0,\log 2] \subset \mathbb{R}$. A plotted point $(c,h) \in \mathbb{C} \times \mathbb{R}$ represents the data that the core entropy of the polynomial $z \mapsto z^2+c$ is $h$.  Data is shown for a selection  of parameters $c$ in the boundary of the Mandelbrot set with rational external angle.  Points are color-coded according to the $h$-coordinate, core entropy.
  }
\end{figure}

\begin{figure}
  \[
  \includegraphics[width=.7\linewidth]{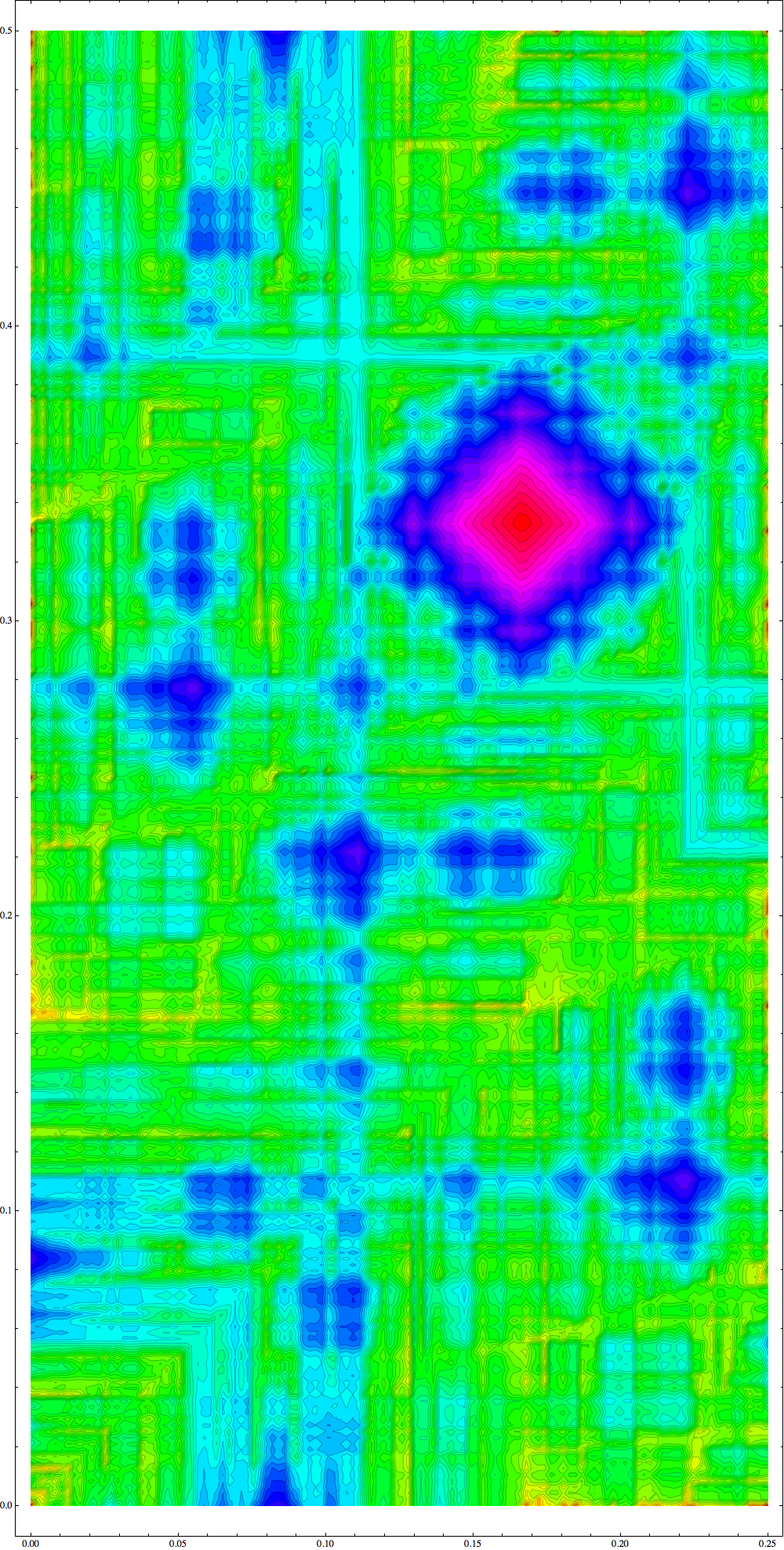}
  \]
  \caption{A contour plot by W. Thurston of core entropy as a function
    on $\textrm{PM}(3)$, using the starting point
    parametrization of Section~\ref{sec:starting-point}.}
\end{figure}

\begin{figure}
  \[
  \includegraphics[width=\linewidth]{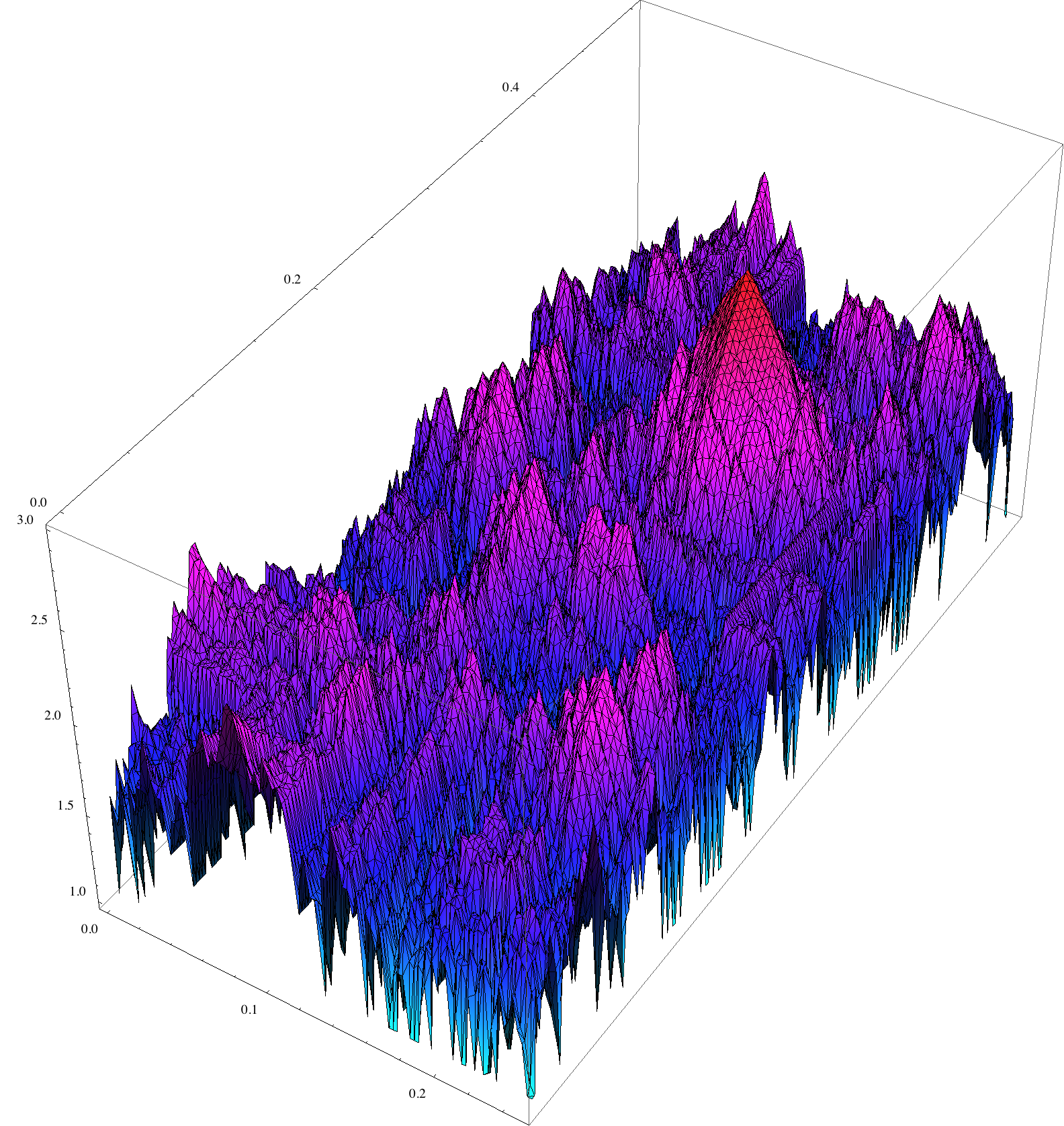}
  \]
  \caption{A plot by W. Thurston of the exponential of core entropy as
    a function on
    $\textrm{PM}(3)$, using the starting point
    parametrization of Section~\ref{sec:starting-point}.}
  \label{fig:entropy-3d}
\end{figure}

\FloatBarrier

\subsection*{Acknowledgements} The authors thank the anonymous referee
for their careful reading and helpful comments.  During the
preparation of this manuscript, Kathryn Lindsey received supported
from the NSF via a Graduate Research Fellowship and, later, a
Postdoctoral Research Fellowship.
Hyungryul Baik was partially supported by Samsung Science \&
Technology Foundation grant No. SSTF-BA1702-01.
Yan Gao was partially supported by NSFC grant No. 11871354.
Dylan Thurston was partially supported by NSF Grant Number DMS-1507244.


\bibliographystyle{amsalpha}
\addtocontents{toc}{\setcounter{tocdepth}{-10}} 
\bibliography{bibliography}

\providecommand{\bysame}{\leavevmode\hbox to3em{\hrulefill}\thinspace}
\providecommand{\MR}{\relax\ifhmode\unskip\space\fi MR }
\providecommand{\MRhref}[2]{%
  \href{http://www.ams.org/mathscinet-getitem?mr=#1}{#2}
}
\providecommand{\href}[2]{#2}
\begin{thebibliography}{BMOV13}

\bibitem[AKM65]{AKM}
R.~L. Adler, A.~G. Konheim, and M.~H. McAndrew, \emph{Topological entropy},
  Trans. Amer. Math. Soc. \textbf{114} (1965), 309--319.

\bibitem[BCO11]{BCO}
Alexander~M. Blokh, Clinton~P. Curry, and Lex~G. Oversteegen, \emph{Locally
  connected models for {J}ulia sets}, Adv. Math. \textbf{226} (2011), no.~2,
  1621--1661. \MR{2737795 (2012d:37106)}

\bibitem[BL02]{BL}
A.~Blokh and G.~Levin, \emph{An inequality for laminations, {J}ulia sets and
  ``growing trees''}, Ergodic Theory Dynam. Systems \textbf{22} (2002), no.~1,
  63--97. \MR{1889565 (2003i:37045)}

\bibitem[BMOV13]{BMOV}
Alexander~M. Blokh, Debra Mimbs, Lex~G. Oversteegen, and Kirsten I.~S.
  Valkenburg, \emph{Laminations in the language of leaves}, Trans. Amer. Math.
  Soc. \textbf{365} (2013), no.~10, 5367--5391. \MR{3074377}

\bibitem[BO04a]{BO1}
Alexander Blokh and Lex Oversteegen, \emph{Backward stability for polynomial
  maps with locally connected {J}ulia sets}, Trans. Amer. Math. Soc.
  \textbf{356} (2004), no.~1, 119--133 (electronic). \MR{2020026 (2005c:37081)}

\bibitem[BO04b]{BO2}
\bysame, \emph{Wandering triangles exist}, C. R. Math. Acad. Sci. Paris
  \textbf{339} (2004), no.~5, 365--370. \MR{2092465 (2005g:37081)}

\bibitem[BOPT14]{BOPT}
Alexander Blokh, Lex Oversteegen, Ross Ptacek, and Vladlen Timorin, \emph{The
  main cubioid}, Nonlinearity \textbf{27} (2014), no.~8, 1879--1897.
  \MR{3246159}

\bibitem[DH]{DH}
A.~Douady and J.~Hubbard, \emph{Exploring the {M}andelbrot set. {T}he {O}rsay
  notes}.

\bibitem[dMvS93]{DV}
Welington de~Melo and Sebastian van Strien, \emph{One-dimensional dynamics},
  Ergebnisse der Mathematik und ihrer Grenzgebiete (3) [Results in Mathematics
  and Related Areas (3)], vol.~25, Springer-Verlag, Berlin, 1993. \MR{1239171}

\bibitem[Dou95]{Do}
A.~Douady, \emph{Topological entropy of unimodal maps: monotonicity for
  quadratic polynomials}, Real and complex dynamical systems ({H}iller\o d,
  1993), NATO Adv. Sci. Inst. Ser. C Math. Phys. Sci., vol. 464, Kluwer Acad.
  Publ., Dordrecht, 1995, pp.~65--87.

\bibitem[Fur67]{Fu}
Harry Furstenberg, \emph{Disjointness in ergodic theory, minimal sets, and a
  problem in {D}iophantine approximation}, Math. Systems Theory \textbf{1}
  (1967), 1--49.

\bibitem[Gao]{Gao2}
Yan Gao, \emph{On {T}hurston's core entropy algorithm}, To appear in
  Transactions of the AMS.

\bibitem[Gao13]{Gao}
Yan Gao, \emph{Dynatomic curve and core entropy for iteration of polynomials},
  Ph.D. thesis, Universit{\'e} d'Angers, France, April 2013.

\bibitem[Gol94]{Goldberg}
L.~R. Goldberg, \emph{On the multiplier of a repelling fixed point}, Invent.
  Math. \textbf{118} (1994), 85--108.

\bibitem[Jun14]{Jung}
Wolf Jung, \emph{Core entropy and biaccessibility of quadratic polynomials},
  Preprint, 2014, arXiv:1401.4792.

\bibitem[Kiw02]{Ki1}
Jan Kiwi, \emph{Wandering orbit portraits}, Trans. Amer. Math. Soc.
  \textbf{354} (2002), no.~4, 1473--1485. \MR{1873015 (2002h:37070)}

\bibitem[Kiw04]{Ki2}
\bysame, \emph{{$\mathbb{R}$}eal laminations and the topological dynamics of
  complex polynomials}, Adv. Math. \textbf{184} (2004), no.~2, 207--267.
  \MR{2054016 (2005b:37094)}

\bibitem[Kiw05]{Kiwi05}
J.~Kiwi, \emph{Combinatorial continuity in complex polynomial dynamics}, Proc.
  London Math. Soc. \textbf{91} (2005), no.~3, 215--248.

\bibitem[Lev98]{Le}
G.~Levin, \emph{On backward stability of holomorphic dynamical systems}, Fund.
  Math. \textbf{158} (1998), no.~2, 97--107. \MR{1656942 (99j:58171)}

\bibitem[McM94]{Mc}
Curtis~T. McMullen, \emph{Complex dynamics and renormalization}, Annals of
  Mathematics Studies, vol. 135, Princeton University Press, Princeton, NJ,
  1994. \MR{1312365}

\bibitem[Mil06]{Mil}
John Milnor, \emph{Dynamics in one complex variable}, third ed., Annals of
  Mathematics Studies, vol. 160, Princeton University Press, Princeton, NJ,
  2006. \MR{2193309}

\bibitem[Mim10]{Mimbs}
Debra~L. Mimbs, \emph{Laminations: {A} topological approach}, ProQuest LLC, Ann
  Arbor, MI, 2010, Thesis (Ph.D.)--The University of Alabama at Birmingham.
  \MR{2801683}

\bibitem[Poi09]{Poi1}
Alfredo Poirier, \emph{Critical portraits for postcritically finite
  polynomials}, Fund. Math. \textbf{203} (2009), no.~2, 107--163.

\bibitem[Poi10]{Poi2}
\bysame, \emph{Hubbard trees}, Fund. Math. \textbf{208} (2010), no.~3,
  193--248.

\bibitem[Pta13]{Pt}
Ross~M. Ptacek, \emph{Laminations and the dynamics of iterated cubic
  polynomials}, ProQuest LLC, Ann Arbor, MI, 2013, Thesis (Ph.D.)--The
  University of Alabama at Birmingham. \MR{3187504}

\bibitem[Thu09]{Th}
William~P. Thurston, \emph{Polynomial dynamics from combinatorics to topology},
  Complex Dynamics: Families and Friends, A. K. Peters, Wellesley, MA, 2009,
  pp.~1--109.

\bibitem[Tio13]{Ti1}
Giulio Tiozzo, \emph{Topological entropy of quadratic polynomials and dimension
  of sections of the {M}andelbrot set}, Preprint, 2013, arXiv:1305.3542.

\bibitem[Tio14]{Ti2}
\bysame, \emph{Continuity of core entropy of quadratic polynomials}, Preprint,
  2014, arXiv:1409.3511.

\bibitem[Tom14]{Toma}
J.~Tomasini, \emph{G{\'e}om{\'e}trie combinatoire des fractions rationnelles},
  Ph.D. thesis, Universit{\'e} d'Angers, 2014.

\bibitem[Zen14]{Zeng}
Jinsong Zeng, \emph{On the existence of shift locus for given critical
  portrait}, Preprint, 2014.

\bibitem[Zen15]{Zeng2}
\bysame, \emph{Criterion for rays landing together}, Preprint online at
  https://arxiv.org/abs/1503.05931, March 2015.

\end{thebibliography}

\nocite{*}

\end{document}